\documentclass[11pt]{article}
\usepackage{amsmath,amsfonts,amsthm,amssymb}
\usepackage{enumerate}

\makeatletter
\def\blfootnote{\xdef\@thefnmark{}\@footnotetext}
\makeatother

\newtheorem{thm}{Theorem}[section]
\newtheorem{cor}[thm]{Corollary}
\newtheorem{lem}[thm]{Lemma}
\newtheorem{pro}[thm]{Proposition}

\theoremstyle{remark}
\newtheorem{rem}[thm]{Remark}

\theoremstyle{definition}
\newtheorem{defn}[thm]{Definition}

\newcommand{\mC}{{\mathbb C}}

\newcommand{\mH}{{\mathbb H}}
\newcommand{\mN}{{\mathbb N}}

\newcommand{\mR}{{\mathbb R}}

\newcommand{\mZ}{{\mathbb Z}}

\newcommand{\cQ}{{\mathcal Q}}
\newcommand{\cP}{{\mathcal P}}
\newcommand{\cS}{{\mathcal S}}
\newcommand{\cM}{{\mathcal M}}

\renewcommand{\Re}{\operatorname{Re}}
\renewcommand{\Im}{\operatorname{Im}}

\DeclareMathOperator*{\esssup}{ess\,sup}
\DeclareMathOperator{\tr}{tr}

\begin{document}
\title{Logarithmic Potentials and Quasiconformal Flows on the Heisenberg Group}
\date{First version submitted January 2017. This revised version submitted January 2020.}
\author{Alex D. Austin\thanks{The author was partially funded by grant NSF DMS-1201875 `Geometric mapping theory in sub-Riemannian and metric spaces'.}}
\newcommand{\addresses}{{
		\bigskip
		\footnotesize
		
		\textsc{Department of Mathematics, University of California, Los Angeles, Box 951555, Los Angeles, CA 90095-1555}\par\nopagebreak
		\textit{E-mail address:} \texttt{aaustin@math.ucla.edu}
	}}
\maketitle
\begin{abstract}
	Let $\mH$ be the sub-Riemannian Heisenberg group. That $\mH$ supports a rich family of quasiconformal mappings was demonstrated by Kor\'{a}nyi and Reimann using the so-called flow method. Here we supply further evidence of the flexible nature of this family, constructing quasiconformal mappings with extreme behavior on small sets. More precisely, we establish criteria to determine when a given logarithmic potential $\Lambda$ on $\mH$ is such that there exists a quasiconformal mapping of $\mH$ with Jacobian comparable to $e^{2\Lambda}$ (so that the Jaobian is zero or infinity at the same points as $e^{2\Lambda}$). When $\Lambda$ is continuous and meets the criteria, we show the canonical (sub-Riemannian) metric $g_0$ and the weighted metric $g = e^\Lambda g_0$ generate bi-Lipschitz equivalent distance functions. These results rest on an extension to the theory of quasiconformal flows on $\mH$ and constructions that adapt the iterative method of Bonk, Heinonen, and Saksman.  
\end{abstract}
\blfootnote{Key words and phrases: quasiconformal Jacobian problem, Heisenberg group, logarithmic potential, bi-Lipschitz equivalence, sub-Riemannian geometry, CR geometry, normal metric, strong-$A_{\infty}$ weight, metric doubling measure.}
\section{Introduction} \label{sec:intro}

The purpose of this paper is to take a first look at the quasiconformal Jacobian problem outside the Euclidean setting. We choose as location the sub-Riemannian Heisenberg group $\mH$, a fundamental object that has played many roles. Most recently $\mH$ has attracted considerable interest as a testing ground for the development of analysis in metric spaces. While amenable to analysis, $\mH$ is highly non-Euclidean in that it does not admit a bi-Lipschitz embedding into any finite-dimensional Euclidean space. The results of our investigations are several: we extend the flow method of generating quasiconformal mappings of $\mH$; we identify a rich supply of non-smooth quasiconformal mappings of $\mH$ with certain prescribed behaviors; and we give an application of the existence of such mappings to a geometric recognition problem. The latter is representative of other interesting questions (regarding metric deformations of the Heisenberg group) made more accessible by this work.

Before specializing our discussion to $\mH$, we introduce some terminology for an arbitrary metric measure space $M = \left(X,d_M,\nu\right)$. This is done solely to illustrate that our motivating problem has a natural formulation in this general setting.

A quasiconformal mapping (of $ M $) is a homeomorphism $ f : M \to M $ such that
\begin{equation} \label{eq:qcdef}
	p \mapsto H_f (p) = \limsup_{r \to 0} \frac{\sup_{d_M (p,q)\leq r} d_M (f(p),f(q))}{\inf_{d_M (p,q)\geq r} d_M (f(p),f(q))}
\end{equation}
is bounded on $M$. If $ H_f $ is not only bounded, but essentially bounded by $ K \geq 1 $ then we say $ f $ is a $ K $-quasiconformal mapping.

For $ f $ a quasiconformal mapping, we define the Jacobian of $ f $ as
\begin{equation} \label{eq:volderiv}
	J_f (p) = \limsup_{r \to 0} \frac{\nu\left(f B_{d_M} (p,r)\right)}{\nu \left(B_{d_M} (p,r)\right)}.
\end{equation}
Here $B_{d_M}(p,r)$ is the open ball (with respect to the metric $d_M$) of radius $r>0$ and center $p \in M$.
  
The quasiconformal Jacobian problem on $M$ asks the following: given $\omega \in L_{\text {loc}}^{1}(M)$ with $\omega \geq 0$, when does there exist $C \geq 1$ and a quasiconformal mapping $f: M \to M$ such that
\begin{equation}\label{eq:qcjp}
	\frac{1}{C} \omega \leq J_f \leq C \omega
\end{equation}
almost everywhere? Such $\omega$ is sometimes called a weight on $M$. The problem was first posed for $M = \mR^n$ by David and Semmes in \cite{David&SemmesStrong}, though it was not until \cite{BHSI} that the name `quasiconformal Jacobian problem' was coined. Even in the special case $M=\mR^n$ the problem is wide open, but attempts to elucidate the situation have generated some wonderful mathematics; in addition to the two aforementioned papers see for example \cite{BishopA1}, \cite{BHSII}, \cite{KovalevMWuQCJP3}, \cite{KovalevMQCJP1}, \cite{KMConvex}, \cite{LaaksoPlaneNoBL}, \cite{SemmesStronger}, and \cite{SemmesNoBL}.

\subsection{Statement of Main Result}
The Heisenberg group is the metric measure space $\mH = (\mR^3 , d_{\mH}, m)$. Here $m$ is Lebesgue $3$-measure. References to `almost everywhere' will be to this measure unless specified otherwise. We delay the details of the metric $d_{\mH}$ until Section \ref{sec:not}.

Any signed measure $\mu$ (on $\mH$) is assumed to be defined on the Borel $\sigma$-algebra (so that the constituent parts of its Jordan decomposition $\mu_+$ and $\mu_-$ are Borel measures). A Borel measure is Radon if it is finite on all compact sets, outer regular on all Borel sets, and inner regular on all open sets. A signed measure is Radon if $\mu_+$ and $\mu_-$ are Radon. A signed measure is finite if the total variation $\|\mu\| = \mu_+ (\mH) + \mu_- (\mH)$ is finite. Since $\mH$ is a locally compact Hausdorff space for which every open set is $\sigma$-compact, every finite signed measure is Radon (see \cite[Theorem~7.8,~p.217]{FollandBook}).

If $\mu$ is a signed measure then $\Lambda_\mu : \mH \to [-\infty,\infty]$ given by
\[
    \Lambda_{\mu} (p) = -\int \log \left( d_{\mH} (p,q) \right) \, \mathrm{d}\mu (q)
\]
is called a logarithmic potential (associated to $\mu$). If $\mu$ is a finite signed measure with
\[
    \int \log^+ \left( d_{\mH} (0,q) \right) \,\mathrm{d}|\mu|(q) < \infty
\]
then we call $\mu$ admissible. Let $\cM$ be the set of admissible (signed) measures. For $\epsilon > 0$ write $\cM (\epsilon)$ for those admissible measures with total variation at most $\epsilon$.

Let $\cQ$ be the set of quasiconformal mappings (of $\mH$), and for $K\geq 1$ write $\cQ (K)$ for the $K$-quasiconformal mappings.

The main result of this paper is the following statement.

\begin{thm} \label{thm:mainint}
Given $K\geq 1$, there exist $ \epsilon > 0 $ and $ C, K' \geq 1 $ such that if $\mu \in \cM (\epsilon)$ and $g\in \cQ (K)$, then there is $f \in \cQ (K')$ with
\[
	\frac{1}{C} e^{2 (\Lambda_\mu \circ\, g)} \leq J_f \leq C e^{2 (\Lambda_\mu \circ\, g)}
\]
almost everywhere.
\end{thm}

If $\Lambda_\mu$ is a logarithmic potential and $g$ a quasiconformal mapping, we will sometimes call $\Lambda_\mu \circ g$ a quasilogarithmic potential.

It is implicit in the writing of the theorem (and worth restating for emphasis) that the constants $\epsilon$, $C$, and $K'$ depend on $K$ only. Specializing to the case of $K=1$ and $g$ the identity, the result says essentially this: there exists $\epsilon_0 > 0$ such that for all $\mu \in \cM (\epsilon_0)$ there is a quasiconformal mapping of $\mH$ with Jacobian comparable to $e^{2\Lambda_{\mu}}$.  

The work leading to this theorem was inspired by the beautiful paper \cite{BHSII} of Bonk, Heinonen, and Saksman, and we follow its overall scheme. It has been pleasant to discover (and somewhat surprising) that the Heisenberg group supports the exact analog of the main result of \cite{BHSII}. A rich family of quasiconformal mappings is far from automatic when moving beyond the Euclidean setting. The Heisenberg group is an example of a Carnot group, and the quasiconformal mappings of a Carnot group are a subset of what are typically called its contact mappings. There are many Carnot groups for which the family of contact mappings is finite-dimensional (suitably understood -- see e.g., \cite{OttaziWarhurst1} for the details).

That quasiconformal mappings on $\mH$ have some flexibility is known. This is due to Kor\'{a}nyi and Reimann who developed the flow method of constructing them in \cite{KRI} and \cite{KRII}. An extension to this method is achieved in Propositions \ref{pro:flow} and \ref{pro:link} that are of independent interest. These results are then invoked in an intricate iteration scheme, an adaptation of the machine of \cite{BHSII}. At the heart of the paper, however, is a construction: we build vector fields with prescribed horizontal divergence.
   
Before moving on, we note that the main result of \cite{BHSII} led to some very interesting results in conformal geometry. We hope that our Theorem \ref{thm:mainint} has similar applications to CR geometry, and encourage the reader to consult Section \ref{subsec:intga} of this introduction for a brief discussion of this fascinating topic.
	
\subsection{Outline} \label{subsec:out}

Some background material (and a guide to the notation we use) is collected in Section \ref{sec:not}. In particular, Section \ref{subsec:hg} contains an introduction to the structural basics of the sub-Riemannian Heisenberg group.

In Section \ref{sec:qc}, we take a brisk look at the required features of quasiconformal mappings. Most is well known. We develop some elementary results that (if known) are harder to find, but nothing that will surprise an expert.

Some technical lemmas regarding (quasi)logarithmic potentials are proved in Section \ref{sec:qlplems}.

Section \ref{sec:qcf} contains our first true innovations. There are two parts. The first extends the flow method of Kor\'{a}nyi and Reimann for generating quasiconformal mappings on the Heisenberg group. In \cite{KRI}, the vector field generating the flow is stipulated to be of $C^2$ regularity and long-time existence of the flow is assumed. In \cite{KRII}, minimal regularity is assumed and existence (uniqueness of solutions to the relevant ODE) of the flow is proved, however, the vector fields are compactly supported (so long-time existence is guaranteed). In Proposition \ref{pro:flow}, we introduce some growth conditions on a vector field of possibly unbounded support that allow us to retain minimal regularity and still prove long-time existence of the flow. These growth conditions correspond to similar conditions imposed in the Euclidean case in \cite{ReimannODE}. They should not be considered restrictive since in dimension $2$ the Euclidean conditions are known to be necessary for quasiconformal flow. Proposition \ref{pro:flow} puts quasiconformal flows on $\mH$ on roughly the same footing as those on $\mR^n$, $n \geq 3$. The second part of Section \ref{sec:qcf} identifies (in Proposition \ref{pro:link}) a suitable means of linking the Jacobian of the flow mappings with the horizontal divergence of the vector field.

The vector field constructions of Section \ref{sec:ap} are made with a twofold purpose; they should satisfy the requirements of the results of Section \ref{sec:qcf}, and the horizontal divergence should approximate a given quasilogarithmic potential in a suitable way. When reading the details of the construction, it is useful to keep in mind the following. For $p = (x,y,t) \in \mH$, let $\|p\| = d_H (p,0)$ (this should be considered notation rather than a definition) and consider the quasilogarithmic potential $-2\log \|p\|$ (the signed measure on $\mH$ has been taken to be twice the Dirac measure centered at the origin and the quasiconformal mapping has been taken to be the identity). Now set
\begin{equation}\label{eq:protophi}
	\phi(p) = -2t\log\|p\|.
\end{equation}
Miner in \cite{Miner} identified that if $\phi$ is used as a vector field potential in an appropriate (non-standard) way (see Section \ref{sec:qcf}), then the time-$s$ flow mapping $f_s$ acts `essentially' as $p \mapsto \|p\|^{e^s - 1} p$. The full mappings $f_s$ later appeared as the radial stretch mappings of Balogh, F\"{a}ssler, and Platis in \cite{BFP}. They were identified as being the correct analog (in terms of their extremal properties) of the Euclidean radial stretch mappings. Radial stretch mappings appear frequently in the Euclidean quasiconformal Jacobian problem. They are simple examples of quasiconformal mappings with explosive volume change at a point (the Jacobian is infinite at the origin). The quasiconformal Jacobian problem is only interesting if the given weight comes arbitrarily close to (or equals) either zero or $\infty$. Otherwise, the Jacobian of the identity is comparable. A vector field suitably generated by the $\phi$ of \eqref{eq:protophi} has horizontal divergence $-2\log\|p\| + \zeta(p)$, with $\zeta$ a bounded function. Consequently, the horizontal divergence nicely approximates the logarithmic potential. It is something like this we require in the more general situation. This all being said, we should be careful not to give the wrong impression. Let $\mu_0$ be the Dirac measure centered at the origin. While \eqref{eq:protophi} is a useful heuristic, it is nevertheless true that $\|2 \mu_0 \| \geq \epsilon$ if $\epsilon = \epsilon(1) > 0$ is as in Theorem \ref{thm:mainint}. We know this not because we give an explicit value for $\epsilon$ (which we do not), but because $e^{2(-2\log\|p\|)}=\|p\|^{-4}$ which is not locally integrable at the origin (and so cannot be comparable to a quasiconformal Jacobian).

In Section \ref{sec:ic} we use the constructions of Section \ref{sec:ap}, along with the results of Section \ref{sec:qcf}, to construct the quasiconformal mapping whose Jacobian achieves comparability with a given quasilogarithmic potential. It is here the technical reason for considering \emph{quasi}logarithmic potentials will become apparent (-- at a step in the iterative procedure, we need to feed in the mapping created thus far). We find our desired mapping in the limit of a sequence $(f_m)$, with each $f_m$ the composition of $m$ (normalized) time-$1/m$ flow mappings. This is an adaptation of the machine of \cite[Section 6]{BHSII}. Listing the changes made to the process of \cite{BHSII} would not serve this outline well, however, the reason for making them is illuminating. The main difficulty is that $\mH$ has a somewhat less flexible family of whole-space \textit{conformal} mappings as compared to $\mR^n$. In $\mR^n$ there are translations, dilations, rotations, and compositions thereof. In $\mH$, there are left-translations, intrinsic dilations, rotations about the group center, and compositions thereof. It is the paucity of the rotations in particular that prevents use of the normalization strategy of \cite{BHSII} (such normalization allows utilization of compactness results for our sequences of quasiconformal mappings).

The arguments we use in Section \ref{sec:wsrm} have become standard in the Euclidean case. To the best of our knowledge we are writing them down for the first time in the case of the Heisenberg group. This claim is supported by our use of the recent paper \cite{NagesCurves} -- curve families controlled in measure as in \cite{SemmesFindingCurves} were suspected (or known by indirect arguments) to exist in $\mH$, however, \cite{NagesCurves} is the first explicit construction we are aware of. We use them at a crucial step in our bi-Lipschitz equivalence result (Theorem \ref{thm:specialgeom} below) using a David-Semmes deformation of $\mH$ as an auxiliary space.

The paper ends with a short appendix containing simple results that can be time-consuming to check, might be useful again, and that were better separated from the already congested Section \ref{sec:ap}. 

\subsection{Geometric Applications}\label{subsec:intga}

An interesting class of sub-Riemannian manifolds is given by the `conformal' equivalence class of the sub-Riemannian Heisenberg group, the set of all $(\mH,e^{u}g_0)$ with $g_0$ the canonical sub-Riemannian metric on $\mH$ (see Section \ref{subsec:hg}) and $u:\mH\to\mR$ a continuous function. Let $\rho$ be the (Carnot-Carath\'{e}odory) distance function associated to $g_0$, and $\rho_u$ that associated to $e^{u}g_0$. We call $(\mH,e^{u}g_0)$ and $(\mH,g_0)$ bi-Lipschitz equivalent if there exists $L\geq 1$ and homeomorphism $f:\mH \to \mH$ such that for all $p,q\in\mH$,
\begin{equation} \label{eq:bldef}  
	\frac{1}{L}\rho(p,q) \leq \rho_u (f(p),f(q)) \leq L \rho(p,q).
\end{equation}
It is important to know when one of the $(\mH,e^{u}g_0)$ is bi-Lipschitz equivalent to $(\mH,g_0)$, for then $( \mH, e^{u}g_0 )$ shares many of the (sometimes well understood) geometric and analytic properties of $( \mH,g_0 )$ itself. One goal of the program initiated here is the sub-Riemannian analog of the next theorem.

\begin{thm}[Bonk, Heinonen, Saksman, Wang] \label{thm:gthm}
Suppose $( \mR^4,e^{2u}g_E )$ is a complete Riemannian manifold with normal metric. If the $Q$-curvature satisfies
\[
	\int |Q| \, \mathrm{d}\textnormal{vol} < \infty
\]
and
\[
	\frac{1}{4 \pi^2} \int Q \, \mathrm{d} \textnormal{vol} < 1
\]
then $(\mR^4,e^{2u}g_E)$ and $(\mR^4,g_E)$ are bi-Lipschitz equivalent. 
\end{thm}

Here $g_E$ is the canonical Euclidean metric and $\mathrm{d}\textnormal{vol}$ indicates (integration with respect to) the volume measure associated to the Riemannian manifold $( \mR^4,e^{2u}g_E )$. In these circumstances the defining equation for the  $Q$-curvature is $\Delta^2 u = 2Q e^{2u}$ (the Paneitz operator associated to $g_E$ reduces to the biharmonic operator $\Delta^2$). A metric $e^{2u}g_E$ on $\mR^4$ is normal if at all $x\in \mR^4$,
\[
	u(x) = -\frac{1}{4\pi^2}\int \log\frac{|x-y|}{|y|} Q(y) e^{4u(y)} \, \mathrm{d}y + C
\]
with $C$ a constant. Here $\mathrm{d}y$ indicates the Lebesgue measure on $\mR^4$. In other words, $u$ is essentially a logarithmic potential (on $\mR^4$) with respect to the measure $\mu$ determined by $\mathrm{d}\mu(y) = Q(y)e^{4u(y)}\,\mathrm{d}y$.
In this paper we take (what should be) a substantial step toward a sub-Riemannian counterpart, proving the following theorem. We write the statement so as to emphasize the similarity with Theorem \ref{thm:gthm}. To aid in this aim, let $\|q^{-1}p\| = d_\mH (p,q)$ (again, for now this should be considered notation). 

\begin{thm} \label{thm:specialgeom}
There exists $ \epsilon > 0 $ such that if $\mu$ is a finite signed measure on $\mH$ with
\[
	\Lambda_\mu (p) = -\int \log\|q^{-1}p\| \, \mathrm{d}\mu (q)
\]
continuous,
\[
	\int \log^+ \|q\| \,\mathrm{d}|\mu|(q) < \infty,
\]
and
\[
	\int \,\mathrm{d}|\mu| < \epsilon,
\]
then $ (\mH,g_0) $ and $ (\mH,e^{\Lambda_{\mu}} g_0) $ are bi-Lipschitz equivalent.
\end{thm}

A stronger version of Theorem \ref{thm:specialgeom} (for quasilogarithmic potentials) holds. It is stated as Theorem \ref{thm:maingeom}.

How this result might be used is open to speculation. In the Heisenberg setting there is currently no notion of normal metric to take aim at\footnote{In the time since this paper was first submitted, important steps in the program hinted at here have been taken by Wang and Yang in \cite{WangYang1} and \cite{WangYang2}}. In a way, we are working backwards; in the Euclidean setting, normal metrics were known (and known to be interesting) prior to Theorem \ref{thm:gthm}. For example, the $Q$-curvature can be thought of as a higher-dimensional version of the Gaussian curvature, and Chang, Qing, and Yang prove in \cite{CQYNormal} something like a Gauss-Bonnet theorem for manifolds satisfying the hypotheses of Theorem \ref{thm:gthm}. In the same paper they show that normal metrics are not unusual: if $(\mR^4,e^{2u}g_E)$ is complete, has integrable $Q$-curvature, and the scalar curvature is non-negative at infinity then the metric is normal. Nevertheless, there is cause for optimism and we take the viewpoint that our results suggest a potentially rich thread in sub-Riemannian / CR geometry. There is other evidence to suggest phenomena similar to the Riemannian case should exist. The correct definition of sub-Riemannian normal metric will likely exploit, then strengthen, what Case and Yang in \cite{CaseYangPPrime} call the `deep analogy between the study of three dimensional CR manifolds, and four dimensional conformal manifolds'. Suitable objects for such an investigation were only recently made available: the Paneitz-type operator and $Q$-like curvature introduced for the CR-sphere and Heisenberg group by Branson, Fontana, and Morpurgo in \cite{BFM}, and abstracted to the more general CR setting in \cite{CaseYangPPrime}. This is a fascinating area, with many strands to pursue, however, we say no more about it here.

Theorem \ref{thm:gthm} as stated is Wang's, it can be found in \cite{WangQ1}. Wang was building on the work of \cite{BHSII}. The primary contribution of Wang to Theorem \ref{thm:gthm} was to give the sharp constants on the size of the measure (which translate into the integral bounds on the $Q$-curvature). Wang showed that the negative part of the measure (the negative variation) need only be finite and that the Dirac measure (on $\mR^4$) identifies the end point for the signed mass. Given these developments and our comments in Section \ref{subsec:out} above, it is tempting to conjecture that Theorem \ref{thm:mainint} is true for all quasilogarithmic potentials with $\mu \in \cM$ such that $\|\mu\|<\infty$ and $\mu(\mH) < 2$.

\subsection{Acknowledgments}
This work was completed while the author was a PhD student. He would like to thank his advisor Jeremy Tyson for all his support. The author also wishes to thank Mario Bonk and Leonid Kovalev for useful conversations. Finally, we thank the first referee for a very careful reading and comments that have significantly improved the exposition.

\section{Background and Notation} \label{sec:not}
\subsection{The Heisenberg Group} \label{subsec:hg}

Beyond the current interest in $\mH$ for the development of analysis in metric spaces, the Heisenberg group plays an important (and more classical) role in harmonic analysis, and is a fundamental example of a sub-Riemannian manifold. When viewed from the right perspective, it is also the prototypical model for CR geometry.

Recall, $m$ denotes the Lebesgue measure on $\mR^3$. For Lebesgue-measurable $ E \subset \mR^3 $, we define $ |E| = m(E) $. As usual we write $\mathrm{d}p$ in place of $\mathrm{d} m (p)$. 

Equip $\mR^3 = \{ (x,y,t) \}$ with group product
\begin{equation} \label{eq:hprod}
( x_1 , y_1 , t_1 ) \star ( x_2, y_2 , t_2 ) = ( x_1 + x_2 , y_1 + y_2 , t_1 + t_2 + 2( x_2 y_1 - x_1 y_2 ) ).
\end{equation}
It is easily checked that the inverse $(x,y,t)^{-1}$ of $(x,y,t)$ with respect to this product is $(-x,-y,-t)$. To alleviate the notational burden, for $ p,q \in \mH $ let $ pq := p \star q$.

For $(x,y,t) \in \mR^3$, let
\begin{equation} \label{eq:hnorm}
\| ( x , y , t ) \| := \left( ( x^2 + y^2 )^2 + t^2 \right)^{\frac{1}{4}}.
\end{equation}
This is sometimes referred to as the Kor\'{a}nyi gauge and satisfies a triangle inequality: if $p,q\in\mR^3$ then $\|q^{-1} p\|\leq \|p\|+\|q\|$ (and also $\|p q\|\leq \|p\|+\|q\|$). The metric (distance function) $ d_{\mH} : \mR^3 \times \mR^3 \to [0,\infty) $ is given by
\begin{equation} \label{eq:hmetric}
    d_{\mH}(p,q) = \| q^{-1} p \|.
\end{equation}
We will describe points in $\mR^3$ and functions defined on $\mR^3$ as being points in $\mH$ and functions defined on $\mH$, respectively, to indicate (in particular) we are interested in the geometry determined by $d_{\mH}$. 

For $r \geq 0$, we define the dilation $\delta_r : \mH \to \mH$ by  
\[
    \delta_r (p) = \delta_r (x,y,t)= (rx,ry,r^2t).
\]
The Kor\'{a}nyi guage $\|\cdot\|$ is referred to as a \emph{homogeneous norm} on $\mH$ because $\|\delta_r (p)\|=r\|p\|$. We will be consistent in our use of $\delta_r$ for the dilations (and for the most part avoid using it for other objects). For $p\in \mH$ and $r>0$, let $B(p,r) \subset \mH$ be the open ball (with respect to the metric $d_{\mH}$) of center $p$ and radius $r$. For $q \in \mH$, let $L_q : \mH \to \mH$ be left-translation by $q$, $L_q (p) = q p$. It is quickly found that $\mH$ is self-similar with respect to the natural operations: $ B(p,r) = L_p(\delta_r(B(0,1))) $. The standard change of variable formula implies there exists absolute constant $C>0$ such that for all $p \in \mH$ and for all $r>0$,
\begin{equation}\label{eq:ar}
	|B(p,r)| = C r^4.
\end{equation}   
We will sometimes refer to \eqref{eq:ar} by saying $\mH$ is Ahlfors 4-regular (though Ahlfors regularity requires only that $|B(p,r)|$ is comparable with $r^4$). The topology of $\mH$ is the same as the topology of Euclidean $\mR^3$, hence $\mH$ has topological dimension $3$. The Lebesgue measure on $\mR^3$ (or $\mH$ if you prefer) is equal (modulo a constant factor) to the Hausdorff measure determined by $d_\mH$. It follows the Hausdorff dimension of $\mH$ is $4$. In that it is both self-similar and has Hausdorff dimension greater than its topological dimension, $\mH$ qualifies as a fractal.

If we give $\mR^3$ its usual smooth manifold structure then $(\mR^3,\star)$ is a Lie group. There is no harm in referring to this as $\mH$ even if (for the moment) we do not have $d_\mH$ etc. in mind. For $p = (x,y,t) \in \mH$, define 
\begin{equation}\label{eq:labasis} 
X_p = \partial_x + 2y \partial_t, \quad Y_p = \partial_y - 2x \partial_t, \quad \text{and} \quad T_p = \partial_t.
\end{equation}
These form a basis for the Lie algebra $\mathfrak h$ of left-invariant vector fields. This basis arises in the following way: $X_p = D L_{p}(0)(\partial_x)$, $Y_p = D L_{p}(0)(\partial_y)$, and $T_p = D L_{p}(0)(\partial_t)$.

If $V$ and $W$ are vector fields on $\mH$ then $[V,W]$ is the usual Lie bracket:
\[
    [V,W] = VW - WV.
\]
Note that $[X,Y]_p = -4 T_p$ so that $X_p$, $Y_p$, and $[X,Y]_p$ span the tangent space at $p$ (in common parlance, the vector fields $X$ and $Y$ satisfy H\"{o}rmander's condition). It follows $\mH$ is a Carnot group. Let $ \textnormal{H}\mH \subset \textnormal{T}\mH $ be defined by $\textnormal{H}_p \mH = \text{span} ( X_p , Y_p )$. We call $\textnormal{H}\mH$ the horizontal layer of the tangent bundle.

We may identify the Lie algebra with the tangent space at the origin: if $V$ is a left-invariant vector field, identify $V$ with $V_0$. The previously given basis corresponds to the basis $X_0 = \partial_x$, $Y_0 = \partial_y$, and $T_0 = \partial_t$. A bracket can be defined on $\textnormal{T}_0 \mH$ by $[V_0,W_0] = [V,W]\big|_0$. It is typically this identification and basis we have in mind when we involve the exponential mapping. It is easy to see the unique one-parameter subgroup $\gamma$ satisfying $\gamma(0)=0$ and $\gamma'(0) = W_0$ is given by $s \mapsto (s w_1 , s w_2 , s w_3)$. Here the $w_i$ are defined by $ W_0 = w_1 X_0 + w_2 Y_0 + w_3 T_0 $. It follows that $ \exp : \mathfrak h \to \mH $ is given by $ \exp(w_1 X_0 + w_2 Y_0 + w_3 T_0 ) = (w_1 , w_2 , w_3) $. When $\mathfrak h$ is identified with $\textnormal{T}_0\mH$, we call $\text{span}( X_0 , Y_0 )$ its horizontal layer.

Define an inner product $g_0(p)$ on each $\textnormal{H}_p \mH $ by
\begin{equation} \label{eq:cansr}
g_0 (X_p,X_p) = 1, \quad g_0 (X_p,Y_p) = 0, \quad\text{and}\quad g_0 (Y_p,Y_p) = 1.
\end{equation}
We refer to $p\mapsto g_0(p)$ as the canonical sub-Riemannian metric on $\mH$. It gives rise to a Carnot-Carath\'{e}odory distance function as follows. Let $ b>0 $. A continuous mapping $ \gamma : [0,b] \to \mH $ is a horizontal curve if $ \gamma \in C^{1}(0,b) $ with $ \gamma'(s) \in \textnormal{H}_{\gamma(s)}\mH $ for all $ s \in (0,b) $. Define
\begin{equation} \label{eq:ccdistance}
\rho(p,q) = \inf_{\gamma} \int_0^b \sqrt{g_0(\gamma'(s),\gamma'(s))} \,\mathrm{d}s,
\end{equation}
where the infimum is taken over all piecewise-horizontal curves joining $p$ and $q$. When $\gamma$ is a horizontal curve, $ g_0(\gamma'(s),\gamma'(s)) = \gamma_1 '(s)^2 + \gamma_2 ' (s)^2 $. If $ V \in \textnormal{H}_p \mH $ for some $p$, we will sometimes write $ |V|_H = \sqrt{g_0(V,V)}$. We intend the distance function $\rho$ be implicit in the notation $(\mH,g_0)$ used for the sub-Riemannian manifold just described.

The majority of our time will be spent working with the metric $ d_{\mH} $ defined in \eqref{eq:hmetric}. Let $d := d_{\mH}$. This is not the Carnot-Carath\'{e}odory distance function associated to a sub-Riemannian metric. Our earlier assertions (such as those of the abstract) are valid since $d$ and $\rho$ are bi-Lipschitz equivalent. The explicit expression for $d$ makes it easier to work with than the Carnot-Carath\'{e}odory distance. That said, $ \rho $ and weighted versions of $g_0$ are the subject of Section \ref{sec:wsrm}. For use in that section, we define the length of a continuous curve $\gamma : [0,b] \to \mH$ (with respect to the metric $d$) as
\begin{equation} \label{eq:ld}
    l_d (\gamma) = \limsup_{m \to \infty} \sum_{i=1}^{m} d(\gamma(s_i),\gamma(s_{i-1}))
\end{equation}
with $s_i = ib/m$. It is shown in \cite[Lemma~2.4, p.~20]{TysonHBook} that if $\gamma \in C^1 (0,b)$ then $l_d (\gamma)$ coincides with $\int_0^b |\gamma'|_H $ if $\gamma$ is horizontal, and is infinite otherwise.

The next fact gives a useful comparison between the Heisenberg and Euclidean metrics on $\mR^3$. Given a compact set $\Omega\subset \mH$ there is $C=C(\Omega)>0$ such that for all $p,q\in\Omega$,
\begin{equation} \label{eq:hemcomp}
\frac{1}{C} |p-q| \leq d(p,q) \leq C |p-q|^{\frac{1}{2}}.
\end{equation} 
Here $|p-q|$ is the Euclidean distance between the points $p,q\in\Omega\subset\mH$ treated as points of $\mR^3$. This says (among other things) the set identity from Euclidean $\mR^3$ to $\mH$ is locally $(1/2)$-H\"{o}lder continuous. A look at the expression for $d$ shows the identity map is not locally $\alpha$-H\"{o}lder continuous for any $\alpha>1/2$ (thus, in particular, not locally Lipschitz).

Lastly, we require the following formula for integration in polar coordinates (for a proof of which see \cite{FollandSteinHardy}),
\begin{equation} \label{eq:polarint}
\int_{\mH} f(p) \, \mathrm{d}p = \int_{S(1)} \int_0^\infty f(\delta_r(q)) r^3 \, \mathrm{d}r \mathrm{d}\sigma(q)
\end{equation} 
with $\sigma$ an appropriate measure on $S(1)$ (the unit sphere with respect to $d$). The formula is valid for all $f\in L^1 (\mH)$.

\subsection{Notation} \label{subsec:not}

If several points are in play, any mention of $x$, $y$, or $t$ always refers to the coordinates of the point labeled $p$.

If $\mu$ is a signed measure with Jordan decomposition $\mu = \mu_{-} - \mu_{+}$, then $\|\mu\|$ is the total variation: $\|\mu\|=\mu_{-}(\mH)+\mu_{+}(\mH)$.

$\chi_E$ is the indicator function of the set $E$.

We write $M_n(\mR)$ for the $n \times n$ matrices with real entries. If $M \in M_n(\mR)$ for some $n$, then $|M|$ is the operator norm ($|M| = \sup_{v\in \mR^n, |v|=1}|Mv|$) and $\det M$, $\tr M$, and $M^T$ are respectively the determinant, trace, and transpose of $M$. $I_n$ is the $n\times n$ identity matrix.

For $p \in \mH$ and $r > 0$, $B(p,r)$ is the open ball of center $p$ and radius $r$. Let $B(r) := B(0,r)$. For $p \in \mH$ and $r>0$, $S(p,r)$ is the sphere of center $p$ and radius $r$. Let $S(r) := S(0,r)$. $B(p,r)$ and $S(p,r)$ are defined with respect to the metric $d$ (as are all other metric statements unless obviously otherwise).

The function spaces $L^r := L^r(\mH)$ for $1\leq r \leq \infty$ (with norm $\|\cdot\|_r$) have their usual definition. Since they are implicitly determined by the Lebesgue measure on $\mR^3$ they are identical to their Euclidean counterparts $L^r(\mR^3) $. Similarly $L_{\text{loc}}^{r} := L_{\text{loc}}^{r}(\mH)$.

The spaces $C^{k} := C^{k}(\mH)$ and $C_{0}^{k} := C_{0}^{k}(\mH)$ for $1\leq k \leq \infty$ are defined with respect to the smooth manifold structure on $\mH$ (and have their usual meaning in that context). Since that structure has the set identity of $\mR^3$ as global chart, they are respectively identical with $C^k (\mR^3)$ and $C_0^k (\mR^3)$. 

$ HC^1$ is the space of continuous functions $ F : \mH \to \mR $ such that the (classical) horizontal derivatives $XF$ and $YF$ exist and are continuous everywhere (we might say they are continuously differentiable in the horizontal directions).

For $1\leq r \leq \infty$, $ HW_{\text{loc}}^{1,r} $ is a first horizontal Sobolev space. This is the set of $F \in L_{\textnormal{loc}}^1 $ with distributional derivatives $ XF , YF \in L_{\text{loc}}^r$. Such distributional derivatives will be referred to as weak derivatives -- see below for more discussion.

If $ F $ is a function on $ \mH $ of several real-valued components we write $ F \in HC^1 $ (respectively $ F \in HW_{\text{loc}}^{1,r}$) if each component is in $ HC^1 $ (respectively in $ HW_{\text{loc}}^{1,r} $).

A left-invariant vector field $V$ determines a differential operator. Let $\Omega \subset \mH$ be open and let $F : \Omega \to \mR$. If $F$ is differentiable in the direction $V_p$ at all $p\in \Omega$, then $VF : \Omega \to \mR$ is given by $VF (p) = (V_p F)(p)$. When $F$ is not necessarily differentiable in the direction $V_p$ at all $p\in \Omega$, $VF$ will represent a distributional derivative on $\Omega$. If a distributional derivative acts via integration against a locally integrable function we call it a weak derivative. This will almost always be the case in this paper. Any weak derivative on $\Omega$ can be represented by many different functions, each defined almost everywhere (on $\Omega$). In this sense a weak derivative $VF$ is an equivalence class of functions, with two functions identified if they are equal almost everywhere. Despite this, it is natural to think of $VF$ as a legitimate function. Such distinctions tend to be ignored in the literature and this rarely causes any danger. Occasionally we will want to draw attention to the use of a particular representative at the risk of some clumsy language. We find it convenient to express some of the relevant differential operators using the Wirtinger-like derivatives
\begin{equation}\label{eq:wderiv}
    Z := \frac{1}{2}\left( X - i Y \right) \quad \text{and} \quad \bar{Z} := \frac{1}{2} \left( X + i Y \right).
\end{equation}

We make heavy use of the notation $\lesssim$, $\gtrsim$, and $\simeq$, writing $A \lesssim B$ to mean there exists $C > 0$ such that $A \leq C B $, $ A \gtrsim B $ to mean $ B \lesssim A $, and $ A \simeq B $ to mean there exists $ C > 0 $ with
\[
	\frac{1}{C} B \leq A \leq C B. 
\]
If $A$ or $B$ are functions then the implied $C$ is a constant in that it does not depend on any variables. It may depend on parameters. Our convention is to identify dependence on pertinent parameters in the statement of a result, using $ A \lesssim_{P_1,\ldots,P_k} B $ for $ A \leq CB $ with $C = C(P_1,\ldots,P_k)>0$ a constant dependent on the parameters $P_1, \ldots, P_k$. Similarly for $\gtrsim$ and $\simeq$. Typically we do not indicate dependence on parameters in the proofs of statements. Whenever we say that $A$ and $B$ are comparable, we mean that $A \simeq B$.

If at any time $C$ appears without introduction then it represents a positive constant, whose value may change at each use, and whose dependencies are unimportant. At times we will have need to be explicit about the form of a constant. In such cases they will typically look like $\exp\left( A K^{\frac{2}{3}} \right)$ where $K$ is a number that has already been introduced. In such an expression $A>0$ is an absolute constant whose value is unimportant, and the value taken by $A$ may vary at each appearance, even within the same line. For example, \eqref{eq:bdist} is shorthand for there exist absolute constants $A', A'' > 0$ such that
\begin{equation}\label{eq:abscaexp}
    \exp\left( - A' K^{\frac{2}{3}} \right) d( f(p) , f(q) ) \lesssim | f B(p,r) |^{\frac{1}{4}}  \lesssim \exp\left( A'' K^{\frac{2}{3}} \right) d( f(p) , f(q) ).
\end{equation}
 
\section{Quasiconformal Mappings of $\mH$} \label{sec:qc}

Aiming for an efficient summary of key aspects of the theory, we generally do not include citations in the body text of this section. Some bibliographical notes appear at the end. Definitions given here are intended to supersede any given in the first paragraphs of the introduction (that said, in the case of $M=\mH$ they are always equivalent when there is an overlap).

Let $ U , U' \subset \mH $ be open connected sets. A homeomorphism $ f : U \to U' $ is said to be a quasiconformal mapping if
\begin{equation}\label{eq:dilatdef}
	p\mapsto H_f ( p ) = \limsup_{r \to 0} \frac{ \max_{d(p,q) = r} d( f(p) , f(q) ) }{ \min_{d(p,q) = r} d( f(p) , f(q) ) }
\end{equation}
is bounded on $U$. The function $p\mapsto H_f (p)$ is called the dilatation of $ f $ at $ p $. We are only interested in whole-space mappings: by quasiconformal mapping we mean a homeomorphism $ f : \mH \to \mH $ with dilatation bounded on $\mH$. Recall we denote the family of quasiconformal mappings by $\cQ$. We call $ f \in \cQ $ a $ K $-quasiconformal mapping and write $f \in \cQ (K)$ if the dilatation is (not only bounded but also) essentially bounded by $ K $ (necessarily $ 1 \leq K < \infty $). Such $ K $ is named the essential dilatation $ f $. It is convenient to define $ K(f) = \esssup H_f $ for quasiconformal $ f $.   

A quasiconformal mapping $f$ is Pansu-differentiable ($\cP$-differentiable) at $p$ if the mappings  
\[
	q \mapsto \delta^{-1}_{s} \left[ f(p)^{-1} f( p \, \delta_s (q) ) \right], \quad q \in \mH
\]
converge locally uniformly as $ s \to 0 $ to a homomorphism of $ \mH $. If this is the case we denote the resulting homomorphism by $h_p f$. A quasiconformal mapping is $\cP$-differentiable almost everywhere. At a point $p$ of $\cP$-differentiability, $h_p f$ gives rise via the exponential mapping to a Lie algebra homomorphism $ (h_p f)_* $. In these circumstances, the horizontal partial derivatives $ X f_1$, $Y f_1$, $X f_2$, and $Y f_2 $ exist at $p$ and $ (h f)_* $ acts on $ \mathfrak h $ with respect to the basis $ \{ X_0 , Y_0 , T_0 \} $ via the matrix
\[
	\cP f =
	\begin{pmatrix}
		X f_1 	& Y f_1 	& 0 \\
		X f_2 	& Y f_2 	& 0 \\
		0		& 0			& X f_1 Y f_2  - X f_2 Y f_1 
	\end{pmatrix}.
\]
The horizontal differential of $ f $ is the matrix $ D_H f $ defined by the relationship
\begin{equation} \label{eq:dhf}
	\cP f = \begin{pmatrix}
					D_H f 	& 0 \\
					0 		& \det D_H f 
				\end{pmatrix},  
\end{equation}
or to be explicit
\[
	D_H f = \begin{pmatrix}
				X f_1 	& Y f_1 \\
				X f_2 	& Y f_2 
	\end{pmatrix}.
\]

The Jacobian of a quasiconformal mapping $ f $ is
\[
	J_f := \det \cP f.
\]
By comments above it exists at almost every $ p \in \mH $. This agrees (at points of existence) with the definition given in \eqref{eq:volderiv} of the introduction (the Jacobian as volume derivative). Note that $ J_f = \det^2 D_H f $. If $ f $ is a $ K $-quasiconformal mapping then
\begin{equation} \label{eq:jdef}
	| D_H f |^4 \leq K^2 J_f
\end{equation}
almost everywhere. Indeed, if $ f $ is quasiconformal then $ f $ is $ K $-quasiconformal if and only if
\begin{equation} \label{eq:analdef}
	| D_H f |^2 \leq K \det D_H f
\end{equation}
almost everywhere.

A mapping $ f : \mH \to \mH $ is contact at $ p \in \mH $ if $ X f_3$ and $Y f_3 $ exist at $p$ with
\begin{align}
	X f_3	&= 2 f_2 X f_1 - 2 f_1 X f_2 \quad \text{and} \label{eq:contact1} \\
	Y f_3	&= 2 f_2 Y f_1 - 2 f_1 Y f_2. \label{eq:contact2} 
\end{align}
If $ f : \mH \to \mH $ is $\cP$-differentiable at $ p \in \mH $ then it is contact at $p$. A quasiconformal mapping is weakly contact in that it is contact almost everywhere. This is a prerequisite for a mapping to act in a constrained manner with respect to the Heisenberg geometry. Suppose a mapping $ f : \mH \to \mH $ is differentiable at a point $ p $ in the Euclidean sense and contact at that point. Then the Euclidean differential $ Df $ maps $ \text{H}_p \mH $ (the horizontal layer at $ p $) to $ \text{H}_{f(p)} \mH $. On the other hand, if $ f $ is $ \cP $-differentiable at $ p $, then the restriction of $ f $ to $ p \exp \left[ \text{span} ( X_0 , Y_0 ) \right] $ is differentiable in the Euclidean sense at $p$. This latter derivative is given by $ h f_* $ restricted to the horizontal layer of $ \mathfrak h $. The matrix of this restriction (with regard to the appropriate bases) is given by $ D_H f $. We will discuss the Sobolev regularity of quasiconformal mappings briefly in Section \ref{sec:qcf}.

We define a conformal mapping to be a $1$-quasiconformal mapping (here as elsewhere we consider only mappings defined on all of $\mH$). It turns out that these are $C^\infty$-smooth and equal to a M\"{o}bius-like mapping (in that they are a composition of a finite number of left-translations, dilations, and rotations about the vertical axis) though we do not require this result.

The following lemma is well known (which is not to say the argument is brief).

\begin{lem} \label{lem:qcinv}
If $ f $ is a $ K $-quasiconformal mapping then $ f^{-1} $ is also $ K $-quasiconformal.
\end{lem}

The next lemma has not been as oft-used as its Euclidean counterpart (and is not readily found in the literature) therefore we provide the short proof. Here as elsewhere in this section we rely on deeper results that are glossed over.

\begin{lem}
If $ f_1 $ and $ f_2 $ are $ K_1 $ and $ K_2 $-quasiconformal mappings respectively, then $ f_1 \circ f_2 $ is a $ (K_1 K_2) $-quasiconformal mapping.
\end{lem}

\begin{proof}
From the quasisymmetric characterization of quasiconformal mappings (discussed below) it is easy to see that $ f_1 \circ f_2 $ is quasiconformal. The only question is with regard to the essential dilatation. For this we use the analytic characterization \eqref{eq:analdef}. For all $ E \subset \mH $,
\[
	| E | = 0 \iff | f_2 E | = 0.
\] 
It follows there is a set $E$ with $ | \mH \setminus E | = 0 $ such that for all $ p \in E $, $ f_1 \circ f_2 $ is $ \cP $-differentiable at $p$, $ f_2 $ is $ \cP $-differentiable at $p$, and $ f_1 $ is $ \cP $-differentiable at $ f_2 (p) $. A calculation similar to that in a typical proof of the chain rule for the traditional derivative shows we have the following,
\[
	h_{p} ( f_1 \circ f_2 ) = h_{ f_2 (p) } f_1 \circ h_{p} f_2 .
\]
Since
\[
	( h_{ f_2 (p) } f_1 \circ h_{p} f_2 )_* = ( h_{ f_2 (p) } f_1 )_* \circ ( h_{p} f_2 )_*  
\]
we have
\[
	\cP ( f_1 \circ f_2 ) = ( \cP f_1 \circ f_2 ) \cP f_2,
\]
and consequently
\[
D_H (f_1 \circ f_2) = ( D_H f_1 \circ f_2) D_H f_2.
\]
It follows that
\begin{align*}
	| D_H (f_1 \circ f_2) |^2
	&\leq | D_H f_1 \circ f_2 |^2 | D_H f_2 |^2 \\
	&\leq K_1 K_2 \det ( D_H f_1 \circ f_2) \det D_H f_2 \\
	&= K_1 K_2 \det D_H ( f_1 \circ f_2 ).							
\end{align*}
Thus $ f_1 \circ f_2 $ is $ (K_1 K_2) $-quasiconformal by \eqref{eq:analdef}. 
\end{proof}

Let $ f $ be a $ K $-quasiconformal mapping and let $ p \in \mH $. Consider the quantity
\[
	H_{f,p} (r,s) := \frac{ \max_{d(p,q) = r} d( f(p) , f(q) ) }{ \min_{d(p,q) = s} d( f(p) , f(q) ) }
\]
for $ 0 < s \leq r < \infty $. It can be shown that
\begin{equation} \label{eq:wqs}
	H_{f,p} (r,s) \leq \exp\left(AK^{\frac{2}{3}}\right)\left(\frac{r}{s}\right)^{K^{\frac{2}{3}}}. 
\end{equation}
Here the convention exemplified by \eqref{eq:abscaexp} is in place. Given the importance of \eqref{eq:wqs} to our development we break with the style of this section and observe this follows from a combination of (the proofs of) Proposition 12 of \cite[p.~48]{KRII} and Lemma 3.2 of \cite[p.~273]{Capogna&Tang}.

If $ p , q , u \in \mH $ are distinct points such that $d( p , q ) \leq d( p , u ) $ then (since $f$ is a homeomorphism) we have
\[
d( f(p) , f(q) ) \leq \max_{ d(p,w) = d(p,u) } d( f(p) , f(w) ).
\] 
Combining this with \eqref{eq:wqs} in the case $r=s=d(p,u)$ we get 
\begin{equation}\label{eq:wqs2}
    d( f(p) , f(q) ) \leq \exp\left( AK^{\frac{2}{3}} \right) \min_{ d(p,w) = d(p,u) } d( f(p) , f(w) ) \leq \exp\left( AK^{\frac{2}{3}} \right) d( f(p) , f(u) ).  
 \end{equation}
This says $ f $ is $H$-weakly-quasisymmetric with $H = \exp\left( AK^{2/3} \right)$ (so dependent on $ K $ only). It happens to be true that $ f $ is also quasisymmetric: there exists a homeomorphism $ \eta : [0,\infty) \to [0,\infty) $ such that for distinct $ p , q , u \in \mH $,
\begin{equation} \label{eq:qscontrol}
	\frac { d( f(p) , f(q) ) }{ d( f(p) , f(u) ) } \leq \eta \left( \frac{ d(p,q) }{ d(p,u) } \right). 
\end{equation}
Let us deduce some easy consequences of \eqref{eq:wqs}.

First, there exists $ C = C(K) = \exp\left( AK^{2/3} \right) >0 $ such that for all $ p \in \mH $ and for all $ r > 0 $ there is $ s = \min_{d(p,q)=r} d(f(p),f(q)) > 0 $ with
\[
	B( f(p) , s ) \subset f B(p,r) \subset B( f(p) , C s ).
\]
Indeed, it is frequently useful that 
\begin{equation} \label{eq:bdist}
    \exp\left( - AK^{\frac{2}{3}} \right) d( f(p) , f(q) ) \lesssim | f B(p,r) |^{\frac{1}{4}}  \lesssim \exp\left( AK^{\frac{2}{3}} \right) d( f(p) , f(q) )
\end{equation}
when $ q $ is any point on $ S(p,r) $.

Now consider $ g $, also $ K $-quasiconformal but with $ g(0) = 0 $. Then \eqref{eq:wqs} leads easily to
\[
	\| g(p) \| \lesssim_K \| g(q) \| \left( 1 + \left( \frac{ \| p \| }{ \| q \| } \right)^{ K^{\frac 2 3} } \right)
\]
for all $ p , q \in \mH $. Suppose $ g_i $ ($ i \in I $, $ I $ some index set) is a family of $ K $-quasiconformal mappings (the same $ K $ for each $ i $) each of which fixes $ 0 $. Furthermore, suppose there exist $ D , D' > 0 $ such that for each $ i\in I $ there is $ q_i \in \mH $ with $ \| q_i \| \geq D $ and $ \| g(q_i) \| \leq D' $. Then we have a uniform distortion estimate for the $ g_i $:
\begin{equation} \label{eq:udist}
	\| g_i (p) \| \lesssim_{K,D'} 1 + \left( \frac{ \| p \| }{ D } \right)^{K^{2/3}}, \quad p \in \mH.
\end{equation}
We typically use this estimate with $ D = D' = 1 $ and for this reason introduce the following notation/definition.

\begin{defn} \label{defn:qsub0}
We write $g \in \cQ_0 (K)$ to mean that $g \in \cQ (K)$ with $g(0) = 0$ and there exists $p \in S(1)$ with $g(p) \in S(1)$. 
\end{defn}

The next lemma follows by specializing the discussion preceding the definition.

\begin{lem} \label{lem:udist}
Given $ R > 0 $ and $K \geq 1$, there exists $ R'=R'(K,R) > 0 $ such that for all $ g \in \cQ_0(K) $,
\[
	g B(R) \subset B(R').
\]
\end{lem}

Quasiconformal mappings are locally H\"older continuous. Given $K$-quasiconformal mapping $ f $ and $ R > 0 $, let $ R' > 0 $ be such that $ fB(3R+1) \subset B(R') $. Then there exists $ \alpha = \alpha(K) > 0 $ such that for all $p,q\in B(R)$,
\begin{equation} \label{eq:hold}
	d(f(p),f(q)) \lesssim_{K,R,R'} \, d(p,q)^\alpha .
\end{equation}
If we combine this with Lemma \ref{lem:udist} we achieve the following result.

\begin{lem} \label{lem:uhold}
Given $ R > 0 $ and $K \geq 1$, there exists $ \alpha = \alpha(K) > 0 $ such that for all $ g \in \cQ_0 (K) $,
\[
	d(g(p),g(q)) \lesssim_{K,R} \, d(p,q)^\alpha
\]
for all $ p , q \in B(R) $.
\end{lem}
Crucially, in the previous lemma the implied constant is dependent on $ g $ only through its dependence on $ K $.

We now record the various results that are pertinent to our focus on the quasiconformal Jacobian. On numerous occasions we use that a quasiconformal mapping $ f $ satisfies the change of variable formula
\begin{equation} \label{eq:cvf}
	\int_{ f \Omega } u = \int_{ \Omega } ( u \circ f ) J_{f},
\end{equation}
valid for all non-negative, measurable functions $ u : \mH \to \mR $ and measurable $ \Omega \subset \mH $ (with the necessary measurability of $ ( u \circ f ) J_{f} $ part of the result).

The formula just recorded relies on the fact that $J_f > 0$ almost everywhere, and $ J_f \in L_{\text{loc}}^1 $. Actually, more is true: the Jacobian of a quasiconformal mapping $ f $ satisfies a reverse H\"{o}lder inequality. The power of this fact will be amply demonstrated by multiple appearances at critical moments later. To be precise, if $ f $ is a $ K $-quasiconformal mapping there exists $ r = r(K) > 1 $ such that, if $ B \subset \mH $ is a ball then
\begin{equation} \label{eq:rhi}
	\left( \frac{ 1 }{ | B | } \int_B J_{f}^{r} \right)^{ \frac{ 1 }{ r } } \lesssim_{K} \frac{ 1 }{ | B | } \int_{B}J_{f}.
\end{equation}  
We stress that $r$ and the implied constant are independent of $ B $ (indeed both can be taken to depend on $ K $ only). That $ J_f $ satisfies a reverse H\"{o}lder inequality implies it is an $ A_\infty $ weight as in \cite{BigStein}. It is also true that the inequality can be shown to imply the $ A_p $ condition for some $ 1 \leq p < \infty $ (the calculation can be found in \cite{BigStein}). We do not record the $ A_p $ condition here, but observe that it has the following easy implication: there exists $\alpha > 0$ such that if $ B \subset \mH $ is a ball, 
\begin{equation} \label{eq:apw}
	\frac{ 1 }{ | B | } \int_B J_{f}^{ -\alpha } \lesssim \left( \frac{ | B | }{ | f(B) | } \right)^{ \alpha }
\end{equation}  
independently of $ B $ (see (45) of \cite[p.~212]{BigStein}).

Ultimately, the mapping we construct (with Jacobian comparable to a given weight) will be found in the limit of a sequence of mappings. We therefore need to be able to say something useful about the limiting behavior of the Jacobians. The following weak convergence result will suffice.

\begin{lem} \label{lem:wcj}
Suppose $ (f_m) $ is a sequence of quasiconformal mappings converging locally uniformly to a quasiconformal mapping $ f $. Suppose also that the $ f_{m}^{-1} $ converge pointwise to $ f^{-1} $. Then given $ \xi \in C_{0}^{\infty} $ with $ \xi \geq 0 $,
\[
	\lim_{ m \to \infty } \int \xi J_{ f_{m} } = \int \xi J_{ f }.
\] 
\end{lem}

\begin{proof}
Let $ R > 0 $ be such that $ \text{support}(\xi) \subset B(R) $. As the $ f_m $ converge locally uniformly, there exists $ R' > 0 $ such that $ f_m B(R) \subset B(R') $ for all $ m $. It follows that $\text{support}( \xi \circ f_m^{-1} ) \subset B(R')$ for all $ m $. Using the change of variable formula \eqref{eq:cvf},
\begin{align*}
	\int \xi J_{f_m}
	&= \int ( \xi \circ f_m^{-1} )( f_m ) J_{ f_m } \\
	&= \int  \xi \circ f_m^{-1} .
\end{align*}
Since $ | \xi \circ f_m^{-1} | \leq \max( | \xi | ) \chi_{B(R')} $ for all $ m $, the dominated convergence theorem applies and we conclude that
\[
	\lim_{ m \to \infty } \int \xi J_{f_m} = \int \xi \circ f^{-1} = \int \xi J_{f}. \qedhere
\]
\end{proof}

We end this section with some instances in which a sequence $(f_m)$ of quasiconformal mappings converges to a quasiconformal mapping $f$. These are mild variations of well known results with the statements tailored to our purpose. 

\begin{lem} \label{lem:qsconv}
Suppose $ (f_m) $ is a sequence in $\cQ_0 (K)$ such that there exists $p_0 \in S(1)$ with $f_m (p_0) \in S(1)$ for all $m$. Then the $ f_m $ subconverge locally uniformly to a $ K $-quasiconformal mapping $ f $. Furthermore, any convergent subsequence $ (f_{m_{k}}) $ has the $ f_{m_{k}}^{-1} $ converging pointwise to $ f^{-1} $. 
\end{lem}

\begin{proof}
Local uniform subconvergence of the $ f_m $ to a quasiconformal mapping is standard in these circumstances. That the essential dilatation of the limit mapping is the same as those of the sequence is somewhat less expected, a proof can be found in \cite{KRII}. We are left to prove the statement regarding the inverses (which one would think was automatic but we have no better argument than the following).

Abusing notation, let $ (f_m) $ be a convergent subsequence. By Lemma \ref{lem:qcinv} each $ f_m^{-1} $ is $ K $-quasiconformal. It is also true that $ f_m^{-1}(0)=0 $ for all $ m $. Our assumption regarding the existence of $ p_0 $ implies that for each $ m $, there exists $ p_m $ with $ \| p_m \| = 1 $ and $ \| f_m^{-1}(p_m) \| = \| p_0 \| = 1 $. It follows that $ (f_m^{-1}) \subset \cQ_0(K) $.

Choose some $ q \in \mH $, and let $ p $ be such that $ f(p) = q $ (where $ f = \lim f_m $). There is $ 0 < R < \infty $ such that $ f(p) \in B(R) $ and $ f_m(p) \in B(R) $ for all $ m $. As in Lemma \ref{lem:uhold}, let $ \alpha > 0 $ be such that
\[
	d ( f_m^{-1}( u_1 ) , f_m^{-1}( u_2 ) ) \lesssim d( u_1 , u_2 )^{\alpha} 
\]    
for all $ u_1 , u_2 \in B(R) $ independently of $ m $. Then
\begin{align*}
	d( f_m^{-1}(q) , p ) 	&= d( f_m^{-1}(f(p)) , f_m^{-1}( f_m ( p ) ) ) \\
							&\lesssim d( f(p) , f_m (p) )^\alpha.
\end{align*}
Consequently, $ \lim f_m^{-1}(q) = p = f^{-1}(q) $ as required.  
\end{proof}
Once it is known the $ f_m^{-1} $ converge pointwise local uniform convergence follows (in the specified circumstances of the lemma), but this is not a hypothesis of Lemma \ref{lem:wcj} so we can make do without it.

\begin{lem} \label{lem:jacconv} 
Suppose $ (f_m) $ is a sequence of $ K $-quasiconformal mappings with $ f_m (0) = 0 $ for all $ m $ and
\[
	\int_{B(1)} J_{ f_m } \simeq 1
\]
independently of $ m $. Then the $ f_m $ subconverge locally uniformly to a $ K' $-quasiconformal mapping. Furthermore, if $ (f_{m_k}) $ is a convergent subsequence then the $ f_{m_k}^{-1} $ converge pointwise to $ f^{-1} $.
\end{lem}

\begin{proof}
Fix a point $ p_0 \in S(1) $. It follows from \eqref{eq:bdist} and \eqref{eq:cvf} that
\[
	\int_{B(1)} J_{f_m} \simeq_{K} d( f_m (0) , f_m ( p_0 ) )^4
\]
independently of $m$. Consequently, our assumption on the size of the integral means that  
\[
	d( f_m (0) , f_m ( p_0 ) ) \simeq 1
\]
independently of $m$. This, coupled with $ f_m (0) = 0 $ for all $ m $, is enough to conclude the locally uniform subconvergence using well known compactness properties of quasisymmetric mappings (in other words we have essentially reduced to the hypotheses of Lemma \ref{lem:qsconv}).

As for the pointwise convergence of $ f_m^{-1} $ to $ f^{-1} $ (with respect to a convergent subsequence we denote $(f_m)$), the argument is largely the same as that for Lemma \ref{lem:qsconv} (we just need to work a little harder). We have the existence of $ 0 < R \leq R' < \infty $ such that for each $ m $ there exists a point $ p_m $ with $ R \leq \| p_m \| \leq R' $ and $ \| f_m^{-1}(p_m) \| = \| p_0 \| = 1 $. We have, therefore, a uniform distortion estimate for the $ f_m $ as in \eqref{eq:udist}. We can use this to derive H\"{o}lder continuity with uniform constants (as in \eqref{eq:hold}) on a suitable ball, then proceed as in the proof of Lemma \ref{lem:qsconv}. 
\end{proof}

\textbf{Notes to Section \ref{sec:qc}:} The primary reference for the results of this section is \cite{KRII}. That said, the statements of \cite{KRII} are for the $n$-th Hesisenberg group, $1\leq n < \infty$. Our present context is $\mH = \mH^1$ which is also the setting for \cite{KRI}. The reader may like to consult \cite{KRI} for an (easier) introduction to the basics. The main obstacle to our citing \cite{KRI} more frequently (other than that certain aspects of the theory simply are not explored) is that usually more regularity is assumed than is suitable for our purposes. For the almost everywhere differentiability of quasiconformal mappings see \cite{PansuCC}. For the matrix of the $\cP$-differential see \cite{DairbekovHGroup}. The analytic criterion \eqref{eq:analdef} can be found in \cite{KRII} along with Lemma \ref{lem:qcinv}. The case $ r = s $ of \eqref{eq:wqs} is in \cite{KRII}, the general form is in \cite{Capogna&Tang}. It should be observed that the proof of \eqref{eq:wqs} rests on a suitable capacity estimate as proved in \cite{ReimannCapacity}.  The local H\"{o}lder continuity of quasiconformal mappings is in \cite{KRII}. The change of variable formula \eqref{eq:cvf} is in \cite{DairbekovHGroup}. The reverse H\"{o}lder inequality is proved in \cite{KRII}. Lemma \ref{lem:qsconv} is in \cite{KRII} -- but these things hold for quasisymmetric mappings in a more general setting with the results nicely stated in \cite{HeinonenLectures}.

\section{Some Lemmas About Quasilogarithmic Potentials} \label{sec:qlplems}

In Section \ref{sec:ic} we will have need to work with regularized or restricted measures, achieving our desired outcomes as we take a limit of regularizations or restrictions. For this reason we need to know that the relevant objects behave well in the limit, and the purpose of this section is to ascertain exactly that (the pertinent results being Lemmas \ref{lem:theta1} and \ref{lem:theta2}). While the development follows very closely that of Section 4 of \cite{BHSII}, that things go through as they do could be considered non-obvious in this new setting; thus we include the proofs for the convenience for the reader.

In this section we will largely write $\left\| q^{-1}p \right\|$ for $d(p,q)$. It will be useful at times to keep in mind that the polar integration formula \eqref{eq:polarint} can be used to show that for all $p\in \mH$, for all $\lambda < 4$, and for all $R>0$,
\[
	\int_{B (R)} \frac{1}{\| q^{-1}p \|^\lambda} \,\mathrm{d}\nu (q) \lesssim_{R, \lambda}\, 1.
\]	
We also make use of the following elementary inequalities, both here and in later sections:
\begin{equation}
    \log^+ (ab) \leq \log^+ (a) + \log^+ (b), \quad a,b > 0, \text{ and}
\end{equation}
\begin{equation}
    \log^+ (a+b) \leq 1 + \log^+ (a) + \log^+ (b), \quad a,b > 0.
\end{equation}
In order to streamline the exposition, in general we will not make explicit reference to the use of the observations contained in this paragraph.

Before getting to the results regarding good approximations of (quasi)logarithmic potentials, we first take care of a result of a different nature that will also play a crucial role later. To be precise, we will prove that in certain special circumstances a logarithmic potential is a Lipschitz function on $\mH$. This is achieved in Lemma \ref{lem:lpsilip} below.

For $\mu$ a signed measure on $\mH$ we define the maximal function,
\[
	M\mu(p) = \sup_{r>0} \frac{1}{| B (p,r) |}\int_{B (p,r)} \,\mathrm{d} |\mu|(q).
\]
It can be shown (see \cite[p.~43]{BigStein}) that if $\mu$ is finite then for all $\alpha > 0$,
\[
    | \{ p \in \mH \,:\, M\mu (p) > \alpha \} | \lesssim \frac{\| \mu \|}{\alpha}.
\]
This supplies the proof of the next lemma.

\begin{lem}\label{lem:maxfinae}
If $\mu$ is a finite signed measure on $\mH$ then $M\mu < \infty$ almost everywhere.
\end{lem}

The maximal function of $\mu$ provides a useful control of the logarithmic potential of $\mu$ via the following.

\begin{lem} \label{lem:finmax}
Let $\mu$ be a signed measure on $\mH$. If $p\in \mH$ is such that $M\mu(p) < \infty$ then
\[
	\int \log^+ \left( \frac{1}{\| q^{-1}p \|} \right)\,\mathrm{d}|\mu|(q) < \infty.
\]
\end{lem}
\begin{proof}
By the definition of $\log^+$ we have
\[
	\int \log^+ \left( \frac{1}{\| q^{-1}p \|} \right)\,\mathrm{d}|\mu|(q) = \int_{B (p,1)} \log^+ \left( \frac{1}{\| q^{-1}p \|} \right)\,\mathrm{d}|\mu|(q).
\]	
Let $A (p,r,R)$ be the annulus $B(p,R)\setminus B(p,r)$. Decomposing along annuli, making use of the Ahlfors $4$-regularity of $\mH$, and invoking the assumption on $M\mu(p)$ we find,
\begin{align*}
	\int_{B (p,1)} \log^+ \left( \frac{1}{\| q^{-1}p \|} \right)\,\mathrm{d}|\mu|(q) &= \sum_{j=0}^{\infty} \int_{A (p,2^{-(j+1)},2^{-j})} \log^+ \left( \frac{1}{\| q^{-1}p \|} \right)\,\mathrm{d}|\mu|(q) \\
	&\lesssim \sum_{j=0}^{\infty} 2^{j+1} \int_{A (p,2^{-(j+1)},2^{-j})}\,\mathrm{d}|\mu|(q) \\
	&\lesssim \sum_{j=0}^{\infty} 2^{j+1}\frac{2^{-4j}}{\left| B (p,2^{-j}) \right|}\int_{B (p,2^{-j})}\,\mathrm{d}|\mu|(q) \\
	&\lesssim \sum_{j=0}^{\infty} \left(\frac{1}{2^{3}} \right)^j <\infty. \qedhere
\end{align*}
\end{proof}	

Recall that a finite signed measure on $\mH$ is called admissible if $\int \log^+ \| q \| \,\mathrm{d}|\mu|(q) < \infty$ and we write $\cM$ for the set of admissible measures. Recall also that for $\mu \in \cM$ the logarithmic potential $\Lambda_\mu$ is given by $\Lambda_\mu (p)= -\int \log \| q^{-1}p \| \,\mathrm{d}\mu(q)$.

\begin{lem}\label{lem:finwmfin}
Let $\mu \in \cM$. Then $\Lambda_{\mu}$ is finite almost everywhere.	
\end{lem}

\begin{proof}
Suppose $p\in \mH$ is a point at which $M\mu$ is finite. It follows, 
\begin{align*}
	|\Lambda_{\mu} (p)| &\leq \int |\log \| q^{-1}p \| |\,\mathrm{d}|\mu|(q) \\
	&= \int  \log^+\| q^{-1}p \| + \log^+\left(\frac{1}{\| q^{-1}p \|} \right)  \,\mathrm{d}|\mu|(q) \\
	&\leq \int 1 + \log^+ \|p\| + \log^+ \|q\| + \log^+\left(\frac{1}{\| q^{-1}p \|} \right)  \,\mathrm{d}|\mu|(q).
\end{align*}
This is finite by the admissibility of $\mu$ and Lemma \ref{lem:finmax}. The statement is now immediate from Lemma \ref{lem:maxfinae}.
\end{proof}

If $\psi \in L^1$ then
\[
	\Lambda_{\psi} (p) := -\int \log \| q^{-1}p \| \,\psi(q)\mathrm{d}q.
\]

\begin{lem}\label{lem:lpsilip}
Let $\psi \in L^\infty$ with compact support, and let $R>0$ be such that $\text{support}(\psi)\subset B (R)$. Then $\Lambda_{\psi}$ is $L$-Lipschitz continuous with $L = L(\|\psi\|_\infty ,R)$.
\end{lem}
\begin{proof}
Let $\mu$ be the signed measure determined by $\mathrm{d}\mu(q) = \psi(q)\mathrm{d}q$. The obvious inequality
\[
	\int \log^+ \|q\| \,|\psi(q)|\mathrm{d}q \leq \| \psi \|_\infty \int_{B(R)} \log^+ \|q\| \,\mathrm{d}q < \infty
\]
identifies $\mu$ as an admissible measure. It is also easy to see that for all $p\in \mH$, $M\mu(p)\leq \|\psi\|_\infty$. Looking at the proof of Lemma \ref{lem:finwmfin} we see $\Lambda_\psi$ is finite everywhere. If $p,u, q\in \mH$ are such that $\| q^{-1}p \|,\|q^{-1}u\|>0$ then single variable calculus and the triangle inequality tell us
\begin{align*}
	|\log \| q^{-1}p \| - \log\|q^{-1}u\|| &= \left| \int_{\|q^{-1}u\|}^{\| q^{-1}p \|} \frac{1}{s}\,\mathrm{d}s \right|\\
	&\leq \max\left\{\frac{1}{\|q^{-1}u\|} ,\frac{1}{\| q^{-1}p \|} \right\}|\| q^{-1}p \| - \|q^{-1}u\||\\
	&\leq \left(\frac{1}{\|q^{-1}u\|}+\frac{1}{\| q^{-1}p \|}\right) \|u^{-1}p\|.
\end{align*}
We now find that $\Lambda_{\psi}$ is Lipschitz as follows,
\begin{align*}
 |\Lambda_{\psi}(p)-\Lambda_{\psi}(u)| &\leq \|\psi\|_\infty \|u^{-1}p\| \int_{B (R)}\frac{1}{\|q^{-1}u\|}+\frac{1}{\| q^{-1}p \|}\,\mathrm{d}q \\
 &\lesssim \|u^{-1}p\|. \qedhere
\end{align*}
\end{proof}

For $\mu \in \cM$ let 
\[
	c_{\mu}:= \int \log^+ \|q\| \,\mathrm{d}|\mu|(q).
\]
When $\mu$ is determined by $d\mu(q) = \psi(q)\mathrm{d}q$ with $\psi \in L^1$ we write $c_{\psi}$ for $c_{\mu}$.

Again for $\mu \in \cM$, we use the Jordan decomposition to define a helper function
\[
    \Lambda_{\mu}^+ (p) = \int \log^+ (\| q^{-1}p \|)\,\mathrm{d}\mu_- (q) + \int \log^+ \left(\frac{1}{\| q^{-1}p \|}\right)\,\mathrm{d}\mu_+ (q).
\]

\begin{lem}\label{lem:eblest}
Let $\mu \in \cM$ and let $R>0$. For every $0<\beta<\frac{4}{\| \mu \|}$ we have
\[
	\int_{B (R)} \exp\left( \beta \Lambda_\mu^+ (p) \right)\,\mathrm{d}p \lesssim_{R,c_{\mu},\beta,\|\mu\|} 1.	
\]
\end{lem}
\begin{proof}
Let $p \in B (R)$. We have
\[
	\int \log^+ \| q^{-1}p \|\,\mathrm{d}\mu_- (q) \leq \int  \left( 1 + \log^+ \|p\| + \log^+ \|q\| \right)  \,\mathrm{d}|\mu|(q) \lesssim 1
\]
so that if $\|\mu_+\| = 0$ there is nothing left to prove.

If $\|\mu_+\| > 0$ we define probability measure $\kappa = \frac{\mu_+}{\|\mu_+\|}$. Let $\beta$ be as in the statement and set $\lambda = \beta\|\mu_+\|<4$. The observation made in the first part of this proof followed by an application of Jensen's inequality give
\begin{align*}
	\int_{B (R)} \exp\left( \beta\Lambda_\mu^+ (p) \right)\,\mathrm{d}p &\lesssim \int_{B (R)}\exp\left( \beta \int \log^+ \left(\frac{1}{\|q^{-1}p\|}\right)\,\mathrm{d}\mu_+ (q) \right)\,\mathrm{d}p\\
	&=\int_{B (R)} \exp\left( \lambda\int \log^+\left( \frac{1}{\|q^{-1}p\|} \right) \,\mathrm{d}\kappa(q) \right) \,\mathrm{d}p\\
	&\leq \int_{B (R)} \int \max\left\{1,\frac{1}{\|q^{-1}p\|^{\lambda}}\right\} \,\mathrm{d}\kappa(q)\mathrm{d}p.
\end{align*}
It follows by Fubini's theorem,
\[
	\int_{B (R)} \exp\left( \beta\Lambda_\mu^+ (p) \right)\,\mathrm{d}p \lesssim \int \int_{B (R)} \left(1+\frac{1}{\|q^{-1}p\|^\lambda}\right) \,\mathrm{d}p\mathrm{d}\kappa(q) \lesssim 1. \qedhere
\]
\end{proof}

The many invocations of Fubini's theorem in the remainder of the section will be made tacitly.

\begin{cor}\label{cor:eblest}
Let $\mu \in \cM$ and let $R>0$. For every $0<\beta<\frac{4}{\|\mu\|}$ we have
\[
	\int_{B (R)} \exp\left( \beta \Lambda_\mu (p) \right)\,\mathrm{d}p \lesssim_{R,c_{\mu},\beta,\|\mu\|} 1.	
\]
\end{cor}
\begin{proof}
Notice that $\Lambda_\mu (p)$ is the sum of four parts,
\begin{align*}
\Lambda_{\mu}(p) = \int_{B(p,1)} \log \|q^{-1}p\|\,\mathrm{d}&\mu_{-}(q) + \int_{\mH\setminus B(p,1)} \log \|q^{-1}p\|\,\mathrm{d}\mu_{-}(q) \\ &-\int_{B(p,1)} \log \|q^{-1}p\|\,\mathrm{d}\mu_{+}(q)-\int_{\mH\setminus B(p,1)} \log \|q^{-1}p\|\,\mathrm{d}\mu_{+}(q).
\end{align*}
Looking at the sign of each part we find
\begin{align*}
    \Lambda_{\mu}(p) &\leq \int_{\mH\setminus B(p,1)} \log \|q^{-1}p\|\,\mathrm{d}\mu_{-}(q) + \int_{B(p,1)} \log \left( \frac{1}{\|q^{-1}p\|} \right)\,\mathrm{d}\mu_{+}(q)\\
    &= \int \log^+ \|q^{-1}p\|\,\mathrm{d}\mu_{-}(q) + \int \log^+ \left(\frac{1}{\|q^{-1}p\|}\right)\,\mathrm{d}\mu_{+}(q) = \Lambda_{\mu}^+ (p).
\end{align*}
It follows that for all $\beta > 0$, $\exp(\beta\Lambda_{\mu}(p)))\leq \exp(\beta\Lambda_{\mu}^+ (p)))$ and the desired conclusion is obtained from Lemma \ref{lem:eblest}.
\end{proof}

It is time to introduce the regularization procedure that will play an important role later in this paper. Let non-negative $ \Psi \in C_0^{\infty} $ be such that $ \text{support}(\Psi) \subset B(1) $ and $ \int \Psi = 1 $. For each $ k \in \mN $ let
\begin{equation} \label{eq:regularizers}
\Psi_k (p) := k^4 \Psi( \delta_k (p) ). 
\end{equation}
Given a finite signed measure $ \mu $, the $ k^{\text{th}} $ smooth regularization of $ \mu $ is
\begin{equation} \label{eq:regmu}
 \psi_k (p) = \psi_k^{\mu} (p) = \int \Psi_k (q^{-1}p) \, \mathrm{d}\mu(q).
\end{equation}

\begin{lem}\label{lem:flmpsik}
If $\mu$ is a finite signed measure then for all $k \geq 1$ we have
\begin{gather}
    \|\psi_k \|_1 \leq \| \mu \| \text{ and}  \label{eq:l1bound}\\
    c_{\psi_k} = \int \left(\log^+ \|p\|\right) |\psi_k (p) | \, \mathrm{d}p \lesssim_{\|\mu\|,c_{\mu}} \,1.\label{eq:cpsibnd}
\end{gather}
\end{lem}
\begin{proof}
Estimate \eqref{eq:l1bound} is straightforward since
\[
	\| \psi_k \|_1 = \int |\psi_k (p)| \,\mathrm{d}p = \int \left| \int \Psi_k (q^{-1}p) \mathrm{d}\mu (q) \right| \mathrm{d}p \leq \int \int \Psi_k (q^{-1} p )\,\mathrm{d}p \mathrm{d}|\mu| (q) = \| \mu \|.
\]

As for \eqref{eq:cpsibnd}, the definition of $\psi_k$ gives
\begin{align*}
    \int \left(\log^+ \|p\|\right) |\psi_k (p) |\, \mathrm{d}p &\leq \int\int (\log^+ \|p\|) \Psi_k(q^{-1}p) \,\mathrm{d}|\mu|(q)\mathrm{d}p \\
    &= \int\int (\log^+ \|p\|) \Psi_k(q^{-1}p) \,\mathrm{d}p\mathrm{d}|\mu|(q).
\end{align*}
From there, a simple change of variable and the triangle inequality allow us to continue as follows,
\begin{align*}
    \int \left(\log^+ \|p\|\right) |\psi_k (p) |\, \mathrm{d}p &\leq \int \int (\log^+ \|q p\|) \Psi_k (p) \,\mathrm{d}p\mathrm{d}|\mu|(q)\\
    &\leq \int \int \left(1 + \log^+ \|q\| + \log^+ \| p\| \right) \Psi_k (p) \,\mathrm{d}p\mathrm{d}|\mu|(q).
\end{align*}
In order to simplify the last expression we observe (among other things) that 
\[
    \int (\log^+ \|p\|) \Psi_k (p)\,\mathrm{d}p = \int_{B(1/k)} (\log^+ \|p\|) \Psi_k (p)\,\mathrm{d}p = 0,
\]
and find that
\[
    \int \left(\log^+ \|p\|\right) |\psi_k (p) |\, \mathrm{d}p \leq \|\mu\| + c_{\mu}
\]
as required.
\end{proof}

We pause here to introduce a known inequality that will be used in the next lemma: for all $a,b>0$ and for all $0<\alpha \leq 1$, 
\[
	|\log a - \log b| \leq \frac{1}{\alpha}|a-b|^\alpha \left(\frac{1}{a^{\alpha}} + \frac{1}{b^{\alpha}}\right).
\]	
(For the interested reader, the proof can be thought of as divided into two stages.  First, since $\log a - \log b = \frac{1}{\alpha}(\log a^{\alpha} - \log b^{\alpha})$ we have by a single variable integral estimate that
\[
	|\log a - \log b | \leq \frac{1}{\alpha} |a^{\alpha} - b^{\alpha}|\left(\frac{1}{a^{\alpha}} + \frac{1}{b^{\alpha}}\right).
\]
Second, given our assumption $0<\alpha\leq 1$, it is true that $|a^\alpha - b^\alpha|\leq |a-b|^{\alpha}$. The result follows.)

\begin{lem}\label{lem:1limonzero}
Let $\mu \in \cM$. If $B\subset \mH$ is a ball and $l \geq 1$ then
\[
	\int_B |\Lambda_{\psi_k}(p) - \Lambda_{\mu} (p)|^l \,\mathrm{d}p \to 0 \quad \text{as}\quad k\to \infty.
\]	
\end{lem}
\begin{proof}
Expand according to the definitions to find
\begin{align*}
	\int_B &|\Lambda_{\psi_k}(p) - \Lambda_{\mu} (p)|^l \,\mathrm{d}p \\
	&= \int_B \left|\int \log \|q^{-1}p\| \,\psi_k(q)\mathrm{d}q - \int\log \|u^{-1}p\| \,\mathrm{d}\mu (u)\right|^l \,\mathrm{d}p \\
	&= \int_B \left|\int\left(\int \log \|q^{-1}p\| \Psi_k(u^{-1}q)\mathrm{d}q - \log \|u^{-1}p\|\right) \,\mathrm{d}\mu (u)\right|^l \,\mathrm{d}p.
\end{align*}
Since $\int \Psi_k (u^{-1}q)\,\mathrm{d}q = 1$, we may restate this as
\[
	\int_B |\Lambda_{\psi_k}(p) - \Lambda_{\mu} (p)|^l \,\mathrm{d}p = \int_B \left|\int\int(\log \|q^{-1}p\|  - \log \|u^{-1}p\| )\,\Psi_k(u^{-1}q)\mathrm{d}q\mathrm{d}\mu (u)\right|^l \,\mathrm{d}p.
\]
Next, the inequality mentioned before this lemma and the triangle inequality give
\begin{align*}
	\int_B &|\Lambda_{\psi_k}(p) - \Lambda_{\mu} (p)|^l  \,\mathrm{d}p\\
	&\lesssim \int_B\left(\int\int |\|q^{-1}p\|-\|u^{-1}p\||^{\frac{1}{l}}\left( \frac{1}{\|q^{-1}p\|^{\frac{1}{l}}} + \frac{1}{\|u^{-1}p\|^{\frac{1}{l}}} \right) \, \Psi_k(u^{-1}q)\mathrm{d}q\mathrm{d}\mu(u)\right)^l \mathrm{d}p\\
	&\leq \int_B\left(\int\int \|u^{-1}q\|^{\frac{1}{l}}\left( \frac{1}{\|q^{-1}p\|^{\frac{1}{l}}} + \frac{1}{\|u^{-1}p\|^{\frac{1}{l}}} \right) \, \Psi_k(u^{-1}q)\mathrm{d}q\mathrm{d}\mu(u)\right)^l \mathrm{d}p.
\end{align*}
Judicious application of the H\"{o}lder inequality and a simple integral estimate results in
\begin{align*}
	\int_B &|\Lambda_{\psi_k}(p) - \Lambda_{\mu} (p)|^l  \,\mathrm{d}p\\
	&\lesssim \int_B \int\int \|u^{-1}q\|\left( \frac{1}{\|q^{-1}p\|} + \frac{1}{\|u^{-1}p\|} \right) \, \Psi_k(u^{-1}q)\mathrm{d}q\mathrm{d}\mu(u)\mathrm{d}p \\
	&\lesssim \int\int \|u^{-1}q\| \, \Psi_k(u^{-1}q)\mathrm{d}q\mathrm{d}\mu(u). 
\end{align*}
Observing that $\Psi_k (u^{-1}q)$ is non-zero only when $\|u^{-1}q\|\leq 1/k$ allows us to conclude that
\[
		\int_B |\Lambda_{\psi_k}(p) - \Lambda_{\mu} (p)|^l  \,\mathrm{d}p \lesssim \frac{1}{k}
\] 
(with the implied constant independent of $k$) and the statement easily follows.
\end{proof}	

\begin{lem}\label{lem:2limonzero}
Let $\mu \in \cM$ and let $l \geq 1$. If $\beta > 0$ is such that $\beta l < 4/\| \mu \|$ then for every ball $B\subset \mH$ we have
\[
	\int_B \left| e^{\beta \Lambda_{\psi_k}(p)} - e^{\Lambda_\mu (p)} \right|^l \,\mathrm{d}p \to 0 \quad \text{as} \quad k\to \infty.
\] 	
\end{lem}
\begin{proof}
Let $\beta > 0$ be such that $\beta l < 4/\| \mu \|$. Choose $\alpha$ such that $0 < \beta l < \alpha < 4/\| \mu \|$. It follows by \eqref{eq:l1bound} that for all $k$ we have $0<\alpha<4/\| \psi_k \|_1$. Consequently, a combination of Corollary \ref{cor:eblest} and Lemma \ref{lem:flmpsik} allows us to deduce that for all $k$,
\[
    \int \exp\left( \alpha \Lambda_{\psi_k} (p) \right) \,\mathrm{d}p \lesssim \, 1
\]
with the implied constant independent of $k$.

Let $r := \alpha / (\beta l)>1$ and let $s:= r / (r-1)$. Recalling that $e^b - e^a = \int_a^b e^x \,\mathrm{d}x$, it is easily found that $|e^a - e^b| \leq |b-a|(e^a + e^b)$ (while obvious, this appears in the next calculation sufficiently obscured as to be mysterious). The H\"{o}lder inequality, Corollary \ref{cor:eblest}, and the preceding observations can now be used to obtain
\begin{align*}
    \int_B &\left| e^{\beta \Lambda_{\psi_k}(p)} - e^{\beta \Lambda_{\mu}(p)} \right|^l \, \mathrm{d}p \\
    &\lesssim \left( \int_B \left( e^{\beta l \Lambda_{\psi_k}(p)} + e^{\beta l \Lambda_{\mu}(p)} \right)^{r} \,\mathrm{d}p\right)^{\frac{1}{r}} \left( \int_B \left|  \Lambda_{\psi_k}(p) - \Lambda_{\mu}(p) \right|^{ls} \,\mathrm{d}p \right)^{\frac{1}{s}}\\
    &\lesssim \left( \int_B \left|  \Lambda_{\psi_k}(p) - \Lambda_{\mu}(p) \right|^{ls} \,\mathrm{d}p \right)^{\frac{1}{s}}.
\end{align*}
As $r>1$ we have $ls\geq 1$, thus $\int_B \left| e^{\beta \Lambda_{\psi_k}(p)} - e^{\beta \Lambda_{\mu}(p)} \right|^l \, \mathrm{d}p \to 0$ as $k \to \infty$ by Lemma \ref{lem:1limonzero}
\end{proof}

The next two lemmas are the main purpose of this section and will be applied directly later.

Recall, if $\mu$ is a finite signed measure then $ \psi_k = \psi_k^{\mu} $ is the $ k^{\text{th}} $ smooth regularization of $ \mu $ as in \eqref{eq:regmu}.

\begin{lem} \label{lem:theta1}
Let $\mu \in \cM$, let $K\geq 1$, and let $g \in \cQ (K)$. Let $ \Lambda $ denote the quasilogarithmic potential $ \Lambda_\mu \circ g $. For each $ k \in \mN $, define $ \Lambda_k = \Lambda_{\psi_k} \circ g $. In these circumstances there exists $ \theta = \theta(K) > 0 $ such that for every $ 0 < \beta < \theta / \| \mu \| $ the function $ e^{\beta \Lambda} $ is locally integrable, and for every ball $ B \subset \mH $ we have
\[
	\int_B \left| e^{ \beta \Lambda_k (p)} - e^{ \beta \Lambda (p) } \right| \,\mathrm{d}p \to 0 \quad \text{as} \quad k \to \infty.
\]
\end{lem}

\begin{proof}
Let $\alpha>0$ be as in \eqref{eq:apw} (so that $\alpha$ depends on $K$ only). Set $\theta = \frac{4\alpha}{1+\alpha}$ and let $0<\beta<\theta/\| \mu \|$. Now set $l = 1 + 1/\alpha$ so that $\beta l < 4/\| \mu \|$.

Fix a ball $B\subset \mH$. We will first demonstrate the local integrability of $e^{\beta \Lambda}$. Invoking H\"{o}lder's inequality and the change of variable formula \eqref{eq:cvf} we find, 
\begin{align*}
 \int_B  e^{ \beta (\Lambda_\mu \circ g)(p)} \,\mathrm{d}p &= \int_B  e^{ \beta (\Lambda_\mu \circ g)(p)} J_g (p)^{\frac{1}{l}}J_g (p)^{-\frac{1}{l}} \,\mathrm{d}p \\
 &\leq \left(\int_B  e^{ \beta l (\Lambda_\mu \circ g)(p)} J_g (p) \,\mathrm{d}p\right)^{\frac{1}{l}} \left(\int_B J_g (p)^{-\alpha} \,\mathrm{d}p\right)^{\frac{1}{1+\alpha}} \\
 &\lesssim \left(\int_{g(B)}  e^{ \beta l \Lambda_\mu(q)} \,\mathrm{d}q\right)^{\frac{1}{l}}.
\end{align*}
Note the defining property of $\alpha$ was used in moving to the last line. Since $\beta l < 4 / \| \mu \|$ and there exists $R>0$ such that $g(B)\subset B(R)$, the desired conclusion follows from Corollary \ref{cor:eblest}. 

The remainder of the statement follows in a similar way; observe that
\begin{align*}
    \int_B &\left| e^{ \beta (\Lambda_{\psi_k} \circ g)(p) } - e^{ \beta (\Lambda_\mu \circ g)(p) } \right| \,\mathrm{d}p \\
    &\leq \left( \int_B \left| e^{ \beta (\Lambda_{\psi_k} \circ g)(p) } - e^{ \beta (\Lambda_\mu \circ g)(p) } \right|^l J_g (p) \,\mathrm{d}p \right)^{\frac{1}{l}} \left( \int_B J_g(p)^{-\alpha} \,\mathrm{d}p \right)^{\frac{1}{1+\alpha}} \\
    &\lesssim \left( \int_{g(B)} \left| e^{ \beta \Lambda_{\psi_k} (q) } - e^{ \beta \Lambda_\mu (q) } \right|^l \,\mathrm{d}q \right)^{\frac{1}{l}}.
\end{align*}
Consequently, $\int_B \left| e^{ \beta (\Lambda_{\psi_k} \circ g)(p) } - e^{ \beta (\Lambda_\mu \circ g ) (p) } \right| \,\mathrm{d}p \to 0$ as $k \to \infty$ by Lemma \ref{lem:2limonzero}. 
\end{proof}

\begin{lem} \label{lem:theta2}
Let $\mu \in \cM$, let $K\geq 1$, and let $g \in \cQ (K)$. Let $ \Lambda $ denote the quasilogarithmic potential $ \Lambda_\mu \circ g $. For each $ k \in \mN $, define $ \Lambda_k = \Lambda_{\mu_k} \circ g $ with $ \mu_k : = \mu\big|_{B(k)} $. In these circumstances there exists $ \theta = \theta(K) > 0 $ such that for every $ 0 < \beta < \theta / \| \mu \| $ the function $ e^{\beta \Lambda} $ is locally integrable, and for every ball $ B \subset \mH $ we have
\[
	\int_B |e^{ \beta \Lambda_k (p) } - e^{ \beta \Lambda (p) } | \,\mathrm{d}p \to 0 \quad \text{as} \quad k \to \infty.
\]
\end{lem} 

\begin{proof}
Fix $p\in \mH$ such that $|\Lambda_{\mu}(p)| < \infty$ (by Lemma \ref{lem:finwmfin} we have almost every $p \in \mH$ to choose from). Now observe,
\begin{align*}
    |\Lambda_{\mu_k}(p) - \Lambda_{\mu}(p)| &= \left| \int \log \|q^{-1}p\|  \,\mathrm{d}\mu(q) - \int_{B(k)} \log \|q^{-1}p\| \,\mathrm{d}\mu(q) \right| \\
    &\leq \int_{\mH\setminus B(k)} |\log \|q^{-1}p\| | \, \mathrm{d}|\mu|(q) \\
    &\leq \int_{\mH\setminus B(k)} \log^+ \|q^{-1}p\|  + \log^+ \left(\frac{1}{\|q^{-1}p\|}\right) \, \mathrm{d}|\mu|(q),
\end{align*}
so that if $k$ is large enough,
\[
    |\Lambda_{\mu_k}(p) - \Lambda_{\mu}(p)| \leq \int_{\mH\setminus B(k)} \log^+ \|q^{-1}p\| \, \mathrm{d}|\mu|(q).
\]
We thus have
\[
    |\Lambda_{\mu_k}(p) - \Lambda_{\mu}(p)| \leq \int_{\mH\setminus B(k)} \left(1 + \log^+ \|p\| + \log^+ \|q\| \right) \, \mathrm{d}|\mu|(q),
\]
and since $\|\mu\|<\infty$ it must be that $\Lambda_{\mu_k}(p) \to \Lambda_{\mu}(p)$ as $k \to \infty$.

If $\Omega \subset \mH$ is such that $| \Omega | = 0$, then as $g$ is a quasiconformal mapping $|g^{-1} \Omega| = 0$ also. It follows from the above that $\Lambda_k \to \Lambda$ pointwise almost everywhere. 

Now proceed in a similar way to the proof of the last lemma. Let $\alpha > 0$ be as in \eqref{eq:apw}. Setting $\theta = \alpha/(1+\alpha)$, $l=1+1/\alpha$, and letting $0<\beta<\theta/\| \mu \|$, we have as before that $\beta l < 4/\| \mu \|$. As such Lemma \ref{lem:eblest} shows $\exp\left(\beta l \Lambda_{\mu}^+\right)$ to be locally integrable. The implication of the reverse H\"{o}lder inequality (that is, the defining property of $\alpha$) and the change of variable formula \eqref{eq:cvf} (used in a way very similar to the proof of the previous lemma) give that $\exp\left( \beta \Lambda_{\mu}^+ \circ g \right)$ is locally integrable. As in the proof of Corollary \ref{cor:eblest}, we have $\Lambda = \Lambda_\mu \circ g \leq \Lambda_{\mu}^+ \circ g$. It is also easy to deduce that for all $k$, 
\[
    \Lambda_k = \Lambda_{\mu_k} \circ g \leq \Lambda_{\mu_k}^+ \circ g \leq \Lambda_{\mu}^+ \circ g.
\]

In summary, the functions $|e^{ \beta \Lambda_k } - e^{ \beta \Lambda } |$ converge to zero pointwise almost everywhere as $k\to \infty$, and are bounded above by the locally integrable function $2\exp\left( \beta \Lambda_{\mu}^+ \circ g \right)$. Thus the dominated convergence theorem applies to give 
\[
    0 = \lim_{k\to \infty} \int_B \left| e^{ \beta \Lambda_k (p) } - e^{ \beta \Lambda (p) } \right| \,\mathrm{d}p 
\]
for any given ball $B\subset \mH$ as desired.
\end{proof}

\section{Quasiconformal Flows on $\mH$} \label{sec:qcf}

The measurable Riemann mapping theorem guarantees a plentiful supply of quasiconformal mappings $ f : \mC \to \mC $. It is a consequence of this theorem that any quasiconformal mapping of the complex plane embeds as the time-$ s $ flow mapping of a suitably well-behaved vector field.

While quasiconformal mappings of the Heisenberg group satisfy a ``Beltrami system'' of equations, no similar results on the existence of solutions are known. We may, however, identify suitable conditions on a vector field $ v : \mH \to \text{T}\mH $ such that the flow is quasiconformal. Such conditions were first identified by Kor\'{a}nyi and Reimann in \cite{KRI} and \cite{KRII}. The results of \cite{KRI} are for reasonably smooth flows. In \cite{KRII} the main relevant result requires significantly less regularity, but demands that the vector field be compactly supported. See the introduction for more discussion.

Our task requires both low regularity and unbounded support. It is the purpose of the first part of this section to remove the assumption of compact support from Theorem H of \cite{KRII}. In its place we make stipulations on the growth of the vector field, then use a cut-off argument to reduce to the compactly supported case.

Remember that a quasiconformal mapping of $ \mH $ is almost everywhere contact. It is a theorem of Liebermann \cite{Liebermann} that in order for a vector field to generate contact flow it must be of the form
\begin{equation} \label{eq:contactv}
	v = v_{\phi} = -\tfrac{1}{4} (Y \phi) X + \tfrac{1}{4} (X \phi) Y + \phi T,
\end{equation}
for a function $ \phi : \mH \to \mR $. In this context $ \phi $ is called a contact-generating potential, or simply a potential. Whenever a function $\phi$ is in play and we write $ v_\phi $, we mean the above expression. As in the work of Kor\'{a}nyi and Reimann mentioned above, we will typically work at the level of the potential, deducing from its properties the properties of the flow. Indeed, Section \ref{sec:ap} is all about constructing a potential that matches our requirements.

On occasion we will need to discuss the component functions of a vector field $v$ as in \eqref{eq:contactv}. Perhaps the natural choice would be to define these with respect to the basis $ X , Y , T $ of $ \text{T} \mH $. Nevertheless, in order to be consistent with something that comes later we let $ v_1 , v_2 , v_3 $ be defined by
\[
	v = v_1 \partial_x + v_2 \partial_y + v_3 \partial_t.
\]
To be explicit, we have
\begin{equation}\label{eq:vcomp}
    v_1 = -\tfrac{1}{4} Y\phi, \quad v_2 = \tfrac{1}{4} X\phi, \quad \text{and} \quad v_3 = -\tfrac{1}{4} (2y) Y\phi - \tfrac{1}{4} (2x) X\phi + \phi.  
\end{equation}

The second part of this section is dedicated to proving a variational equation that links the Jacobian of the flow mapping to the horizontal divergence of the vector field. The results of this part follow a similar sequence of results in \cite{BHSII}. If $ \phi $ is a potential the horizontal divergence of $ v_{ \phi } $ is given by $ T\phi $:
\begin{equation}\label{eq:mif}
    \text{div}_H (v_{\phi}) = \tr (D_H v_{\phi}) = X v_1 + Y v_2 = -\tfrac{1}{4}(XY - YX)\phi = T\phi.
\end{equation}
A large part of the work of Section \ref{sec:ap} goes toward designing a $ \phi $ with $ T $ derivative resembling a given logarithmic potential. The variational equation is then the key stepping stone linking Jacobian to the weight determined by the logarithmic potential.

The notation $D_H v_{\phi}$ in \eqref{eq:mif} is currently an abuse that deserves clarification (up till now $D_H f$ is defined only for $f$ a quasiconformal mapping, and that definition was given in Section \ref{sec:qc} in terms of the Pansu-derivative). From here on, if function $F$ defined on open $U\subset \mH$ takes values in a space that is $\mR^3$ as a set and $ F_1 , F_2 \in HW_{\text{loc}}^{1} $, then
\[  
	D_H F := \begin{pmatrix} 
		    XF_1 & YF_1 \\
		    XF_2 & YF_2
		    \end{pmatrix}.
\]
Note we do not require such $ F $ to be contact (not even weakly so) in order to discuss $ D_H F $. At this level (the Sobolev level) of regularity we refer to $ D_H F $ as \emph{the} formal horizontal differential (of $F$) if the choice of almost everywhere defined functions is unimportant, and as \emph{a} formal horizontal differential if we are talking about a particular representative.

Given its importance to this paper we end the introduction to this section by recording that
\[
    ZZ = \tfrac{1}{4}\left[(XX - YY) - i (XY + YX) \right]
\]
(see \eqref{eq:wderiv} for the definition of $Z$).

\subsection{Vector Fields with Unbounded Support} 

The following is a special case of Theorem H of \cite{KRII}. 

\begin{thm}[Kor\'{a}nyi, Reimann] \label{thm:kr}
Suppose $ \phi \in HC^1 $ is compactly supported and that $ ZZ \phi \in L^{\infty} $. Let $ 0 \leq c < \infty $ be such that
\[
	\sqrt{2} \| ZZ\phi \|_{\infty} \leq c.
\]
Then for each $ p \in \mH $ the flow equation for $ v_{\phi} $ at $p$,
\[
	\gamma'(s) = v_{\phi} ( \gamma(s) ), \quad \gamma(0) = p,
\]
has exactly one solution $ \gamma_p : \mR \to \mH $. Furthermore, for $ s \geq 0 $ the time-$s$ flow mapping $f_s : \mH \to \mH$ defined by
\[
 	f_s (p) = \gamma_p (s)
\]
is a $K$-quasiconformal homeomorphism with $ K + K^{-1} \leq 2 e^{cs} $.
\end{thm}

In the statement of the theorem, $ZZ\phi$ is a distributional (or weak) derivative.

\begin{rem}
Theorem H of \cite{KRII} is stated and proved for the $n$-th Heisenberg group, $1 \leq n < \infty$. The reader may like to consult also Theorem 6 of \cite{KRI} which is for $\mH = \mH^1$ only. The downside (for us) of this latter theorem is that it assumes the vector field is $C^2$-smooth (and so the flow mappings are $C^2$-smooth also). This is not true of the vector fields we construct in Section \ref{sec:ap}.  
\end{rem}

We intend to adapt Theorem \ref{thm:kr} by identifying suitable means of removing the assumption of compact support. First we have a smaller (but still important) improvement to make. The proof of Theorem \ref{thm:kr} makes use of the square  (or Frobenius) norm on $ D_H f_s $, which leads to the form of the bound on $ K $. Unfortunately, this bound is not suitable for our later arguments as the factor $2$ accumulates problematically on taking repeated compositions. This can be avoided if we rework part of the proof in the smooth case, using the operator norm in place of the square norm. We need only the smooth case since it is this that feeds into the proof of Theorem \ref{thm:kr} in an approximation argument. We thank Jeremy Tyson for improving the proof of the following lemma.
\begin{lem} \label{lem:dil}
Suppose $ \phi \in C_{0}^{\infty} $ and $ \sqrt{2} \| ZZ\phi \|_{\infty} \leq c $ for some $ 0 \leq c < \infty $. Then $ v_{\phi} $ generates a smooth flow of homeomorphisms and each time-$ s $ flow mapping, $ 0 \leq s < \infty $, is $ K $-quasiconformal with $ K \leq e^{ c s } $.  
\end{lem}
\begin{proof}
Taking $ \phi \in  C_{0}^{\infty} $ is already enough for existence and uniqueness of solutions to the flow equation for $ v_{\phi} $. The time-$ s $ flow mappings will be well-defined $ C^{\infty} $-smooth homeomorphisms of $ \mH $.

In the present context, $ ZZ \phi $ should be considered the Heisenberg equivalent of what in the Euclidean case is sometimes called the Ahlfors conformal strain of the vector field. For $M\in M_2(\mR)$, let 
\[
	\cS ( M ) := \tfrac 1 2 ( M + M^T ) - \tfrac 1 2 ( \tr M) I_2
\]
(this is the symmetric, trace-free part of $ M $). Writing $ v = v_{\phi} $, let
\[
	\cS_H v := \frac 1 2 	\begin{pmatrix} 
						Xv_1-Yv_2 & Xv_2+Yv_1 \\
						Xv_2+Yv_1 & Yv_2-Xv_1
						\end{pmatrix} = \cS (D_H v).
\]
For $M\in M_2(\mR)$, let $ \| M  \| := \sqrt{\tr [ M M^T ] } $ (called the square norm of $M$). Note that (or see \cite[p.~333]{KRI}), 
\[
	\sqrt 2 | ZZ\phi | = 2 \| \cS_H v \|.
\]
As $ | \cS_H v | \leq \| \cS_H v \| $, our assumed bound on $ | ZZ\phi | $ translates to
\[
	2 \| \cS_H v \|_{\infty} \leq c	
\]
where (at the risk of confusion) we write $ \| \cS_H v \|_{\infty} $ for $\sup_{p \in \mH} | \cS_H v (p) |$. For $s\geq 0$ let $ f_s $ be the time-$ s $ flow mapping generated by $ v $. Fix $p \in \mH$. The integral formula for solutions to the flow equation and the contact equations \eqref{eq:contact1} and \eqref{eq:contact2} lead to
\begin{equation} \label{eq:diffint}
	( D_H f_s (p) )' = D_H v ( f_s (p) ) D_H f_s (p).
\end{equation}
For notational convenience, let $ A = A (s) := D_H f_s (p) $ and for all $q\in \mH$ let $ B = B(q) := D_H v (q) $. Then \eqref{eq:diffint} becomes
\[
	A' = B ( f_s ) A.
\]
For all $s \geq 0$ the matrix $A(s)$ is invertible, and we find 
\begin{equation}\label{eq:sbfs}
	\cS (B ( f_s ) ) = \tfrac  1 2 A' A^{-1} + \tfrac 1 2 ( A^{-1} )^T ( A' )^T - \tfrac 1 2 \tr( A' A^{-1} ) I_2.
\end{equation}

Our chosen $p$ was arbitrary, hence if the smooth quasiconformal mapping $ f_s $ satisfies 
\begin{equation} \label{eq:analdil}
	\frac{ | D_H f_s (p) |^2 }{ \det D_H f_s (p) } = \frac{ | A |^2 }{ \det A} \leq K 
\end{equation}
it is $K$-quasiconformal -- this is Theorem 4 of \cite{KRI}. Consequently, we need only show
\[
	\frac{ | A |^2 }{ \det A} \leq \exp \left( c s  \right).
\]
To this end, recall that $ | A |^2 $ is equal to the larger eigenvalue $ \lambda = \lambda(s) $ of the matrix $ A^T A $. For each $ s \geq 0 $, there is a unit eigenvector $ u = u(s) $ for the eigenvalue $ \lambda $ such that
\[
	\lambda = \langle A^T A u, u \rangle = |A u|^2.
\]
Differentiating with respect to $ s $ gives
\begin{equation} \label{eq:eigdiff}
	\lambda' = 2 \langle A' u, A u \rangle + 2 \langle A u', A u \rangle.
\end{equation}
As $ | u |^2 = 1 $ it must be that $ u' $ and $ u $ are orthogonal:
\[
	0 = (| u |^2) ' = \langle u , u \rangle' = 2 \langle u' , u \rangle. 
\]
It follows the second term on the right side of \eqref{eq:eigdiff} is zero, indeed
\[
	\langle A u', A u \rangle = \langle u', A^T A u \rangle = \langle u', \lambda u \rangle = \lambda \cdot 0.
\]

Using the standard formula (sometimes called Jacobi's formula) for $s \mapsto M(s) \in M_2 (\mR)$ such that each $M(s)$ is differentiable and invertible 
\[
	( \det M )' = ( \det M ) \tr( M' \, M^{-1} ),
\]
we have
\[
	\left( \log \frac{ | A |^2 }{ \det A} \right)' = \frac{ \lambda' }{ \lambda } - \frac{ (\det A)' }{ \det A } = \frac{ 2 \langle A' u , A u \rangle }{ | A u |^2 } - \tr ( A' A^{-1} ).
\]
Set $w = A u$. Then
\begin{equation*}
	\begin{split}
		\left( \log \frac{ | A |^2 }{ \det A} \right)' &= 2 \frac{ \langle A' A^{-1} w , w \rangle }{ | w |^2 } - \tr ( A' A^{-1} ) \\
		&= \left\langle \frac{ w }{ | w | } , \frac{ A' A^{-1} w }{ | w | } + \frac{ ( A^{-1} )^T ( A' )^T w }{ | w | } - \tr ( A' A^{-1} ) \frac{ w }{ | w | } \right\rangle \\
		&\leq \frac{ | A' A^{-1} w + ( A^{-1} )^T ( A' )^T w -  \tr ( A' A^{-1} ) w | }{ | w | } \\
		&\leq | A' A^{-1} + ( A^{-1} )^T ( A' )^T - ( \tr A' A^{-1} ) I_2 | \\
		&= 2 | \cS ( B ( f_s ) ) |
	\end{split}
\end{equation*}
by \eqref{eq:sbfs}. Observe that $ A(0) = I_2 $, $ \cS (B( f_s )) = \cS_H v ( f_s ) $, and for all $s\geq 0$, $ | \cS_H v ( f_s )| \leq \| \cS_H v \|_\infty \leq c/2 $. Thus
\[
    \frac{|A|^2}{\det A} \leq \exp\left( 2 s \sup_{\sigma \in [0,s]} | \cS_H v ( f_\sigma )| \right) \leq e^{cs}
\]
as desired.
\end{proof}

A careful check of the proof of Theorem \ref{thm:kr} in \cite{KRII} shows that the quasiconformal mappings it promises have dilatation bounded in the same manner as the smooth mappings of Lemma \ref{lem:dil}. We are now ready to formulate a new version of Theorem \ref{thm:kr} with the assumption of compact support replaced by some natural growth conditions. The proof uses some ideas from Reimann's work in the Euclidean setting \cite{ReimannODE}, especially Propositions 4 and 12 of that paper.

\begin{pro} \label{pro:flow}
Suppose $ \phi \in HC^1 $ and that $ ZZ \phi \in L^{\infty} $. Let $ 0 \leq c < \infty $ be such that
\begin{equation} \label{eq:zzbnd}
	\sqrt{2} \| ZZ \phi \|_{\infty} \leq c.
\end{equation}
Further suppose that
\begin{equation} \label{eq:growth1}
	|\phi(p)| 	\lesssim 1 + \|p\|^{2} \log \|p\| 
\end{equation}
and
\begin{equation} \label{eq:growth2}
	|Z \phi (p)| \lesssim 1 + \|p\| \log\|p\|.
\end{equation}
Then for all $ p \in \mH $, the flow equation for $ v_{\phi} $ at $p$,
\[
	\gamma'(s) = v_\phi ( \gamma(s) ), \quad \gamma(0) = p,
\]
has a unique solution that exists for all time. We denote this solution $ \gamma_p : \mR \to \mH $. For $ s \geq 0 $, the time-$s$ flow mapping $f_s : \mH \to \mH$ defined by
\[
 	f_s (p) = \gamma_p (s)
\]
is a $K$-quasiconformal homeomorphism with $ K \leq e^{cs} $.
\end{pro}

We separate the proof into two parts. The first part is contained in the following lemma of independent interest.

\begin{lem}\label{lem:bndsolns}
Suppose $\phi \in HC^1$ is such that \eqref{eq:growth1} and \eqref{eq:growth2} hold. Then for all $p\in\mH$, any solution to the flow equation for $v_{\phi}$ at $p$ remains a bounded distance from the origin on any interval of existence $(-s_0, s_0)$ with $0 \leq s_0 < \infty$.
\end{lem}

\begin{proof}
Recall that
\[
	v_{\phi} = -\tfrac{1}{4}(Y\phi) X + \tfrac{1}{4}(X\phi) Y + \phi T
\]
and the components $v_i$ are determined by $v_{\phi} = v_1 \partial_x + v_2 \partial_y + v_3 \partial_t$ as in \eqref{eq:vcomp}. Write $v = v_{\phi}$.

Let $\|(z,t)\|_{a} := |z|+|t|^{\frac{1}{2}}$. This is a homogeneous norm comparable to $\| \cdot \|$. Our assumptions imply that for $i=1,2$,
\begin{equation}\label{eq:vi}
	|v_i(p)| \lesssim 1 + \|p\|_a \log\|p\|_a
\end{equation}
and
\begin{equation}\label{eq:v3}
	|v_3(p)| \lesssim 1 + \|p\|_a^2 \log\|p\|_a . 
\end{equation}

Fix $u \in \mH$ and let $\gamma = (\gamma_1 , \gamma_2 , \gamma_3) : (-s_0,s_0) \to \mH$ be a solution to the flow equation for $v$ at $u$ with $0 < s_0 < \infty$. Let $s \in (-s_0,s_0)$. Since $\widetilde{\gamma} : [0,s_0) \to \mH$ defined by $\widetilde{\gamma}(s) = \gamma(-s)$ is a solution to $\eta' = -v(\eta)$ with $\eta(0) = u$, and $-v$ satisfies the same estimates as $v$, we may in fact assume $s \in [0,s_0)$.  

Define $\gamma_I = \gamma_1 + i \gamma_2$. Using \eqref{eq:vi} and \eqref{eq:v3} we may choose $C_1, C_2 >0$ such that
\[
	|\gamma_I (s)| \leq C_1 + C_2 \int_0^s \|\gamma(\sigma)\|_a \log \|\gamma(\sigma)\|_a \,\mathrm{d}\sigma
\]
and
\[
	|\gamma_3 (s)| \leq (C_1+1)^2 + C_2 \int_0^s \|\gamma(\sigma)\|_a^2 \log \|\gamma(\sigma)\|_a \,\mathrm{d}\sigma.
\]
Note that $C_1$ depends on $u$ and $s_0$.
	
Define
\[
	\lambda_I (s) = C_1 + C_2 \int_0^s (1 + \|\gamma(\sigma)\|_a) \log( 2 + \|\gamma(\sigma)\|_a ) \,\mathrm{d}\sigma
\]
and
\[
	\lambda_3 (s) = (C_1+1)^2 + C_2 \int_0^s (1+\|\gamma(\sigma)\|_a)^2 \log (2+ \|\gamma(\sigma)\|_a) \,\mathrm{d}\sigma.
\]
Then
\[
	\lambda_I '(s) = C_2 (1 + \|\gamma(s)\|_a) \log( 2 + \|\gamma(s)\|_a )
\]
and
\[
	\lambda_3 '(s) = C_2 (1+\|\gamma(s)\|_a)^2 \log (2+ \|\gamma(s)\|_a).
\]
In particular, $\lambda_I$ is $C^1$-smooth and strictly increasing on $[0,s_0)$ with $C^1$-smooth inverse. Define
\[
	w = \lambda_3 \circ \lambda_{I}^{-1}
\]
so that $\lambda_3 (s) = (w \circ \lambda_I)(s)$. It follows
\[
	w'(\lambda_I) = \frac{\lambda_{3}'}{\lambda_{I}'} = 1 + |\gamma_I| + |\gamma_3|^{\frac{1}{2}} \leq 1+ \lambda_I + \sqrt{w(\lambda_I)}.
\]
In these circumstances, $w$ is dominated by any solution of the equation
\[
	g'(\lambda) = 1 + \lambda + \sqrt{g(\lambda)}, \quad g(C_1) = (C_1+1)^2.
\]
On inspection we see that $g(\lambda) = (\lambda+1)^2$ is a solution for $\lambda + 1 \geq 0$. Consequently, 
\[
	w(\lambda_I(s)) \leq (\lambda_I (s) + 1)^2
\]
or to put it another way
\begin{equation}\label{eq:l3bnd}
	\lambda_3 (s) \leq (\lambda_I (s)+1)^2.
\end{equation}
This allows us to bound $\lambda_I$ using a standard Gr\"{o}nwall-type argument. Observe,
\begin{align*}
	\lambda_I' &= C_2\left(1+|\gamma_I|+|\gamma_3|^{\frac{1}{2}}\right)\log\left(2 +|\gamma_I| + |\gamma_3|^{\frac{1}{2}}\right) \\
	&\leq C_2\left(1+\lambda_I + \lambda_3^{\frac{1}{2}}\right)\log\left(2 +\lambda_I + \lambda_3^{\frac{1}{2}}\right) \\
	&\leq C_2(2+2\lambda_I)\log(3 + 2\lambda_I).
\end{align*}
From this we have
\begin{equation}\label{eq:lIbnd}
	\log(3+2\lambda_I(s)) \leq \log(3+2C_1)e^{2C_2 s}.
\end{equation}
With this bound on $\lambda_I$ in place, a bound on $\lambda_3$ is immediate from \eqref{eq:l3bnd} and together these give the desired bound on $\|\gamma\|_a$. 
\end{proof}

\begin{proof}[Proof of Proposition \ref{pro:flow}]
Write $ v = v_{\phi} $. Let $ u \in \mH $ be given. As $ v $ is continuous, a solution to the flow equation for $ v $ at $ u $ exists on some interval $ ( -s_0 , s_0 ) $ with $ s_0 > 0 $. We denote this solution $\gamma$. By Lemma \ref{lem:bndsolns} there exists finite $ R > 0 $ such that $ \gamma ( s ) \in B(R) $ for all $ s \in ( -s_0 , s_0) $ (and this is true for any other solution at $ u $ when restricted to this same interval).

To complete the proof we use a cut-off argument. The auxiliary functions that allow us to smoothly truncate our vector field are defined as follows. For $ e \leq l < \infty $, let $ \widetilde{G}_l : [e,\infty) \to \mR $ be given by
\[
	\widetilde{G}_l ( r ) = 1 - l^{-1} ( \log \log r - \log \log l ).
\]
Now take a smooth function $P:\mR\to \mR$ satisfying $ P(0) = 0, \, P(1) = 1, \, P'(0) = P'(1) = P''(0) = P''(1) = 0 $, and $ 0 \leq P(\sigma) \leq 1 $ when $ 0 \leq \sigma \leq 1 $ (e.g. $ P(z) = 6 z^5 - 15 z^4 + 10 z^3 $), and use this to form
\[
	G_l (r) = \left \{ 
   			\begin{array} {l l}
     		1 & \quad \text{if} \quad 0 \leq r \leq l, \\
     		P(\widetilde{G}_{l}(r)) & \quad \text{if} \quad l \leq r \leq l', \\
     		0 & \quad \text{if} \quad l' \leq r .
   			\end{array} \right.
\]
Here $ l' $ is chosen to be the smallest number such that $ \widetilde{G}_l \left( l' \right) = 0 $ (to be exact, choose $l'$ so that $ \log l' = e^l \log l $). The function $G_l$ is $C^2$-smooth, decreasing from $1$ to $0$, and with the following bounds on its derivatives: for all $r \geq 0$,
\begin{equation}\label{eq:gderiv}
	| G_l ' ( r ) | \lesssim \left( \frac{1}{l} \right)\frac{ 1 }{ r \log r } \quad \text{and} \quad | G_l '' ( r ) | \lesssim \left(\frac{1}{l}\right)\frac{ 1 }{ r^2 \log r }.
\end{equation}

For $ l \geq e$ we form the truncated potential
\[
	\phi_l (p) = G_l \left( \| p \|^4 \right) \phi(p).
\]
Each $ \phi_l $ has continuous horizontal derivatives $ X \phi_l $ and $ \, Y \phi_l $, and is compactly supported. Define $N : \mH \to [0,\infty)$ by $ N(p) = \| p \|^4 $. For all $l$, the weak derivative $ ZZ \phi_l $ exists and 
\[
	ZZ \phi_l = ZZ ( G_l \circ N ) \phi +2 Z ( G_l \circ N ) Z \phi + ( G_l \circ N ) ZZ \phi .
\]
The first computations of the appendix applied to this expression give
\[
    	\sqrt{2} \| ZZ \phi_l \|_{\infty} \leq C \sup_{ l \leq \| p \|^4 \leq l' } \left[ \| p \|^2 \left| G_l '' \left( \| p \|^4 \right) \right| | \phi(p) | + \| p \|^3 \left| G_l ' \left( \| p \|^4 \right) \right| | Z \phi(p) | \right] + c
\]
where $c$ is as in \eqref{eq:zzbnd}. Since $l \geq e > 1$ this immediately gives
\[
    	\sqrt{2} \| ZZ \phi_l \|_{\infty} \leq C \sup_{ l \leq \| p \|^4 \leq l' } \left[ \| p \|^6 \left| G_l '' \left( \| p \|^4 \right) \right| | \phi(p) | + \| p \|^3 \left| G_l ' \left( \| p \|^4 \right) \right| | Z \phi(p) | \right] + c
\]
which is presently more useful. The estimates \eqref{eq:gderiv} can now be used to find
\[
    \sqrt{2} \| ZZ \phi_l \|_{\infty} \leq C\left(\frac{1}{l}\right) \sup_{ l \leq \| p \|^4 \leq l' } \left[ \frac {| \phi(p) | }{ \| p \|^2 \log \| p \| } + \frac{| Z \phi(p) | }{ \| p \| \log \| p \| } \right] + c.
\]
Observe that for $\|p\|\geq e$, $1 + \|p\| \log \|p\| \lesssim \|p\| \log \|p\|$ and $1 + \|p\|^2 \log \|p\| \lesssim \|p\|^2 \log \|p\|$. Consequently, we may now make use of our assumptions \eqref{eq:growth1} and \eqref{eq:growth2} to arrive at
\begin{equation} \label{eq:zzlbnd}
    \sqrt{2} \| ZZ \phi_l \|_{\infty} \leq \frac{C}{l} + c
\end{equation}
when $l \geq e^4$.

Making a choice of $ l \geq e^4 $ so that $ ( G_l \circ N ) \equiv 1 $ on $ B( R ) $ (recall $R$ is such that $ \gamma (s) \in B(R) $ for all $s \in ( -s_0, s_0 ) $), we have that $ v $ and
\[
	v_l := -\tfrac 1 4 (Y \phi_l ) X + \tfrac 1 4 (X \phi_l ) Y + \phi_l T
\]
coincide on $ B(R) $. It is part of Theorem \ref{thm:kr} that the flow equation for $ v_l $ at $ u $ has a unique solution. It follows that $ \gamma $ is the unique solution on the interval $ (-s_0 , s_0) $ to the flow equation for $ v $ at $ u $.

As $ u $ was arbitrary, we have shown that at all $ p \in \mH $ there is a unique solution to the flow equation at $ p $. Moreover, this solution remains bounded on any finite-time interval of existence. It follows that such a unique solution may be continued unambiguously and therefore  will exist for all time. Consequently, we find that $ v $ has a well-defined flow of homeomorphisms $ f_s $ for all $ s \in \mR $.

It remains to show that the time-$ s $ flow mappings $ f_s $ with $ s \geq 0 $ are quasiconformal with the claimed bound on the dilatation. Let $ f_{s}^{l} $ denote the time-$ s $ flow mapping associated to $ v_l $. Using \eqref{eq:zzlbnd} along with Theorem \ref{thm:kr} and Lemma \ref{lem:dil}, we find that $ f_s^l $ is $K$-quasiconformal with $ K \leq e^{ \left( C l^{-1} + c \right) s } $.

Fix $s \geq 0$. Let $ D > 0 $ be given and choose $ D' > 0 $ such that $ f_{s} B(D) \subset B(D') $. Choosing $ l $ so that $ v_l \equiv v $ on $ B(D') $, it follows that the restriction $ f_s \big|_{B(D)} $ is quasiconformal with
\[
	K\left(f_s \big|_{B(D)}\right) \leq e^{ \left( C l^{-1} + c \right) s }.
\]
This is true for all $l$ sufficiently large and the left-hand side does not depend on $l$. We let $ l \to \infty $ to find $ K\left(f_s \big|_{B(D)}\right) \leq e^{ c s } $. As this procedure works for any $ D > 0$, we must have that $ f_s $ is quasiconformal with $ K(f_s) \leq e^{cs} $ as required.
\end{proof}

\subsection{A Variational Equation} \label{subsec:vareq}

For the remainder of the section we fix $ \phi: \mH \to \mR $ satisfying the hypotheses of Proposition \ref{pro:flow}. This includes that $0 \leq c < \infty$ is such that $\sqrt{2}\|ZZ\phi\|_{\infty} \leq c$. We make the additional assumption $ X\phi , Y\phi \in HW_{\text{loc}}^{1,r} $ for all $ 1 \leq r < \infty $. Let $v = v_{\phi}$.

Let $ D_H v $ denote a formal horizontal differential of $ v $. Since
\[
    D_H v = \frac{1}{4}\begin{pmatrix} -XY \phi & -YY\phi \\ XX\phi & YX\phi  \end{pmatrix},
\]
and $|D_H v| \lesssim \max_{i,j = 1,2}|(D_H v)_{i,j}|$, our local integrability assumption on the weak second horizontal derivatives of $ \phi $ is equivalent to $ D_H v $ having the same local integrability (with respect to the operator norm).

In Section \ref{sec:qc} we mentioned that a quasiconformal mapping is $\cP$-differentiable almost everywhere. We also have horizontal Sobolev regularity: if $f$ is a quasiconformal mapping then $ f \in HW_{\text{loc}}^{1,4} $. Moreover, \eqref{eq:jdef} and the reverse H\"{o}lder inequality \eqref{eq:rhi} imply there exists $\epsilon > 0$ such that $f \in HW_{\text{loc}}^{1,4+\epsilon}$. The almost everywhere defined classical horizontal differential determined by the $\cP$-derivative (as in \eqref{eq:dhf}) may serve as representative of the formal horizontal differential. Reserving the notation $ f_s $ for the time-$ s $ flow mappings of $ v $, we let $ D_H f_s $ indicate this choice.

The nature of the argument that follows is designed precisely so that our end goal (Proposition \ref{pro:link}) holds for \textit{all} values $ s \in [0,1]$. It is likely a similar statement could be proved without as much preparation if we were aiming only for almost every $ s \in [0,1] $.       
\begin{lem} \label{lem:jc}
The mapping $( p , s ) \mapsto f_s (p) $ is continuous on $\mH \times [0,1]$. Moreover, for each ball $ B \subset \mH $,
\[
	| f_s (B) | \gtrsim 1.
\]
The implied constant is dependent on $\|ZZ\phi\|_{\infty}$ and $B$, but not on $s$.
\end{lem}
\begin{proof}
We will prove the second statement first. By assumption, there exists $c\geq 0$ such that each $ f_s $ is $ e^{ c s } $-quasiconformal. With $ s \in [0,1] $, each $ f_s $ is $ K $-quasiconformal with $ K = e^{ c } $ independent of $ s $. Let $ p \in \mH $ and let $ R > 0 $. Fix a point $ q \in S(p,R) $. It follows from \eqref{eq:bdist} that
\[
	\left| f_s B(p,R) \right| \gtrsim d \left( f_s (p) , f_s (q) \right)^4 .
\]
Continuity of solutions to the flow equation implies continuity of $ s \mapsto d( f_s (p) , f_s (q) )^4 $. Since each $ f_s $ is injective, $d \left( f_s (p) , f_s (q) \right)^4$ has a positive minimum on $ [0,1] $ as desired.

As for the first statement, let $ s \in [0,1] $ and let $ (p_k,s_k) $ be a sequence in $ \mH \times [0,1] $ such that $ (p_k,s_k) \to (p,s) $. In particular, $ p_k \to p $ and we may assume that $ p_k \in B(p,1) $ for all $ k $. It follows from our bound on solutions to the flow equation (Lemma \ref{lem:bndsolns} -- in particular \eqref{eq:l3bnd} and \eqref{eq:lIbnd}) that there exists $ R' > 0 $ such that for all $ s_k $, 
\[
	f_{s_k} B\left(3(\|p\| + 1)+1\right) \subset B\left( R' \right).
\]
It follows from \eqref{eq:hold} that each $ f_{s_k} $ is H\"{o}lder continuous on $ B(p,1) $ with the coefficient (no name required) and exponent $\alpha$ independent of $ k $. Use first the triangle inequality then the observation just made to find 
\begin{align*}
	d(f_{s_k}(p_k),f_s (p)) 	&\leq d(f_{s_k}(p_k),f_{s_k}(p)) + d(f_{s_k}(p),f_s(p)) \\
								&\lesssim d(p_k,p)^{\alpha} + d(f_{s_k}(p),f_s(p)).
\end{align*}
That $ ( p , s ) \mapsto f_s (p) $ is jointly continuous on $ \mH \times [0,1] $ is easily deduced from this last estimate.  
\end{proof} 

\begin{lem} \label{lem:intdhv}
Suppose $ B \subset \mH $ is a ball. Then the mappings
\[
	(p,s) \mapsto D_H v \left( f_s (p) \right) \quad \quad \text and \quad \quad (p,s) \mapsto D_H v \left( f_s (p) \right) D_H f_s (p)
\]
are measurable and integrable on $ B \times [0,1] $. 
\end{lem}

Integrability of a matrix valued function refers to integrability of the operator norm.

\begin{proof}
It was already observed in the comments with which Section \ref{subsec:vareq} begins that $ D_H v $ is measurable and locally integrable to the power $r$ for any $ 1 \leq r < \infty $.

Let $p\in\mH$ be a point at which $ D_H f_s (p) $ exists in the classical sense. As such, it is the limit of a sequence of matrices, the entries of which are difference quotients. As $ f_s(p) $ is jointly continuous in $ s $ and $ p $, the difference quotients are measurable. It follows that $ D_H f_s $ is measurable. Furthermore, since each $ f_s $ preserves sets of measure zero, $ (D_H v) \circ f_s $ is measurable also. 

The integrability required by the claim will follow if we can show that for each $ s \in [0,1] $, the $L^1$ norm of either function over an arbitrary ball is bounded above by a constant independent of $ s $.

Fix a ball $ B \subset \mH $. Note (as in Lemma \ref{lem:jc}) for all $ s \in [0,1] $, $ f_s $ is $ K $-quasiconformal with $ K $ independent of $ s $. Consequently, by \eqref{eq:apw} and Lemma \ref{lem:jc} there exists $ \alpha > 0 $ such that for all $ s \in [0,1] $,
\begin{equation} \label{eq:apapp}
	\int_B J_{ f_s }^{ -\alpha } \lesssim 1
\end{equation}
independently of $ s $.
 
Now let $ r = 1 + 1 / \alpha $. Again by (or as in the proof of) Lemma \ref{lem:jc}, there exists a ball $ B' $ such that $ f_s B \subset B' $ for all $ s \in [0,1] $. Using H\"{o}lder's inequality for the first estimate, \eqref{eq:apapp} and \eqref{eq:cvf} for the second, and our assumption that $|D_H v|\in L_{\text{loc}}^r$ for the third, we find
\begin{align*}
	\int_B | D_H v ( f_s ) |
	&\leq \left( \int_B | D_H v ( f_s ) |^r J_{f_s} \right)^{1/r} \left( \int_B J_{f_s}^{-\alpha} \right)^{ 1 / ( 1 + \alpha ) } \\
	&\lesssim \left( \int_{B'} | D_H v |^r \right)^{ 1/r } \lesssim 1, 
\end{align*}
where the implied constants depend on $ B $ but are independent of $ s $.

Similarly,
\begin{align*}
	\left( \int_B | D_H v( f_s ) D_H f_s | \right)^4
	&\lesssim \int_B | D_H v( f_s ) D_H f_s |^4 \\
	&\lesssim \int_B | D_H v( f_s ) |^4 J_{f_s} \\
	&\lesssim \int_{B'} | D_H v |^4 \lesssim 1,
\end{align*}
where we used \eqref{eq:jdef} for the second estimate. Again, the implied constants do not depend on $s$.
\end{proof}

The following can be found on page 46 of \cite{KRII}.

\begin{lem} \label{lem:dhcomp}
For all $ s \in [0,1] $, $ v \circ f_s \in HW_{\textnormal{loc}}^{1,1}$ with
\[
	D_H ( v \circ f_s ) := ((D_H v) \circ f_s ) D_H f_s
\]
a formal horizontal differential.
\end{lem}

The next lemma gives an alternative (to the one we have been using) representative of the formal horizontal differential of $ f_s $. It is formally identical to differentiating solutions to the flow equation in the smooth case.

\begin{lem} \label{lem:altrep}
For all $ s \in [0,1] $, the matrix-valued function $ F(\cdot,s) $ defined almost everywhere on $\mH$ by
\[
	F(p,s) = I_2 + \int_0^s D_H v( f_{\sigma}(p) ) D_H f_{\sigma}(p) \, \mathrm d \sigma.
\]
is a formal horizontal differential of $ f_s $.
\end{lem}

\begin{proof}
At almost every $ p \in \mH $ we have
\[
	F(p,s) = I_2 + \int_0^s D_H ( v \circ f_{\sigma} )(p) \, \mathrm d \sigma
\]
by Lemma \ref{lem:dhcomp}. We need to show the components of $ F( \cdot , s ) $ are weak horizontal derivatives as claimed. To this end, let $ \xi \in C_0^{\infty} $. By Lemma \ref{lem:intdhv}, $ (p,s)\mapsto ( D_H (v \circ f_s) (p) )_{i,j} \xi (p) \in L^{1}(\mH \times [0,1]) $ for each choice of $ i,j = 1,2 $. This allows application of Fubini's theorem. For example, the $ (1,1) $-component of $ F $ satisfies
\begin{align*}
	\int F_{1,1}(p,s) \xi(p) \, \mathrm{d} p
	&= \int \xi + \int \int_{0}^{s} X( v \circ f_{\sigma} )_{1}(p) \xi(p) \, \mathrm{d} \sigma \, \mathrm{d} p	\\
	&= \int \xi - \int_{0}^{s} \int( v \circ f_{\sigma} )_{1}(p) X \xi(p) \, \mathrm{d} p \, \mathrm{d} \sigma \\
	&= \int \xi - \int \int_{0}^{s} \frac{ \mathrm{d} }{ \mathrm{d} \sigma }( f_{\sigma} )_{1}(p) X \xi(p) \, \mathrm{d} \sigma \, \mathrm{d} p \\
	&= \int \xi - \int (f_{s})_{1}(p) X \xi(p) \, \mathrm{d}p + \int x X \xi(p) \, \mathrm{d} p \\
	&= - \int (f_{s})_{1}(p) X \xi(p) \, \mathrm{d} p.
\end{align*}
The other components are similar.
\end{proof}

Let $ F $ be as in Lemma \ref{lem:altrep}. Standard product measure arguments imply that $  F(p,s) = D_H f_s (p) $ at almost every $(p,s) \in \mH \times [0,1] $. Consequently, for almost every $ p \in \mH $, $ D_H f_s (p) = F(p,s) $ at almost every $ s \in [0,1] $. We will use this later in conjunction with the next lemma (taken unchanged from \cite{BHSII}).

\begin{lem}[Bonk, Heinonen, Saksman] \label{lem:dettr}
Let $ F , G : [0,1] \to M_n ( \mR ) $ be matrix-valued functions. Suppose that $ F $ is continuous, $ G $ is integrable, and
\[
	F(s) = I_2 + \int_0^s G( \sigma ) F( \sigma ) \, \mathrm d \sigma
\]
for all $ s \in [0,1] $. Then
\[
	\det( F(s) ) = \exp \left( \int_0^s \tr ( G ( \sigma ) ) \, \mathrm d \sigma \right)
\]
for all $ s \in [0,1] $.
\end{lem}

We are now ready to assemble the previous string of results into our variational equation.

\begin{pro} \label{pro:link}
Let $\phi : \mH \to \mR$ satisfy the hypotheses of Proposition \ref{pro:flow}. Further assume that $ X\phi, Y\phi \in HW_{\textnormal{loc}}^{1,r} $ for all $1\leq r < \infty$. Then for all $ s \in [0,1] $ we have
\[
	\log J_{f_s} ( p ) = 2 \int_{0}^{s} T \phi ( f_{\sigma} (p) ) \, \mathrm{d} \sigma
\]
at almost every $p\in\mathbb{H}$.
\end{pro}

\begin{proof}
With the above in place, the proof goes through as in the Euclidean case. Let $ F $ be as in Lemma \ref{lem:altrep}. Let $ p_0 \in \mH $ be such that (i) $ s \mapsto D_H v( f_s ( p_0 ) ) $ is integrable on $ [0,1] $, (ii) $ s\mapsto D_H v ( f_s ( p_0 ) ) D_H f_s ( p_0 ) $  is integrable on $ [0,1] $, and (iii) $ F( p_0 ,s) = D_H f_s ( p_0 ) $ at almost every $ s \in [0,1] $ (we have almost every point of $\mH$ to choose from).

Now let $ G(s) := D_H v( f_s ( p_0 ) ) $ on $[0,1]$. By (i), $ G $ is integrable on $ [0,1] $. Let $ s \mapsto F(s) $ be defined by $ F(s) = F( p_0 ,s) $ on $[0,1]$. By (ii), $ F $ is continuous on $ [0,1] $. Furthermore, (iii) allows us to replace $ \sigma \to D_H f_{\sigma} ( p_0 ) $ with $ \sigma \mapsto F(\sigma) $ in the expression for $ F( p_0 ,s) $,
\[
	F(s) = F( p_0 ,s) = I_2 + \int_0^s G(\sigma) F(\sigma) \, \mathrm{d} \sigma.
\]
Consequently, $ G $ and $ F $ satisfy the hypotheses of Lemma \ref{lem:dettr} and so
\begin{align*}
2 \log \left[ \det F( p_0 ,s) \right] &= 2 \int_0^s \tr D_H v \left( f_{\sigma}( p_0 ) \right) \, \mathrm{d} \sigma \\
	&= - 2 \int_0^s \tfrac 1 4 \left( [X,Y] (\phi) \circ f_{\sigma} \right) ( p_0 ) \,\mathrm{d} \sigma \\
	&= 2 \int_0^s T \phi ( f_{\sigma} ( p_0 ) ) \, \mathrm{d} \sigma.
\end{align*}

Let $ \Omega \subset \mH $ be the set at which properties (i)-(iii) hold. Lemma \ref{lem:altrep} says that for all $s \in [0,1]$, $D_H f_s (p) = F(p,s)$ at almost every $p \in \Omega$. It follows for all $s \in [0,1]$,
\[
	\log J_{f_s} ( p ) = 2\log \left[ \det D_H f_s (p) \right] = 2 \int_0^s T \phi ( f_{\sigma} ( p ) ) \, \mathrm{d} \sigma
\]
at almost every $p \in \mH $.
\end{proof}

\section{Vector Fields with Prescribed Horizontal Divergence} \label{sec:ap}

Consider a logarithmic potential as given. This section demonstrates we can construct a contact-generating potential $\phi : \mH \to \mR$ (see the paragraph containing \eqref{eq:contactv} for terminology) for which all the following hold. First, $\phi$ meets the requirements of Proposition \ref{pro:flow} so that it generates a quasiconformal flow. Second, $X\phi , Y\phi \in HW_{\text{loc}}^{1,r}$ for all $ 1 \leq r < \infty $ so that the results of Section \ref{subsec:vareq} hold (in particular Proposition \ref{pro:link}). Third, the (formal) horizontal divergence of $ v_\phi $ approximates the logarithmic potential in a suitable way so that (by Proposition \ref{pro:link}) the Jacobian of the quasiconformal flow mapping approximates the logarithmic potential.

Constructing such a $\phi$ requires the most granular of our arguments and some of the computations deserve to be described as tedious. It will be convenient to represent the standard basis of the horizontal layer using the notation
\[
	X_1 (p) = X_p,\quad\quad X_2 (p) = Y_p.
\]
To avoid repetition, if we say something is true for $X_i$ then it is true independently of whether $i=1$ or $i=2$.

If $\Omega \subset \mH \times \overbrace{\mH \times \cdots \times \mH}^{n \text{ copies}} $, and $(q_1,\ldots,q_n)$ is such that there exists $p$ with $(p,q_1,\ldots,q_n)\in \Omega$ we define 
\[
    \Omega_{q_1,\ldots,q_n} = \{ p \in \mH \;|\; (p,q_1,\ldots,q_n)\in \Omega \}.
\]
If $F$ is a real valued function with domain $\Omega$ then $F_{q_1,\ldots,q_n} : \Omega_{q_1,\ldots,q_n} \to \mR$ is defined by $F_{q_1,\dots,q_n}(p) = F(p,q_1,\dots,q_n)$. Suppose $\Omega_{q_1,\ldots,q_n}$ is open for each $(q_1,\ldots,q_n)$. Then 
\begin{equation} \label{eq:derivconvention}
    X_i F (p,q_1,\dots,q_n) := X_i F_{q_1,\ldots,q_n} (p)
\end{equation}
whenever the derivative on the right-side of this expression exists. In other words, $ X_i F $ always refers to differentiation in the first coordinate. We will be consistent in our use of $ p $ for this first coordinate and continue the convention that $ p = (x,y,t) $. We say that $X_i F$ exists and is continuous (on $\Omega$) to mean that for all $(p,q_1,\ldots,q_n) \in \Omega$ the derivative $X_i F_{q_1,\ldots,q_n}(p)$ exists and $X_i F$ is (jointly) continuous on $\Omega$. Similarly for second horizontal derivatives.

We begin with some elementary results largely avoiding the proofs. The first is a mild extension of classical differentiation under the integral (we use a horizontal derivative and tailor the statement to our purpose). 

\begin{lem} \label{lem:hdui}
Let $ V \subset \mH \times \mH $ be open and let $ f : V \times \mH \to \mR $ be continuous. Assume for each $ (p,q) \in V $ that $ f(p,q,u) = 0 $ whenever $ u $ is outside $ B(p,d(p,q)/2) $. Let $ \mu $ be a measure on $ \mH $ absolutely continuous with respect to Lebesgue measure. Define $ F : V \to \mR $ by 
\[
	F(p,q) = \int f(p,q,u) \, \mathrm{d}\mu (u).
\]
Then $ F $ is continuous. Furthermore, if $ X_i f, \, X_j X_i f $ exist and are continuous on $ V \times \mH $ then $ X_i F, \, X_j X_i F $ exist and are continuous on $ V $ with
\begin{align*}
	X_i F(p,q) 		&= \int X_i f(p,q,u) \, \mathrm{d}\mu (u) \quad \text{and} \\
	X_j X_i F(p,q) 	&= \int X_j X_i f (p,q,u) \, \mathrm{d}\mu (u).	
\end{align*}   
\end{lem}

If $p_k \to p$ in $\mH$ then for fixed $q \in \mH$ we will have all $B(p_k,d(p_k,q)/2)$ contained in the same large ball for large enough $k$. This allows use of the dominated convergence theorem in the proof of the above lemma. The next lemma is very similar to the first.

\begin{lem} \label{lem:hdui2}
Suppose $ f : \mH \times \mH \to \mR $ is continuous and that $ X_i f $ exists and is continuous on $ \mH \times \mH $. Let $ \psi \in C_0^{\infty} $. For $p \in \mH$ let
\[
	F(p) := \int f(p,q) \psi(q) \, \mathrm{d}q.
\]
Then $ F \in HC^1 $ with
\[
	X_i F(p) = \int X_i f(p,q)\psi(q) \, \mathrm{d}q.
\]
\end{lem}

The preceding two lemmas rely on joint continuity (of both function and derivative) to allow differentiation under the integral in the classical sense. In the next lemma, we want to differentiate under the integral but only weakly so. We retain joint continuity of the function and swap joint continuity of the derivative for joint integrability of the weak derivative. This allows for a Fubini-type argument (which we omit).

\begin{lem} \label{lem:wdui}
Let $ f: \mH \times \mH \to \mR $ be continuous and such that for each $q\in \mH$, $f_q \in HW_{\text{loc}}^{1,1}$. Suppose for some choice of $X_i f_q$  the function $(p,q) \mapsto X_i f (p,q) := X_i f_q (p)$ is well-defined almost everywhere on $\mH \times \mH$ with $ X_i f \in L_{\text{loc}}^r ( \mH \times \mH ) $ for all $ 1 \leq r < \infty $. Let $\psi \in C_0^{\infty}$ and for $p \in \mH$ define
\[
		F(p) = \int f(p,q) \psi (q) \, \mathrm{d}q.
\]
Then $ F \in HW_{\text{loc}}^{1,r} $ for all $ 1 \leq r < \infty $ with 
\[
	X_i F(p) = \int X_i f(p,q) \psi (q) \, \mathrm{d}q
\]
an almost everywhere defined representative of $X_i F$.
\end{lem}

To make our goal more precise, we seek to approximate a given \textit{quasi}logarithmic potential, 
\[
	(\Lambda_\psi \circ g) (p) = -\int \log d(g(p),q) \psi(q)\,\mathrm{d}q
\]
for some $ g \in \cQ_0 (K) $ and $ \psi \in C_0^{\infty} $.

Let $K\geq 1$. Fix $ g \in \cQ_0 (K) $ and  $ \psi \in C_0^{\infty} $. Recall, the notation $\cQ_0 (K)$ means that $ g $ is a $ K $-quasiconformal mapping such that $ g(0) = 0 $ and $ \| g(q) \| = 1 $ for at least one $ q \in S(1) $. 

For ease of reference, we collect here the various layers of our construction. Let $E\subset \mH \times \mH$ be the diagonal, $E = \{ (p,q) \in \mH \times \mH \,:\, p=q \}$. For $ q \in \mH $, let $ L_{q} : \mH \to \mH $ be left translation by $q$, $ L_{q}(p) = qp$.

For $ u \in \mH$, define $ \Gamma_u : (\mH \times \mH)\setminus E \to \mH $ by
\begin{equation}\label{eq:gammadef}
    \Gamma_u (p,q) = \delta_{ d(p,q)^{-1} } ( L_{u^{-1}} p ).
\end{equation} 
Fix a function $ \xi_{0} \in C_{0}^{\infty} $ with $ 0 \leq \xi_{0} \leq 1 $, $ \xi_{0} \equiv 1 $ on $B(1/4) $, and $ \textnormal{support}(\xi_{0}) \subset B(1/2) $. Now define $ \lambda_g : \mH \times \mH \to [0,\infty) $ by
\begin{equation} \label{eq:lambdadef}
\lambda_{g}(p,q) = \left\{ 
\begin{array}{l l}
\left( \int \xi_{0}(\Gamma_{u}(p,q)) J_{g}(u) \, \mathrm{d}u \right)^{\frac{1}{4}} & \quad \text{if $ p\neq q $} \\
0 & \quad \text{if $ p = q $.}
\end{array} \right.
\end{equation}
As we shall see, $\lambda_g (p,q)$ approximates $d(g(p),g(q))$ and has the added benefit that derivatives land on the smooth part of the integrand.

Let $ U_{g} : = \{ (p,q) \in \mH \times \mH \, : \, p \neq g^{-1}(q) \} $. Since $ g^{-1} $ is continuous we have that $ U_g $ is open. Define $ \eta_g : U_g \to \mR $ by
\begin{equation} \label{eq:etadef}
\eta_{g}(p,q) = -\log\lambda_g \left( p , g^{-1}(q) \right).
\end{equation}
With the comment regarding the role of $\lambda_g$ in mind, notice the similarity between $\eta_g$ and the integrand of a quasilogarithmic potential. 

Let $ \widetilde{ \phi }_g : \mH \times \mH \to \mR $ be defined by
\begin{equation} \label{eq:tphidef}
\widetilde{ \phi }_g (p,q) = \left\{ 
\begin{array}{l l}
\eta_g (p,q)\left( g^{-1}(q)^{-1}p \right)_{3} & \quad \text{if $ p \neq g^{-1}(q) $}\\
0 & \quad \text{if $ p = g^{-1}(q) $.}
\end{array} \right.
\end{equation}
Here $ \left( g^{-1}(q)^{-1}p \right)_{3} $ is the third component of $ g^{-1}(q)^{-1}p $. If $q$ is zero then $ \left( g^{-1}(q)^{-1}p \right)_{3} = t$. This exhibits the (deliberate) similarity between $\widetilde{\phi}_g$ and \eqref{eq:protophi}.

We now bring $ \psi $ into the picture ($ \psi $ is to be thought of as the density of a measure associated with our given quasilogarithmic potential). Let $ R > 0 $ be such that $ \textnormal{support}( \psi ) \subset B(R) $. Define $ \phi_{g,\psi}^{1} : \mH \to \mR $ by
\begin{equation} \label{eq:phi1def}
\phi_{g,\psi}^{1}(p) = \int \widetilde{ \phi }_g (p,q) \psi(q) \, \mathrm{d}q.
\end{equation}

We remind the reader that we write $v_\phi$ for the vector field generated by potential $\phi$ as in \eqref{eq:contactv}. Letting 
\begin{equation} \label{eq:videf}
v_{ \phi_{g,\psi}^1 } = v_{1} \partial_{x} + v_{2} \partial_{y} + v_{3} \partial_{t}
\end{equation} 
determine $ v_{1} , v_{2} , v_{3} $, set
\begin{equation} \label{eq:phi2def}
\phi_{g,\psi}^2 (x,y,t) = c_{3} - 4 c_{1}y + 4 c_{2}x
\end{equation}
with
\[
( c_{1} , c_{2} , c_{3} ) := ( v_{1}(0) , v_{2}(0) , v_{3}(0) ).
\]
The vector field $v_{\phi_{g,\psi}^2}$ generates a flow of left-translations which are conformal mappings.

Lastly, define $ \phi_{g,\psi} : \mH \to \mR $ by
\begin{equation} \label{eq:phidef}
\phi_{g,\psi}(p) = \phi_{g,\psi}^1 (p) - \phi_{g,\psi}^2 (p).
\end{equation}
The role of $\phi_{g,\psi}^2$ is to ensure the flow mappings of $v_{\phi_{g,\psi}}$ preserve the origin. This useful normalizing effect can be thought of as a result of a translation at the level of the flow, or more simply because $\phi_{g,\psi}^2$ is chosen so that $v_{\phi_{g,\psi}}(0) = 0$ (as a quick computation shows). 

We now make the statements we will spend the remainder of the section working toward. The first establishes the conditions required by Propositions \ref{pro:flow} and \ref{pro:link}.

\begin{pro}\label{pro:growth}
Let $K\geq 1$ and let $R>0$. Let $ g \in \cQ_0 (K) $ and let $ \psi \in C_0^\infty $ with $ \textnormal{support} (\psi) \subset B(R) $. Define $ \phi = \phi_{g,\psi} $ as in \eqref{eq:phidef}. Then (i) $\phi \in HC^1$ with
\[
	|\phi(p)| 	\lesssim_{K,R,\|\psi\|_1} 1 + \|p\|^{2} \log \|p\| \quad \text{and} 
\]
\[
	|Z \phi (p)| \lesssim_{K,R,\|\psi\|_1} 1 + \|p\| \log\|p\|;
\]
(ii) $ZZ\phi\in L^\infty$ with absolute constants $A_1, A_2 > 0$ such that
\[
	\sqrt{2} \|ZZ\phi\|_{\infty} \leq A_2 \exp\left( A_1 K^{\frac{2}{3}} \right) \| \psi \|_1; 
\]	
and (iii) $X_i \phi \in HW_{\text{loc}}^{1,r}$ for all $1\leq r < \infty$.
\end{pro}

The following statement is an immediate consequence of Propositions \ref{pro:flow}, \ref{pro:link}, and \ref{pro:growth} (along with the observation made above that $v_{\phi_{g,\psi}}(0) = 0$).

\begin{cor} \label{cor:flowprops}
Let ($K$, $R$, $g$, $\psi$ and) $\phi$ be as in Proposition \ref{pro:growth}. Then $ v_\phi $ generates a well-defined flow of homeomorphisms. For $ 0 \leq s < \infty $, let $ h_s $ be the time-$s$ flow mapping of $ v_{\phi} $, and let $ C = A_2 \exp\left( A_1 K^{\frac{2}{3}} \right) $ with $A_1,A_2$ as in Proposition \ref{pro:growth}. Then for all $ s \in [0,\infty)$, $ h_{s}(0) = 0 $ and $ h_{s} $ is quasiconformal with $ K(h_s) \leq e^{C \|\psi\|_1 s} $. Lastly, for all $ s \in [0,1] $,
\begin{equation} \label{eq:vareqinpro}
	\log J_{h_s}(p) = 2 \int_0^s T \phi (h_\sigma (p)) \, \mathrm{d}\sigma
\end{equation}
at almost every $ p \in \mH $.
\end{cor}

The next result contains the important approximating property of the horizontal divergence of $v_{\phi_{g,\psi}}$. This can also be regarded as a splitting for quasilogarithmic potentials.

\begin{pro} \label{pro:vprops}
Let ($K$, $R$, $g$, $\psi$ and) $\phi$ be as in Proposition \ref{pro:growth}. Then  
\[
	\textnormal{div}_{H} v_{\phi} = \Lambda_{\psi} \circ g + \zeta
\]
with $ \zeta \in L^{\infty} $ such that $ \| \zeta \|_{\infty} \lesssim_{K,\| \psi \|_{1}} \, 1  $.
\end{pro}

Now begin the technicalities. The expressions $ \Gamma_u^k $ and $ L_{u^{-1}}^k $ will refer to the $ k^{\text{th}} $ component function of $ \Gamma_u $ and $ L_{u^{-1}} $ respectively, whereas $ d^k $ refers to the $ k^{\text{th}} $ power of the distance function. We suppress $p$ and $q$ from the notation when this cannot cause confusion. To illustrate these choices we record that for $p,q \in \mH$ with $p\neq q$,
\[
    \Gamma_u (p,q) = \left( \frac{(u^{-1}p)_1}{d(p,q)} , \frac{(u^{-1}p)_2}{d(p,q)} , \frac{(u^{-1}p)_3}{d^2 (p,q)} \right) = \left( \frac{L_{u^{-1}}^1}{d} , \frac{L_{u^{-1}}^2}{d} , \frac{L_{u^{-1}}^3}{d^2} \right)
\]
(see \eqref{eq:gammadef} for the definition of $\Gamma_u$).

At times we make use of the fortuitous fact that
\begin{equation}\label{eq:3cd1}
    (X L_{q^{-1}}^3)(p) = X_p (q^{-1} p)_3 = 2(y-q_2) = 2(q^{-1}p)_2 =  2 L_{q^{-1}}^2 (p) 
\end{equation}
and
\begin{equation}\label{eq:3cd2}
    (Y L_{q^{-1}}^3)(p) = Y_p (q^{-1} p)_3 = -2(x-q_1) = -2(q^{-1}p)_1 =  -2 L_{q^{-1}}^1 (p)
\end{equation}
for all $p,q\in\mH$.

Before proceeding, we remind the reader that convention \eqref{eq:derivconvention} with regard to derivatives and the convention exemplified by \eqref{eq:abscaexp} with regard to absolute constants are in place.

The purpose of $\lambda_g$ as in \eqref{eq:lambdadef} is to provide a suitably smoothed version of $ (p,q) \mapsto d(g(p),g(q)) $ from which it is possible to extract some useful estimates. The next lemma summarizes the important properties of $ \lambda_g $. As $ g $ is fixed, we will write $ \lambda = \lambda_g $.

\begin{lem} \label{lem:lambda}
$ \lambda : \mH \times \mH \to [0,\infty) $ is continuous and such that
\begin{description}
\item[(i)]
\[
    \lambda(p,q) \simeq_{K} d(g(p),g(q))
\]
for all $ (p, q) \in \mH\times \mH $; and
\item[(ii)] $ X_i \lambda $, $ X_j X_i \lambda $ exist and are continuous on $(\mH \times \mH)\setminus E$ with
\begin{align*}
    \frac{ | X_i \lambda(p,q) | }{ \lambda(p,q) } &\lesssim \frac{ 1 }{ d(p,q) }\exp\left( A K^{\frac{2}{3}} \right) \quad \text{ and} \\
    \frac{ | X_j X_i \lambda(p,q) | }{ \lambda(p,q) } &\lesssim \frac{ 1 }{ d(p,q)^{2} } \exp\left( A K^{\frac{2}{3}} \right).
\end{align*}
\end{description}
\end{lem}

Recall $E$ is the diagonal.

\begin{proof}
Let $p,q\in\mH$ with $p\neq q$. Let $u\in S(p,d(p,q)/2)$ and $w\in S(p,d(p,q)/4)$. We first observe that by \eqref{eq:bdist},
\[
    \left|gB\left(p,\frac{d(p,q)}{2}\right)\right|^{\frac{1}{4}} \lesssim \exp\left( A K^{\frac{2}{3}} \right) d(g(p),g(u)).
\]
Since $d(p,u)\leq d(p,q)$, this gives by \eqref{eq:wqs2} that
\begin{equation}\label{eq:ballest1}
    \left|gB\left(p,\frac{d(p,q)}{2}\right)\right|^{\frac{1}{4}} \lesssim \exp\left( A K^{\frac{2}{3}} \right) d(g(p),g(q)).
\end{equation}
Using \eqref{eq:bdist} again, we find
\[
    \left|gB\left(p,\frac{d(p,q)}{4}\right)\right|^{\frac{1}{4}} \gtrsim \exp\left( -A K^{\frac{2}{3}} \right) d(g(p),g(w)).
\]
An application of \eqref{eq:wqs} leads to 
\[
    d(g(p),g(q)) \leq 4^{K^\frac{2}{3}} \exp\left( A K^{\frac{2}{3}} \right) d(g(p),g(w)) 
\]
so that
\begin{equation}\label{eq:ballest2}
    \left|gB\left(p,\frac{d(p,q)}{4}\right)\right|^{\frac{1}{4}} \gtrsim \exp\left(-A K^{\frac{2}{3}}\right) d(g(p),g(q)).
\end{equation}
The definition of $ \lambda $ (see \eqref{eq:lambdadef}) together with \eqref{eq:cvf} and \eqref{eq:ballest1} gives
\begin{equation}\label{eq:lambdaest1}
	\lambda (p,q) \leq \left( \int_{ B\left( p, \frac{d(p,q)}{2} \right) } J_g \right)^{\frac{1}{4}} = \left| g B \left( p, \frac{d(p,q)}{2} \right) \right|^{\frac{1}{4}} \lesssim d(g(p),g(q)),
\end{equation}
and using \eqref{eq:ballest2} in place of \eqref{eq:ballest1} provides
\begin{equation}\label{eq:lambdaest2}
	\lambda (p,q) \geq \left( \int_{ B \left( p, \frac{d(p,q)}{4} \right) } J_g \right)^{\frac{1}{4}} = \left| g B\left( p, \frac{d(p,q)}{4} \right) \right|^{\frac{1}{4}} \gtrsim d(g(p),g(q)).
\end{equation}
Together \eqref{eq:lambdaest1} and \eqref{eq:lambdaest2} constitute the comparability of $\lambda (p,q)$ and $d(g(p),g(q))$ (as in statement (i)). Continuity of $ \lambda $ at $(p,q)$ follows from Lemma \ref{lem:hdui} with measure $ \mathrm{d} \mu (u) = J_g (u) \mathrm{d} u $.

Now let $p,q \in\mH$ with $p=q$. Recall $\lambda$ was defined to be zero on the diagonal. Since $\lambda(p,q) = d(g(p),g(q)) = 0$ the desired comparability is trivial. We rely on \eqref{eq:lambdaest1} and the (local) H\"{o}lder continuity of $g$ as in Lemma \ref{lem:uhold} to find $\lambda$ continuous at $(p,q)$. This concludes the proof of statement (i).

The existence and continuity of $ X_i \lambda $ and $ X_j X_i \lambda $ off the diagonal follow from Lemma \ref{lem:hdui} with measure $ \mathrm{d} \mu (u) = J_g (u) \mathrm{d} u $. In order to complete the proof of statement (ii) we make a series of estimates, beginning with the innermost function $\Gamma_u$ and working our way to the exterior.

Let $ p , q \in \mH $ with $ p \neq q $. We remain in this case for the rest of the proof. 

Let $u \in \mH $ satisfy $2 d(p,u) \leq d(p,q)$. Note that for $k=1,2$, $|L_{u^{-1}}^k (p)| \leq d(p,u) \leq d(p,q) $, and $|L_{u^{-1}}^3 (p)| \leq d^2 (p,u) \leq d^2 (p,q) $.

For $k=1,2$, we have 
\[
    X_i \Gamma_{u}^{k} = \frac{\left( X_i L_{u^{-1}}^k \right) d - L_{u^{-1}}^k X_i d}{d^2}
\]
so that
\[
| X_i \Gamma_{u}^{k} | \lesssim \frac{ | X_i L_{u^{-1}}^{k} | }{ d } + \frac{ | X_i d | }{ d }.
\]
We use the fact that $X_i L_{u^{-1}}^{k}$ is either $1$ or $0$ in combination with \eqref{eq:1ddistest} of the appendix to find 
\[
    | X_i \Gamma_{u}^{k} | \lesssim d^{-1}, \quad k = 1,2.
\]
The preliminary estimate for the third coordinate is
\[
    | X_i \Gamma_{u}^{3} | \lesssim \frac{ | X_i L_{u^{-1}}^{3} | }{ d^{2} } + \frac{ | L_{u^{-1}}^{3} X_i \left(d^{2}\right) | }{ d^{4} }
\]
which also leads to
\[
    | X_i \Gamma_{u}^{3} | \lesssim d^{-1}.
\]
In summary,
\begin{equation} \label{eq:Gammaest1}
| X_i \Gamma_{u}^{k} | \lesssim d^{-1}
\end{equation}
for all $ k = 1,2,3 $.

Similarly,
\[
    | X_j X_i \Gamma_{u}^{k} | \lesssim \frac{ | X_i L_{u^{-1}}^{k} X_j d| + |X_j L_{u^{-1}}^{k} X_i d| + |L_{u^{-1}}^{k} X_j X_i d | }{ d^{2} } + |X_i \Gamma_u^k|\left(\frac{ | X_j (d^2) | }{ d^{2} }\right)
\]
when $k=1,2$, and
\begin{equation*}
\begin{split}
    | X_j X_i \Gamma_{u}^{3} | \lesssim &\frac{ | (X_j X_i L_{u^{-1}}^3) d^2 | + | X_i L_{u^{-1}}^3 X_j (d^2)| + |X_j L_{u^{-1}}^3 X_i (d^2)| + |L_{u^{-1}}^3 X_j X_i (d^2) | }{ d^4 } \\
    &+ |X_i \Gamma_u^3|\left(\frac{ | X_j (d^4) | }{ d^4 }\right).
\end{split}
\end{equation*}
Here the conclusion is
\begin{equation} \label{eq:Gammaest2}
| X_j X_i \Gamma_{u}^{k} | \lesssim d^{-2}
\end{equation}
for all $ k = 1,2,3 $.

As $ \xi_0 $ is smooth and compactly supported there are bounds on the size of its derivatives. Since $ \xi_0 $ is fixed and does not depend on any varying quantity or function we introduce, we may consider these bounds as absolute constants. With this in mind, observe that $ X_i ( \xi_{0} \circ \Gamma_{u} ) = ( \nabla \xi_{0} ) ( \Gamma_{u} ) \cdot X_i \Gamma_{u} $ (where we write $ X_i \Gamma_{u} $ for $ ( X_i \Gamma_{u}^{1} , X_i \Gamma_{u}^{2} , X_i \Gamma_{u}^{3} ) $) so that by \eqref{eq:Gammaest1} we have
\begin{equation} \label{eq:lambdaker1}
| X_i ( \xi_{0} \circ \Gamma_{u} ) | \lesssim d^{-1}.
\end{equation}
Similarly, $ X_j X_i ( \xi_{0} \circ \Gamma_{u} ) = \sum_{k=1}^{3} \left[ ( ( \nabla\partial_{k}\xi_{0} ) ( \Gamma_{u} ) \cdot X_i \Gamma_{u} ) X_i \Gamma_{u}^{k} + \partial_{k}\xi_{0} ( \Gamma_{u} ) X_j X_i \Gamma_{u}^{k} \right] $ so that by \eqref{eq:Gammaest2},
\begin{equation} \label{eq:lambdaker2}
| X_j X_i ( \xi_{0} \circ \Gamma_{u} ) | \lesssim d^{-2}.
\end{equation}

Note by part (i) that $ \lambda > 0 $ off the diagonal. Lemma \ref{lem:hdui} allows us to differentiate under the integral to find
\[
    X_i \lambda (p,q) = \frac{1}{4} \frac{X_i (\lambda^4) (p,q)}{\lambda^3 (p,q)} = \frac{1}{4 \lambda^3 (p,q)} \int X_i (\xi_0 \circ \Gamma_u ) (p,q) J_{g}(u) \, \mathrm{d}u
\]
As $ \xi_0 (\Gamma_u (p,q)) = 0 $ for all $ u $ such that $ 2d(p,u) > d(p,q) $, this becomes
\[
	X_i \lambda (p,q) = \frac{1}{4 \lambda^3 (p,q)} \int_{B\left( p , \frac{d(p,q)}{2} \right)} X_i (\xi_0 \circ \Gamma_u ) (p,q) J_{g}(u) \, \mathrm{d}u
\]
and we apply to this expression the estimates derived under the assumption $ 2 d(p,u) \leq d(p,q) $ to find
\[
	\frac{ | X_i \lambda | }{ \lambda } \lesssim \left(\frac{1}{d}\right) \frac{ \int_{ B\left( p , \frac{d}{2} \right) } J_{g} }{ \int_{ B \left( p , \frac{d}{4} \right) } J_{g} } = \left(\frac{1}{d}\right) \frac{ \left| g B \left( p, \frac{d}{2} \right) \right| }{ \left| g B \left( p, \frac{d}{4} \right) \right| }.
\]
Applying \eqref{eq:ballest1} and \eqref{eq:ballest2} we find
\[
	\frac{ | X_i \lambda | }{ \lambda } \lesssim \frac{1}{d} \exp\left( A K^{\frac{2}{3}} \right).
\]

Similarly (but this time using \eqref{eq:lambdaker2} also),
\[
	\frac{ | X_j X_i \lambda | }{ \lambda } \lesssim \frac{ | X_j X_i (\lambda^{4}) | }{ \lambda^{4} } + \frac{ | X_j (\lambda^{4}) X_i (\lambda^{4}) | }{ \lambda^{8} } \lesssim \left(\frac{1}{d^2}\right) \left[ \frac{ \int_{ B\left( p, \frac{d}{2} \right) } J_{g} }{ \int_{ B\left( p, \frac{d}{4} \right) } J_{g} } + \left( \frac{\int_{ B\left( p, \frac{d}{2} \right)} J_{g} }{ \int_{ B\left( p, \frac{d}{4} \right) } J_{g} } \right)^2 \right],
\]
from which
\[
	\frac{ | X_j X_i \lambda | }{ \lambda } \lesssim \frac{1}{d^2} \exp\left( A K^{\frac{2}{3}} \right)
\]
follows.
\end{proof}

Lemma \ref{lem:lambda} can be thought of as giving estimates on logarithmic derivatives of $\lambda$. This was done due to the nature of $\eta_g$ (see \eqref{eq:etadef}) which we now begin working with. Write $\eta = \eta_g$. Recall that $U_g = \{(p,q) \,:\, p\neq g^{-1}(q)\}$.

\begin{lem}\label{lem:etag}
The function $\eta : U_g \to \mR$ satisfies 
\begin{equation}\label{eq:eta}
    | \eta(p,q) | \lesssim_{K} 1 + | \log[ d(g(p),q) ] |.
\end{equation}
Furthermore, $ X_i \eta $, $ X_j X_i \eta $ exist and are continuous on $ U_g $ with 
\begin{equation}\label{eq:etad1}
    | X_i \eta(p,q) | \lesssim \frac{ 1 }{ d(p,g^{-1}(q)) } \exp\left( A K^{\frac{2}{3}} \right)
\end{equation}
and
\begin{equation}\label{eq:etad2}
    | X_j X_i \eta(p,q) | \lesssim \frac{1}{d^2 ( p,g^{-1}(q) )} \exp\left( A K^{\frac{2}{3}} \right).
\end{equation}
\end{lem}

\begin{proof}
It follows from Lemma \ref{lem:lambda} part (i) that 
\begin{equation} \label{eq:eta+log}
\left| \eta(p,q) + \log[ d(g(p),q) ] \right| \lesssim_{K} 1
\end{equation}
and so
\[
    | \eta(p,q) | \lesssim_{K} 1 + | \log[ d(g(p),q) ] |.
\]
By Lemma \ref{lem:lambda} part (ii) we have that $ X_i \eta $, $ X_j X_i \eta $ exist and are continuous on $ U_g $ with 
\[
    | X_i \eta(p,q) | = \frac{ | X_i \lambda( p,g^{-1}(q) ) | }{ \lambda( p,g^{-1}(q) ) } \lesssim \frac{1}{d(p,g^{-1}(q))} \exp\left(AK^{\frac{2}{3}}\right)
\]
and
\begin{equation*}
\begin{split}
    | X_j X_i \eta(p,q) | &\leq \frac{ | X_j X_i \lambda( p,g^{-1}(q) ) | }{ \lambda( p,g^{-1}(q) ) } + \frac{ | X_j \lambda( p,g^{-1}(q) ) X_i \lambda( p,g^{-1}(q) ) | }{ \lambda( p,q^{-1}(q) )^{2} } \\
    &\lesssim \frac{1}{d^2 (p,g^{-1}(q))} \exp\left(AK^{\frac{2}{3}}\right) \qedhere
\end{split}
\end{equation*}
\end{proof}

Moving on to $ \widetilde{ \phi } = \widetilde{\phi}_g $ as in \eqref{eq:tphidef} we need several regularity statements and prefer to break them into small pieces.
\begin{lem} \label{lem:phitcts}
$ \widetilde{ \phi } $ is continuous on $ \mH \times \mH $. 
\end{lem}  
\begin{proof}
If $ (p,q) \in \mH \times \mH $ is such that $ p \neq g^{-1}(q) $ then by Lemma \ref{lem:lambda} part (i), $ \lambda(p,g^{-1}(q)) \neq 0 $ and so $ \widetilde{ \phi } $ is continuous at $ (p,q) $ by the continuity of $ \log $ away from $ 0 $. Suppose, therefore, that $ q = g(p) $ so that $ \widetilde{ \phi }(p,q) = 0 $. Let $ ( p_k, q_k ) $ be a sequence of points limiting on $ (p,q) $ and such that for all $ k $, $ ( p_k, q_k ) \in U_g = \{(p,q) \,:\, p\neq g^{-1}(q) \} $. (The presence of points outside $ U_g $ would not disturb the argument because $ \widetilde{ \phi } $ is zero at these points and we wish to show that $ \widetilde{ \phi }(p_k,q_k) $ converges to zero.) Then for all $ k $,
\[
	\widetilde{ \phi }(p_k,q_k) = \eta( p_k,q_k )( g^{-1}(q_k)^{-1}p_k )_{3}.
\] 
It follows from \eqref{eq:eta} that
\[
\left| \widetilde{ \phi }(p_k,q_k) \right| \lesssim \left| \left( g^{-1}(q_k)^{-1}p_k \right)_{3} \right| + \left| \log[ d\left( g(p_k) , q_k \right) ] \right| \left| \left( g^{-1}(q_k)^{-1}p_k \right)_{3} \right|.
\]
Our assumptions make it obvious that the first term on the right hand side tends to zero as $ k \to \infty $, therefore we only need work with the second term. For all $ u \in \mH $ we have $ |u_3| \leq \| u \|^2 $ so that
\[
	\left| \log\left[ d\left( g\left(p_k\right) , q_k \right) \right] \right| \left| \left( g^{-1}\left(q_k\right)^{-1}p_k \right)_{3} \right| \leq \left| \log\left[ d\left( g\left(p_k\right) , q_k \right) \right] \right| d^{2}\left( p_k , g^{-1}\left(q_k\right) \right).
\]
We complete the proof using the local H\"{o}lder continuity of quasiconformal mappings. Assume the $ ( p_k , q_k ) $ are close to the point $ ( p , q ) $. Then there exists $ 0 < R < \infty $ such that for all $k$ we have $ g(p_k), q_k \in B(R)$. Consequently, there exists $ \alpha > 0 $ such that for all $ k $,
\[
d( p_k , g^{-1}(q_k) ) = d( g^{-1}(g( p_k )) , g^{-1}( q_k ) ) \lesssim d^{\alpha}( g(p_k) , q_k ).
\]
Putting these last two observations together we get
\[
| \log[ d( g(p_k) , q_k ) ] | | ( g^{-1}(q_k)^{-1}p_k )_{3} | \lesssim \log[ d( g(p_k) , q_k ) ] | d^{ 2\alpha }( g(p_k) , q_k ).
\]
It is easy to see the right-side and hence the left-side goes to $ 0 $ as $ k \to \infty $.
\end{proof}

\begin{lem} \label{lem:dphitcts}
$X_i \widetilde{\phi}$ exists and is continuous on $\mH \times \mH$.
\end{lem}

\begin{proof}
Existence and continuity is immediate from previous observations if we are at a point $ (p,q) $ such that $ q \neq g(p) $.

Let us consider these things at a point of the form $ ( p , g(p) ) $. By the definition of the derivative we have
\[
    X_i \widetilde{ \phi }(p,g(p)) = \lim_{h\to 0} \frac{\widetilde{\phi}(p + h X_i (p), g(p)) - \widetilde{\phi}(p,g(p))}{h}. 
\]
This expression is rather jarring in two respects: besides the usual abuse that $X_i (p)$ is in the tangent space, not the space itself, we also have Euclidean addition in an inappropriate place. We rewrite it to be as intrinsic as possible arriving at
\[
    X_i \widetilde{ \phi }(p,g(p)) = \lim_{h\to 0} \frac{\widetilde{\phi}(p \delta_h (\exp X_i (0) ), g(p)) - \widetilde{\phi}(p,g(p))}{h}.
\]
Our first observation is that $\widetilde{\phi}(p,g(p))=0$ by definition (see \eqref{eq:tphidef}). Furthermore,
\[
    \widetilde{\phi}(p \delta_h (\exp X_i (0),g(p)) = \eta (p \delta_h (\exp X_i (0)),g(p)) (p^{-1}p \delta_h (\exp X_i (0)))_3 = 0
\]
because
\[
    (p^{-1}p \delta_h (\exp X_i (0)))_3 = (\delta_h (\exp X_i (0)))_3 = 0.
\]
Consequently, 
\[
    X_i \widetilde{ \phi }(p,g(p)) = \lim_{h\to 0} \frac{0}{h} = 0
\]
and we have demonstrated that for all $p\in\mH$ the derivative $X_i \widetilde{\phi} (p,g(p)) = X_i \widetilde{\phi}_{g(p)}(p)$ exists. As for continuity, the argument is similar to that of the continuity of $ \widetilde{ \phi } $ itself. Let $ (p_k,q_k) \to ( p , g(p) ) $ through points in $ U_g $. Then
\[
X_i\widetilde{ \phi }(p_k,q_k) = X_i\eta(p_k,q_k)( g^{-1}(q_k)^{-1} p_k )_3 + (-1)^j 2 \eta(p_k,q_k) ( g^{-1}(q_k)^{-1} p_k )_j,  
\]
where if $i=1$ then $j=2$ and if $i = 2$ then $j=1$ (here we have made use of \eqref{eq:3cd1} and \eqref{eq:3cd2} again).

Using \eqref{eq:eta}, \eqref{eq:etad1}, and that for $u\in \mH$, $|u_i|\leq \|u\|$ when $i=1,2$, this leads to
\[
| X_i\widetilde{ \phi }(p_k,q_k) | \lesssim d( p_k , g^{-1}(q_k) ) + | \log[ d( g(p_k) , q_k ) ] | d( p_k , g^{-1}(q_k) ).
\]
Now use H\"{o}lder continuity (as we did in Lemma \ref{lem:phitcts}) to conclude that the right hand side goes to $ 0 $ as $ k \to \infty $. It follows that $ X_i \widetilde{ \phi } $ is continuous on $ \mH \times \mH $.
\end{proof}

The observations (made in the proofs of the above two lemmas, possibly not for the first time) that for $u\in \mH$, $|u_i|\leq \|u\|$ when $i=1,2$, and $|u_3|\leq \|u\|^2$ will hereon be used without comment.

\begin{lem} \label{lem:wdphit}
For each $ q \in \mH $, $X_i \widetilde{\phi}_q \in HW_{\text{loc}}^{1,1}$. Furthermore, for a particular choice of $X_j X_i \widetilde{\phi}_q$ the function $(p,q) \mapsto X_j X_i \widetilde{ \phi }(p,q) := X_j X_i \widetilde{\phi}_q (p)$ is well-defined at almost every $(p,q) \in \mH \times \mH$, and $X_j X_i \widetilde{\phi} \in  L_{\textnormal{loc}}^r ( \mH \times \mH ) $ for all $1\leq r < \infty $.
\end{lem}

\begin{proof}
Let $ q \in \mH $ be given. We have seen that $X_i \widetilde{ \phi }_q$ is continuous. Moreover, at $ p \neq g^{-1}(q) $ we have that $ X_i \widetilde{ \phi }_q $ is continuously differentiable in all directions. It follows that $ X_i \widetilde{ \phi }_q $ is absolutely continuous on almost every integral curve of the horizontal, left-invariant vector field determined by $ X_j $. (We have not defined the measure on this fibration of $ \mH $, the details can be found in \cite{KRII}. Suffice to say $ g^{-1}(q) $ lies on only one curve, and a single curve has measure zero.)

It follows (see \cite[pp. 41-42]{KRII}) that the almost everywhere defined classical derivative $ X_j X_i \widetilde{ \phi }_q $ is a representative of the corresponding distributional derivative. Let us record explicit expressions for those derivatives. Let $ u := g^{-1}(q)^{-1} p $ so that $ \widetilde{ \phi } = u_3 \eta $ when $ p \neq g^{-1}(q) $. Then for $ p \neq g^{-1}(q) $ we have 
\begin{align*}
	X_1 X_1 \widetilde{ \phi } &= u_3 X_1 X_1 \eta + 4 u_2 X_1 \eta,\\
	X_1 X_2 \widetilde{ \phi } &= u_3 X_1 X_2 \eta + 2 u_2 X_2 \eta - 2 u_1 X_1 \eta - 2 \eta, \\
	X_2 X_1 \widetilde{ \phi } &= u_3 X_2 X_1 \eta + 2 u_2 X_2 \eta - 2 u_1 X_1 \eta + 2 \eta, \\
	X_2 X_2 \widetilde{ \phi } &= u_3 X_2 X_2 \eta - 4 u_1 X_2 \eta. 
\end{align*}
These expressions give measurable, almost everywhere defined functions on $ \mH \times \mH $. We now consider the local integrability. Given the estimates \eqref{eq:etad1} and \eqref{eq:etad2} worked out above,
\[
	| X_i X_i \widetilde{ \phi } | \lesssim_K \, 1  
\]
almost everywhere on $ \mH \times \mH $.

In the case $ i \neq j $ we use \eqref{eq:eta} in addition to \eqref{eq:etad1} and \eqref{eq:etad2} to find
\[
	| X_j X_i \widetilde{ \phi } | \lesssim_K \, 1 + | \log [ d( g(p) , q ) ] |
\]
almost everywhere on $ \mH \times \mH $. Let $ 1 \leq r < \infty $. Let $\Omega_1 \subset \mH $ be compact. We will show that
\[
	\int_{\Omega_1} | \log [ d( g(p) , q ) ] |^r \, \mathrm{d}p \lesssim h(q)
\]
for $ h \in L_{ \text{loc} }^1 (\mH) $. The claimed local integrability to the power $ r $ on $ \mH \times \mH $ follows.

To this end, let $ r_1 > 1 $ be the exponent appearing in the reverse H\"{o}lder inequality \eqref{eq:rhi} for $ g^{-1} $, and $ r_2 $ the conjugate exponent. Let $ 0 < R < \infty $ be such that $ g \Omega_1 \subset B(R) $. Then
\begin{align*}
	\int_{\Omega_1} | \log [ d( g(p) , q ) ] |^r \, \mathrm{d}p
	&= \int_{ g \Omega_1 } | \log [ d( p , q ) ] |^r J_{ g^{-1}} ( p ) \, \mathrm{d}p \\ 
	&\lesssim \left( \int_{g\Omega_1} | \log [ d( p , q ) ] |^{r r_2} \, \mathrm{d}p \right)^{ \frac{1}{r_2} } \left( \frac{ 1 }{ | B(R) | }\int_{B(R)} J_{g^{-1}}(p)^{r_1} \, \mathrm{d}p \right)^{ \frac{1}{r_1} } \\
	&\lesssim \left( \int_{g\Omega_1} | \log [ d( p , q ) ] |^{r r_2} \, \mathrm{d}p \right)^{ \frac{1}{r_2} }
\end{align*}
with the implied constants dependent on a variety of things but not $ q $. Now let $ \Omega_2 \subset \mH $ be compact and observe by H\"{o}lder's inequality that
\[
	\int_{\Omega_2} \left( \int_{g\Omega_1} | \log [ d( p , q ) ] |^{r r_2} \, \mathrm{d}p \right)^{ \frac{1}{r_2} } \, \mathrm{d}q \lesssim \left( \int_{\Omega_2} \int_{g\Omega_1} | \log [ d( p , q ) ] |^{r r_2} \, \mathrm{d}p \, \mathrm{d}q \right)^{ \frac{1}{r_2} }. 
\]
Note that
\[
	 \int_{g\Omega_1} | \log [ d( p , q ) ] |^{r r_2} \, \mathrm{d}p = \int_{ g \Omega_1 } | \log \| q^{-1} p \| |^{ r r_2 } \, \mathrm{d}p = \int_{ q^{-1} g \Omega_1 } | \log \| p \| |^{ r r_2 } \, \mathrm{d}p.
\]  
As $ \Omega_2 $ is compact, there exists $ 0 < R' < \infty $ such that $ q^{-1} g \Omega_1 \subset B(R') $ for all $ q \in \Omega_2 $.  Consequently, by formula \eqref{eq:polarint},
\[
	\int_{\Omega_2} \int_{g\Omega_1} | \log [ d( p , q ) ] |^{r r_2} \, \mathrm{d}p \, \mathrm{d}q \lesssim \int_{B(R')} | \log \| p \| |^{ r r_2 } \, \mathrm{d}p \lesssim \int_0^{R'} | \log(s) |^{ r r_2 } s^3 \, \mathrm{d}s < \infty.	\qedhere
\]     
\end{proof}

We next take a big step toward verifying Proposition \ref{pro:growth}.

\begin{lem} \label{lem:v3bound}
Let $R>0$. For all $ q \in B(R) $ and for all $ p \in \mH $,
\[
	| \widetilde{ \phi }(p,q) | \lesssim_{K,R} 1 + \|p\|^2 \log \|p\|
\]
and
\[
	| Z \widetilde{\phi}(p,q) | \lesssim_{K,R} 1 + \|p\|\log\|p\|.
\]
\end{lem}

\begin{proof}
Let $ R > 0 $ be given. Let $ p,q \in \mH $ with $ \| q \| < R $. Recall, we are assuming $ g \in \cQ_0(K) $ and the same is therefore true of $ g^{-1} $. By Lemma \ref{lem:udist} we have that both $ gB(R) $ and $ g^{-1}B(R) $ are contained in a ball, the radius of which depends only on $ K $ and $ R $. It follows (in the appropriate situation) we can replace a dependence on one of $ \|q\| $, $ \| g(q) \| $, or $ \| g^{-1}(q) \| $ with a dependence on $ K $ and $ R $. That said, dependence of constants on either of $ K $ or $ R $ will typically not be commented on.

Let $ R' > 0 $ be such that $ g^{-1}B( R + 1 ) \subset B(R') $. Such an $ R' $ depends only on $ K $ and $ R $ as guaranteed by Lemma \ref{lem:udist}.

Let $ u := g^{-1}(q)^{-1}p $. Notice that $\|u\| \lesssim 1 + \|p\|$ which is important throughout the following computations. 

By \eqref{eq:eta} and \eqref{eq:etad1}, 
\begin{equation}
	\left| \widetilde{\phi}(p,q) \right| \lesssim 1 + |u_3| + |u_3| \left| \log \left\| q^{-1}g(p) \right\| \right| .
\end{equation}

The purpose of the next computations is to replace the $\log$ in the previous expression with $\log^+$ so we may use the property $\log^+ (s_1 + s_2) \leq 1 + \log^+ s_1 + \log^+ s_2$ for $s_1, s_2 \geq 0$ (we also use that $\log^+ \left(s^r\right)=r\log^+ s$ when $s\geq 0$ and $r>0$, but this is less important).

If $ \| p \| > R' $ and $ \| q^{-1}g(p) \| = d(g(p),q) < 1 $ then $ g(p) \in B(R+1) $ and $ p = g^{-1}(g(p)) \in B(R') $, a contradiction.

Let $ \| p \| \leq R' $ and assume $\| q^{-1} g(p) \| < 1$. Let $ R'' > 0 $ be such that $ g B(R') \subset B(R'') $ (again, ultimately such an $ R'' $ depends only on $ K $ and $ R $). Then by Lemma \ref{lem:uhold} there exists $ \alpha > 0 $ (dependent on $ K $ only) such that $\| g^{-1}(b)^{-1} g^{-1}(a) \| \lesssim \| b^{-1}a \|^{\alpha}$ whenever $ a, \, b \in B(R'') $. Assuming (as we may) that $ R'' \geq R' \geq R $ we have $ q, g(p) \in B(R'') $ and
\[
	|u_3| \leq \| u \|^2 = \| g^{-1}(q)^{-1} g^{-1}(g(p)) \|^2 \lesssim \| q^{-1}g(p) \|^{2\alpha}.
\]
It follows,
\[
	|u_3| | \log \| q^{-1}g(p) \| | \lesssim \sup_{ \| q^{-1}g(p) \| < 1 } \| q^{-1}g(p) \|^{2\alpha} | \log \| q^{-1}g(p) \| | \lesssim 1.
\]

Consequently, in all cases
\begin{equation} \label{eq:Fbound}
	| \widetilde{\phi}(p,q) | \lesssim 1 + |u_3| + |u_3| \log^{+} \| q^{-1}g(p) \| .
\end{equation}

Since
\begin{equation} \label{eq:ubound}
	|u_3|  \lesssim 1 + \|p\|^2
\end{equation}
and (by \eqref{eq:udist})
\[
	\| q^{-1}g(p) \| \lesssim 1 + \| g(p) \| \lesssim 1 + \| p \|^{K^{2/3}},
\]
the aforementioned properties of $\log^+$ give
\begin{align*}
	| \widetilde{\phi}(p,q) | &\lesssim 1 + \|p\|^2 + \log^{+} \left( 1 + \| p \|^{K^{2/3}} \right) + \|p\|^2 \log^{+} \left( 1 + \| p \|^{K^{2/3}} \right) \\ 
							&\lesssim 1 + \|p\|^2 + \log^+ \|p\| + \|p\|^2 \log^{+} \| p \| \\
							&\lesssim 1 + \|p\|^2 \log^+ \| p \|
\end{align*}
(divide into cases $\|p\|\leq e$ and $\|p\| > e$ for the last line). There exists an absolute constant $C>0$ such that for all $p\in\mH$, $1 + \|p\|^2 \log \| p \| \geq C > 0$ (most pertinently for $\|p\|<1$). This allows us to conclude that
\[
    | \widetilde{\phi}(p,q) | \lesssim 1 + \|p\|^2 \log \|p\|,
\]
the first of the estimates we require.

Moving on to the estimate for the $Z$ derivative, observe (as in the proof of Lemma \ref{lem:dphitcts}) that $X_i \phi(p,q)$ is either zero or (with $u = u(p,q)$ as above) given by
\[
    X_i \widetilde{ \phi }(p,q) = ( X_i \eta(p,q) )u_3 + (-1)^j 2 \eta(p,q) u_j  
\]
with $j=2$ if $i=1$ and $j=1$ if $i = 2$. It follows by \eqref{eq:eta} and \eqref{eq:etad1} that
\[
	\left| X_i \widetilde{\phi}(p,q) \right| \lesssim \|u\| + \left( 1 + \left| \log\left\| q^{-1} g(p) \right\| \right| \right)\|u\|.
\]
Either $\|q^{-1}g(p)\| \geq 1$ or (by observations already made in this proof) $\|p\|\leq R'$ and
\[
	|X_i \widetilde{ \phi }(p,q)| \lesssim 1 + \|p\| + \left\| q^{-1}g(p) \right\|^{\alpha} \left| \log\| q^{-1} g(p) \| \right|
\]
with $\alpha > 0$ as before. In either case (arguing as above) we have
\begin{align*}
    |X_i \widetilde{ \phi }(p,q)| &\lesssim 1 + \|p\| + \log^+ \| q^{-1} g(p) \| + \|p\| \log^+ \|q^{-1} g(p)\| \\
    &\lesssim 1 + \|p\| \log^+ \|p\| \\
    &\lesssim 1 + \|p\| \log \|p\|
\end{align*}
as desired.
\end{proof}

We established regularity of $ \widetilde{ \phi } $ in such a way that it now transfers easily to $ \phi^1  = \phi_{g,\psi}^1$ (see \eqref{eq:phi1def} for the definition of $\phi^1$).

\begin{lem} \label{lem:phi1cts}
$ \phi^1 \in HC^1 $ with
\[
    X_i \phi^1 (p) = \int X_i \widetilde{\phi}(p,q)\psi (q) \,\mathrm{d}q.
\]
\end{lem}

\begin{proof}
This follows from Lemmas \ref{lem:hdui2}, \ref{lem:phitcts}, and \ref{lem:dphitcts}.
\end{proof}

\begin{lem} \label{lem:wdphi1}
$ X_i \phi^1 \in HW_{\text{loc}}^{1,r} $ for all $ 1 \leq r < \infty $ with
\[
	X_j X_i \phi^{1}(p) = \int X_j X_i \widetilde{ \phi }(p,q) \psi(q) \, \mathrm{d}q.
\]
\end{lem}

\begin{proof}
This follows from Lemmas \ref{lem:wdui}, \ref{lem:dphitcts}, and \ref{lem:wdphit}.
\end{proof}

The preliminaries are complete and we are now in a position to supply the proofs of Propositions \ref{pro:growth} and \ref{pro:vprops}. 

\begin{proof}[Proof of Proposition \ref{pro:growth}]
Recall that $K= K(g)$ and $\text{support}(\psi)\subset B(R)$.

Since $ \phi^2 = \phi_{g,\psi}^2 $ (defined in \eqref{eq:phi2def}) is a polynomial, continuity of both $ \phi $ and its first horizontal derivatives follows from Lemma \ref{lem:phi1cts}. Lemmas \ref{lem:hdui2}, \ref{lem:v3bound}, and the simple nature of $\phi^2$ easily give that
\[
	| \phi(p) | \lesssim_{K,R,\|\psi\|_1} 1 + \|p\|^{2} \log\|p\|
\]
and
\[
	| Z\phi(p) | \lesssim_{K,R,\|\psi\|_1} 1 + \|p\| \log\|p\|.
\]
This takes care of part (i).

Let $ u := g^{-1}(q)^{-1}p $. Then
\[
	\Re \left( ZZ \widetilde{\phi} \right) = u_{1} X_2 \eta + u_{2} X_1 \eta + \frac{u_{3}}{4}( X_1 X_1 - X_2 X_2 ) \eta,
\]
and
\[
	\Im \left( ZZ \widetilde{\phi} \right) = u_{1} X_1 \eta - u_{2} X_2 \eta - \frac{u_{3}}{4}( X_1 X_2 + X_2 X_1 ) \eta.
\]
Consequently, using Lemmas \ref{lem:phi1cts} and \ref{lem:wdphi1} along with \eqref{eq:etad1} and \eqref{eq:etad2} we find that	
\[
	\sqrt{2} \| ZZ\phi^{1} \|_{\infty} \lesssim \exp\left(A K^{\frac{2}{3}}\right) \|\psi\|_1
\]
with the implied constant absolute (in particular, independent of $K$).

A quick computation shows that
\[
	\sqrt{2} \| ZZ\phi^{2} \|_{\infty} = 0
\]
(as should be the case given our previous claim that the flow mappings of $v_{\phi^2}$ are conformal) which concludes the proof of part (ii).

Part (iii) of the statement follows from Lemma \ref{lem:wdphi1}.
\end{proof}

\begin{proof}[Proof of Proposition \ref{pro:vprops}]
Let $ p,q \in \mH $ be such that $ p \neq g^{-1}(q) $. Define $ u = g^{-1} (q)^{-1} p $. Then
\begin{align*}
	T\widetilde{\phi} 	&= \eta + u_{3} T\eta \\
					&= \eta - \frac{1}{4} u_{3} [X_1,X_2] \eta.
\end{align*}
Consequently, by Lemma \ref{lem:wdphi1} we have
\[
	T\phi^{1} = \Lambda_{\psi} \circ g + \zeta,
\]
where
\[
	\zeta(p) = \int \left[ \eta(p,q) + \log ( d(g(p),q) ) - \tfrac{1}{4} u_{3} \left(( X_1 X_2 - X_2 X_1 ) \eta\right)(p,q) \right] \psi(q) \, \mathrm{d}q.
\]
The function $ \zeta $ is easily seen to be measurable and (as follows from \eqref{eq:eta+log} and \eqref{eq:etad2}) is essentially bounded,
\[
	\| \zeta \|_{\infty} \lesssim_{K,\| \psi \|_{1}} \,1.
\]
Since $ T\phi^{2} = 0 $, the proof is concluded on remembering that $ \textnormal{div}_{H} v_{\phi} = T\phi $. 
\end{proof}

\section{Iteration and Convergence} \label{sec:ic}
With the large part of the technical work behind us, we are ready to construct the mapping which has Jacobian comparable to a suitable given quasilogarithmic potential.

In a first case, the desired mapping $ f $ is found in the limit of a sequence of mappings $ f_m $. Each $ f_m $ is the composition of $m$ time-$(1/m)$ flow mappings (modulo a minor technical detail that we save for later). The arguments consider the competition between the accumulation of a quantity in one direction, and the contracting effects of a diminishing time step in the other. Keeping the positive part of our measures small enough we will have enough uniformity in the estimates so that the squeeze from the time step dominates the process.


\subsection{Reduction of the Main Theorem}
We may reduce Theorem \ref{thm:mainint} to the following proposition.
\begin{pro} \label{pro:mainthmred}
Given $ K \geq 1 $, there exist $ \epsilon = \epsilon(K) > 0 $ and $ K' = K'(K) \geq 1 $ such that, if $ \mu \in \cM (\epsilon) $ and $ g \in \cQ_0 (K) $ then there is $f \in \cQ (K')$ with $J_f \simeq_K e^{2\Lambda_{\mu} \circ g}$ almost everywhere.
\end{pro}
The only difference between this proposition and Theorem \ref{thm:mainint}, is that here we assume $ g \in \cQ_0(K) $ as opposed to being simply a $ K $-quasiconformal mapping. Recall that $ g \in \cQ_0 (K) $ iff g is a $ K $-quasiconformal mapping such that $ g(0) = 0 $ and there exists $ p_g \in S(1) $ with $ g(p_g) \in S(1) $.

Let us assume Proposition \ref{pro:mainthmred} and explain why it implies Theorem \ref{thm:mainint}. Let $ K \geq 1 $ be given, $ \epsilon > 0 $ as in Proposition \ref{pro:mainthmred}, and $\Lambda = \Lambda_{\mu} \circ g $ a quasilogarithmic potential with $ \mu \in \cM (\epsilon) $ and $ g \in \cQ (K) $.

Pick $ q_0 \in \mH $ such that $ \| g( g^{-1}(0) q_0 ) \| = 1 $. It is automatic that $ q_0 \neq 0 $. Let $ p_0 := \delta_{\| q_0 \|^{-1}} ( q_0 ) $. Note that $ \| p_0 \| = 1 $ and $ q_0 = \delta_{\| q_0 \|}( p_0 ) $. Now define
\[
	h(p) = g\left( g^{-1}(0)\delta_{\|q_0\|}(p) \right).
\]
Since $ h $ is a composition of conformal mappings and a $ K $-quasiconformal mapping it is $ K $-quasiconformal. It is easily checked that $ h \in \cQ_0 (K) $ with $ p_h = p_0 $. Let $ \Lambda_h = \Lambda_\mu \circ h $. From Proposition \ref{pro:mainthmred} we have quasiconformal mapping $f_h$ such that
\begin{equation} \label{eq:hcomp}
	J_{f_h} \simeq_K e^{2 \Lambda_h}
\end{equation}
almost everywhere, with $ K(f_h) $ dependent on $ K $ only.

The definition of $ h $ makes it clear that
\begin{equation} \label{eq:gh}
	g = h \circ \delta_{\|q_0\|^{-1}} \circ L_{g^{-1}(0)^{-1}}. 
\end{equation}
If we define
\[
	f = \delta_{\|q_0\|} \circ f_h \circ \delta_{\|q_0\|^{-1}} \circ L_{g^{-1}(0)^{-1}} ,
\]
then $ f $ is quasiconformal with essential dilatation equal to that of $ f_h $. If $p \in \mH$ is such that $f_h$ is differentiable at $\delta_{ \|q_0\|^{-1} } ( g^{-1}(0)^{-1} p )$ then 
\begin{equation} \label{eq:jfjfh}
	J_f(p)  = J_{f_h} \left( \delta_{ \|q_0\|^{-1} } ( g^{-1}(0)^{-1} p ) \right).
\end{equation}
As $ \delta_{\|q_0\|^{-1}} \circ L_{g^{-1}(0)^{-1}} $ preserves sets of measure zero, it follows from \eqref{eq:hcomp} that
\[
	J_{f_h} \left(\delta_{ \|q_0\|^{-1} } ( g^{-1}(0)^{-1} p )\right) \simeq e^{2 \Lambda_h \left(\delta_{ \|q_0\|^{-1} } ( g^{-1}(0)^{-1} p )\right)}
\]
at almost every $ p $ in $ \mH $. This is seen by \eqref{eq:gh} and \eqref{eq:jfjfh} to be equivalent to $J_f \simeq e^{2\Lambda}$ almost everywhere, which is the conclusion of Theorem \ref{thm:mainint}.

We break the proof of Proposition \ref{pro:mainthmred} into three stages of increasing generality; first a quasilogarithmic potential of the form $\Lambda_{\psi} \circ g $ with $ \psi \in C_0^{\infty} $ (Proposition \ref{pro:mainpsi}), then $\Lambda_{\mu} \circ g $ with $ \mu $ compactly supported (Proposition \ref{pro:maincptsupp}), and finally  $\Lambda_{\mu} \circ g $ with general admissible $ \mu $ (which is Proposition \ref{pro:mainthmred}).

\subsection{$ \mathrm{d} \mu(q) = \psi(q) \, \mathrm{d}q $ with $ \psi \in C_{0}^{\infty} $}

This subsection is dedicated to the proof of the following statement.

\begin{pro} \label{pro:mainpsi}
Given $ K \geq 1 $, there exists $ \epsilon = \epsilon(K) > 0 $ such that if $ \psi \in C_0^{\infty} $ with $ \| \psi \|_1 < \epsilon $ and $ g \in \cQ_0 (K) $, then there exists a quasiconformal mapping $ f $ with $J_f \simeq_K e^{2\Lambda_{\psi} \circ g}$ almost everywhere. The dilatation $ K(f) $ depends on $ K = K(g) $ only.
\end{pro}
In this case, identification of the required $ \epsilon > 0 $ comes from the following lemma (which is essentially Lemma 6.1 of \cite{BHSII}).
\begin{lem} \label{lem:diffineq}
Let $ F_s $, $ s \in [0,1] $, be a family of quasiconformal mappings with $F_0$ the identity. Let $ G : [0,\infty) \to (0,\infty) $ be a continuous, increasing, and locally Lipschitz function, and let $ \epsilon' > 0 $ be such that
\[
	\epsilon' < \epsilon := \int_{0}^{\infty} \frac{1}{G(\sigma)} \, \mathrm{d}\sigma.
\]
Define
\[
	\Phi(s) = \sup_{0 \leq \sigma \leq s} \log K(F_{\sigma}) , \quad 0 \leq s \leq 1,
\]
and assume that for each $m\in\mN$ and $1\leq j\leq m$,
\[
	\sup_{\frac{j-1}{m} \leq \sigma \leq \frac{j}{m}} K( F_{\sigma} ) \leq \exp \left[ \frac{\epsilon'}{m} G\left(\Phi\left(\frac{j-1}{m}\right)\right) \right] \sup_{0 \leq \sigma \leq \frac{j-1}{m}} K(F_{\sigma}).
\]
Then $ F_1 $ is $ K $-quasiconformal with $ K $ dependent only on $ G $.  
\end{lem}
\begin{proof}
As $ \exp \left[ \frac{\epsilon'}{m} G\left(\Phi\left(\frac{j-1}{m}\right)\right) \right] \geq 1 $ we have
\[
	\sup_{0 \leq \sigma \leq \frac{j-1}{m}} K( F_{\sigma} ) \leq \exp \left[ \frac{\epsilon'}{m} G\left(\Phi\left(\frac{j-1}{m}\right)\right) \right] \sup_{ 0 \leq \sigma \leq \frac{j-1}{m}} K(F_{\sigma}),
\]
which coupled with our assumption gives
\[
	\sup_{ 0 \leq \sigma \leq \frac{j}{m}} K ( F_{\sigma} ) \leq \exp \left[ \frac{\epsilon'}{m} G\left(\Phi\left(\frac{j-1}{m}\right)\right) \right] \sup_{ 0 \leq \sigma \leq \frac{j-1}{m}} K(F_{\sigma}).
\]
It follows that
\[
	\Phi \left( \frac{j}{m} \right) \leq \Phi \left( \frac{j-1}{m} \right) + \frac{\epsilon'}{m} G \left( \Phi \left( \frac{j-1}{m} \right) \right).
\]
Given our assumptions on $ G $ and the choice of $ \epsilon' $, the equation
\[
	\Phi_{0}'(s) = \epsilon' G( \Phi_{0}(s) ), \quad \Phi_{0}(0) = 0, \quad 0 \leq s \leq 1,
\]
has a unique, finite solution that we call $\Phi_0$. Note that $ \Phi $ is increasing. We now show by induction that $\Phi(j/m)\leq \Phi_{0}(j/m)$ for all $j=0,\ldots,m$.

As $ \Phi(0) = 0 $ it is trivially valid for $ j = 0 $. Further, if $1\leq j \leq m$, and $\Phi((j-1)/m)\leq \Phi_{0}((j-1)/m)$, we find
\begin{align*}
	\Phi \left( \frac{j}{m} \right)
	&\leq \Phi\left( \frac{j-1}{m} \right) + \frac{\epsilon'}{m} G \left( \Phi \left( \frac{j-1}{m} \right) \right) \\
	&\leq \Phi_{0}\left( \frac{j-1}{m} \right) + \frac{\epsilon'}{m} G \left( \Phi_{0} \left( \frac{j-1}{m} \right) \right) \\
	&=\Phi_{0}\left( \frac{j-1}{m} \right) + \epsilon'\int_{ (j-1)/m}^{j/m}G \left( \Phi_{0} \left( \frac{j-1}{m} \right) \right) \, \mathrm{d}s\\
	&\leq \Phi_{0}\left( \frac{j-1}{m} \right) + \epsilon'\int_{(j-1)/m}^{j/m} G(\Phi_{0}(s)) \, \mathrm{d}s\\
	&=\Phi_{0} \left( \frac{j}{m} \right).
\end{align*}
In conclusion, $\Phi(1)\leq \Phi_{0}(1)$, and $\Phi_0$ depends only on $G$.
\end{proof}

Let us now fix (for the remainder of this subsection) $ K \geq 1 $, $ \psi \in C_{0}^{\infty} $, and $ g \in \cQ_0(K) $. We will write $ \Lambda = \Lambda_{\psi} \circ g $. Let $ p_0 \in S(1) $ be a point such that $ g(p_0) \in S(1) $.
 
In the following $m$ is always a natural number, and once such an $ m $ has been introduced $j$ is a natural number between $1$ and $m$. We reserve $s$ for our time variable, $ s \in [0,1] $.

If $ F \in \cQ_0(K'') $ for some $ K'' \geq 1 $, let $ \phi(F) = \phi_{F,\psi} $ be as in Proposition \ref{pro:growth}, and write $ v(F) = v_{\phi(F)} $.

For each $ m $, we run the following iterative procedure. Step $ 0 $ is to define $ f_{m,0} $ as the identity. Then for $j = 1, 2, \ldots, m$ we perform in turn
\[
	\text{step } j \;
		\left\{
			\begin{array}{l}
				v_{m,j} := v( g \circ f_{m,j-1}^{-1} ), \\
				h_{m,j} \textnormal{ is defined to be the time-}(1/m) \textnormal{ flow mapping of } v_{m,j} , \text{ and}\\
				f_{m,j}:=\delta_{\|h_{m,j}(f_{m,j-1}(p_0))\|^{-1}}(h_{m,j}\circ f_{m,j-1}).
			\end{array}	
		  \right.
\]
We define $ f_m $ to be the mapping $ f_{m,m} $ created by this process. In truth (given our agreed notation) for this algorithm to be well defined we need to know that each $ g \circ f_{m,j-1}^{-1} $ is in $ \cQ_0(K'') $ for some $ K'' $. This observation is included in the proof of the next proposition.

\begin{pro} \label{pro:fmconv}
There exists $ \epsilon = \epsilon(K) > 0 $ such that if $ \| \psi \|_1 < \epsilon $, then the $ f_m $ subconverge as $m\to \infty$ to a $ K' $-quasiconformal mapping with $ K' $ dependent on $ K $ only.
\end{pro}
\begin{proof}
Obviously the identity is a quasiconformal mapping fixing both $ 0 $ and the unit sphere. We have by Corollary \ref{cor:flowprops} that $ h_{m,1} $ is a quasiconformal mapping fixing $ 0 $. Consequently, as dilations are quasiconformal, also fix $ 0 $, and the dilation in play is designed to make $ \| f_{m,1}(p_0) \| = 1 $, we have that $ f_{m,1} \in \cQ_0(K_{m,1}) $ for some $ K_{m,1} \geq 1 $. Furthermore, that
\[
	\left\| (g \circ f_{m,1}^{-1}) (f_{m,1}(p_0)) \right\| = \| g( p_0 ) \| = 1
\]
gives $ g \circ f_{m,1}^{-1} \in \cQ_0 (K K_{m,1}) $.

Working iteratively, we see for each $m$ and $j$ that $ f_{m,j} \in \cQ_0(K_{m,j}) $ with $ 1 \leq K_{m,j} < \infty $ and, in particular, this is true of $ f_m $. Actually more is true since for all $ m $ we can take the same point $ p_0 \in S(1) $ as the point $p_{f_m}$ for which $ \| f_m (p_{f_m}) \| = 1 $.  Define $ K_m = K_{m,m} $.

Given the preceding observations, it follows from Lemma \ref{lem:qsconv} we will have subconvergence if we can demonstrate there exists $ K' $ such that each $ K_m \leq K' $ for all $m$. This is where Lemma \ref{lem:diffineq} comes in.

For each $ m $, define the family of quasiconformal mappings $ f_m (\cdot,s) $, $ s \in [0,1] $, as follows:
\[
	\text{if }  s \in \left[ \frac {j-1}{m} , \frac{j}{m}  \right), \text{ then } f_m (\cdot,s) = h_{m,j,s} \circ f_{m,j-1}  
\]
with $ h_{m,j,s} $ the time-$s$ flow mapping associated to $ v_{m,j} $ (with this notation $ h_{m,j}$ in the above algorithm is the same as $h_{m,j,1/m}$). Since dilations are $1$-quasiconformal,
\begin{align*}
	\sup_{\frac{j-1}{m} \leq \sigma \leq \frac{j}{m}} K( f_m (\cdot,\sigma ) )
	&\leq  \sup_{0 \leq \sigma \leq \frac{1}{m}} K(h_{m,j,\sigma}) K( f_{m,j-1} ) \\
	&\leq \sup_{0 \leq \sigma \leq \frac{1}{m}} K(h_{m,j,\sigma}) \sup_{0 \leq \sigma \leq \frac{j-1}{m}} K( f_m (\cdot,\sigma) ).
\end{align*}
At this point, we only need express $ \sup_{0 \leq \sigma \leq \frac{1}{m}} K(h_{m,j,\sigma}) $ in an appropriate form and we will be ready to invoke Lemma \ref{lem:diffineq}.

First we observe by Corollary \ref{cor:flowprops} that for $ s \in [0,1/m] $,
\[
	K(h_{m,j,s}) \leq e^{ C \| \psi \|_{1} s } \leq e^{ \frac{C \| \psi \|_1 }{ m } }
\]
where
\[
	C = C( K K_{m,j-1} ) = A_2 e^{A_1 (K K_{m,j-1})^{\frac{2}{3}}}
\]
($ A_1, A_2 > 0 $ are absolute constants). Let us define
\begin{equation} \label{eq:Gdef} 
	G(r) = A_2 \exp \left[ A_1 K^{\frac{2}{3}}\exp \left( \frac{2}{3}r \right)  \right], \quad r \in [0,\infty). 
\end{equation}
Then $ C( K K_{m,j-1} ) \leq G( \log K_{m,j-1} ) $ and as $ G $ is an increasing function it follows that
\[
	C( K K_{m,j-1} ) \leq G \left( \sup_{0 \leq \sigma \leq \frac{j-1}{m}} \log K( f_m (\cdot,\sigma) ) \right).
\]	
Note that $ G > 0$ and as $ G \in C^1 [0,\infty) $ it is locally Lipschitz. To summarize, $ G $ and the family of mappings $ f_m (\cdot,s) $ meet the requirements of Lemma \ref{lem:diffineq}, with $ G $ dependent on $ K $ only. It follows that so long as $ \| \psi \|_1 < \epsilon := \int_0^{\infty}1/G $, we will have each $ f_m $ a $ K' $-quasiconformal mapping with $ K' $ dependent on $ K $ only.
\end{proof}

Let us add to our standing assumptions that $\|\psi\|_1 < \epsilon$, where $\epsilon = \epsilon(K) > 0$ is as given by Proposition \ref{pro:fmconv}.

Using Lemma \ref{lem:qsconv}, Proposition \ref{pro:fmconv} identifies a subsequence of the $ f_m $ (which we continue to denote $ f_m $) that converges to a $ K' $-quasiconformal mapping $ f $. This mapping $ f $ is (modulo a small adjustment to come later) our candidate for comparability.

We will hereon use the words uniform and uniformly to indicate that something is independent of both $ m $ and $ j $.

The proof of Proposition \ref{pro:fmconv} gives that the $ g \circ f_{m,j-1}^{-1} $ are uniformly $ K'' $-quasiconformal with $K''= K K'$. This is crucial because it provides uniform estimates on the $ v_{m,j} $. To be more precise, recall by Proposition \ref{pro:vprops} that for each $ v_{m,j} $ we have
\[
	\Lambda \circ f_{m,j-1}^{-1} = \textnormal{div}_{H}v_{m,j} + \zeta_{m,j}
\]
with essentially bounded $ \zeta_{m,j} $ such that $\|\zeta_{m,j}\|_{\infty} \lesssim_{K K_{m,j-1}} \|\psi\|_1 $. With our assumption on $\|\psi\|_1$ and our uniform bound on $ K K_{m,j-1} $, we now have
\begin{equation} \label{eq:zetaest}
	\|\zeta_{m,j}\|_\infty \lesssim_K 1.
\end{equation}

\begin{proof}[Proof of Proposition \ref{pro:mainpsi}.]
For each $ m $ and at almost every $ p $, $ 0 < J_{f_m}(p) < \infty $ with
\[
	J_{f_m}(p) = \prod_{j=1}^{m} \| h_{m,j} (f_{m,j-1}(p_0)) \|^{-4} J_{h_{m,j}}( f_{m,j-1}(p) ).
\]
Consequently, at those same $ p $,
\[
	\log(J_{f_m}(p)) = \sum_{j=1}^{m} \log(J_{h_{m,j}}(f_{m,j-1}(p))) - 4\sum_{j=1}^{m}\log( \| h_{m,j}(f_{m,j-1}(p_0)) \| ).
\]
From now on $ c_m := -4\sum_{j=1}^{m} \log( \| h_{m,j}(f_{m,j-1}(p_0)) \| ) $. As above, we write $ h_{m,j,s} $ for the time-$ s $ flow mapping generated by $ v_{m,j} $ and we suppress dependence on the point $ p $. Using Corollary \ref{cor:flowprops} and Proposition \ref{pro:vprops} we may develop this as
\begin{align*}
	\log(J_{f_m})
	&= 2 \sum_{j=1}^{m} \int_{0}^{1/m} \textnormal{div}_{H} v_{m,j} ( h_{m,j,\sigma}(f_{m,j-1}) ) \, \mathrm{d}\sigma + c_m \\
	&= 2 \sum_{j=1}^{m} \int_{0}^{1/m} ( \Lambda \circ f_{m,j-1}^{-1}) ( h_{m,j,\sigma}(f_{m,j-1}) ) - \zeta_{m,j}( h_{m,j,\sigma}(f_{m,j-1}) ) \, \mathrm{d}\sigma + c_m.
\end{align*}
At those same points
\begin{align}
\begin{split} \label{eq:jacapprox}
	| \log J_{f_{m}}-2\Lambda
	&-c_{m} | \\
	&\leq \frac{2}{m} \sum_{j=1}^{m} \bigg[ \sup_{s\in[0,1/m]} | ( \Lambda\circ f_{m,j-1}^{-1} ) ( h_{m,j,s}(f_{m,j-1}) ) - \Lambda | \\
	& \quad \quad \quad \quad + \sup_{s\in[0,1/m]} | \zeta_{m,j} ( h_{m,j,s}(f_{m,j-1}) ) | \bigg] \\
	&\leq \frac{2}{m} \left[ \sum_{j=1}^{m} \sup_{s\in[0,1/m]} | ( \Lambda \circ f_{m,j-1}^{-1} ) ( h_{m,j,s}(f_{m,j-1}) ) - \Lambda | \right] + C_1,
\end{split}
\end{align}
with $ C_1 > 0 $ a constant dependent only on $ K $ (the appearance of which is justified by the uniform essential boundedness of $ \zeta_{m,j} $ as in \eqref{eq:zetaest}).

Since $ \psi \in C_{0}^{\infty} $, it follows from Lemma \ref{lem:lpsilip} that $ \Lambda_{\psi} $ is Lipschitz continuous. This gives
\begin{align}
\begin{split} \label{eq:uapprox}
	| ( \Lambda \circ f_{m,j-1}^{-1} )( h_{m,j,s}(f_{m,j-1}) ) - \Lambda |
	&= | \Lambda_{\psi} ( g( f_{m,j-1}^{-1}(h_{m,j,s}(f_{m,j-1})) ) ) - \Lambda_{\psi}(g) | \\
	&\lesssim d( g(f_{m,j-1}^{-1}(h_{m,j,s}(f_{m,j-1}))) , g ).
\end{split}
\end{align}
Let $ \xi \in C_{0}^{\infty} $, $ \xi \geq 0 $, $ \int \xi = 1 $, and $ D > 0 $ be such that $ \textnormal{support}( \xi ) \subset B(D) $. The $f_{m,j-1}$ are uniformly $ K' $-quasiconformal and (as already noted) satisfy the hypotheses of Lemma \ref{lem:udist}. Consequently, there is a $ D' > 0 $ such that for all $ m , j $,
\[
	f_{m,j-1}(B(D)) \subset B(D').
\]
The flow mapping $ h_{m,j,s} $ is generated by $ v_{m,j} = v\left( g\circ f_{m,j-1}^{-1} \right)$. In turn, $v_{m,j}$ corresponds to a potential $\phi_{m,j} := \phi \left( g\circ f_{m,j-1}^{-1} \right)$. This notation is detailed in the paragraph preceding the statement of Proposition \ref{pro:fmconv}. Observations already made and Proposition \ref{pro:growth} give uniform estimates
\[
    |\phi_{m,j} (p)| \lesssim 1 + \|p\|^2 \log \|p\| \quad \text{and} \quad |\phi_{m,j}(p)| \lesssim 1 + \|p\| \log \|p\| 
\]
at all $p \in \mH$. Lemma \ref{lem:bndsolns} and its proof now supply $ D'' > 0 $ such that for all $ m , j $, and for all $ s \in [0,1/m] $,
\[
	h_{m,j,s}( B(D') ) \subset B(D'').
\]
Lemma \ref{lem:uhold} gives H\"{o}lder continuity uniformly for the  $ g \circ f_{m,j-1}^{-1} $ on $ B(D'') $. This implies the existence of $ \alpha > 0 $ such that for all $ m , j $,
\begin{align}
\begin{split} \label{eq:holderonbd}
	d(g(f_{m,j-1}^{-1}(h_{m,j,s}(f_{m,j-1}))),g)
	&= d( g(f_{m,j-1}^{-1}( h_{m,j,s}( f_{m,j-1} ) )) , g( f_{m,j-1}^{-1}(f_{m,j-1}) ) ) \\
	&\lesssim d( h_{m,j,s}( f_{m,j-1} ) , f_{m,j-1} )^{\alpha}
\end{split}
\end{align}
at every point of $ B(D) $.
As $h_{m,j,s}(p)$ is the time-$s$ flow mapping of the vector field $v_{m,j}$, we have
\[
	h_{m,j,s}(p)=p+\int_{0}^{s}v_{m,j}(h_{m,j,\sigma}(p))\,\mathrm{d}\sigma.
\]
Having identified the $v_{m,j}$ as uniformly bounded on $ B(D'') $, the preceding expression provides the Euclidean estimate
\[
	|h_{m,j,s}(f_{m,j-1})-f_{m,j-1}|\lesssim \frac{1}{m}
\]
on $B(D)$. Via \eqref{eq:hemcomp} we get the still-useful estimate on the Heisenberg distance, 
\begin{equation} \label{eq:heholder}
	d(h_{m,j,s}(f_{m,j-1}),f_{m,j-1})\lesssim_{D''} |h_{m,j,s}(f_{m,j-1})-f_{m,j-1}|^{\frac{1}{2}}.
\end{equation}
Putting together \eqref{eq:uapprox}, \eqref{eq:holderonbd}, and \eqref{eq:heholder} at points of $B(D)$,
\[
	| ( \Lambda \circ f_{m,j-1}^{-1} ) ( h_{m,j,s} ( f_{m,j-1} ) ) - \Lambda | \lesssim \left( \frac{1}{m} \right)^{\alpha/2} .
\]
Plugging this into \eqref{eq:jacapprox} we find there exists constant $ C_2 > 0 $ such that at almost every $ p \in B(D) $,
\[
	| \log J_{f_{m}} - 2\Lambda - c_{m} | \leq C_2 m^{-\alpha/2} + C_1,
\]
or (to write it another way)
\begin{equation} \label{eq:firstcomp}
	e^{ -C_{2} m^{-\alpha/2} - C_1 } e^{ 2\Lambda } \leq e^{ -c_m } J_{f_m} \leq e^{ C_2 m^{-\alpha/2} + C_1 } e^{ 2\Lambda }.
\end{equation}
It is worth noting that $C_2$ depends on the radius of support of $ \psi $ (in addition to $ K $). This would be problematic were it not for the fact that it is the coefficient of $m^{-\alpha/2}$, and so the term involving $C_2$ will vanish when $m\to \infty$. In contrast, the constant $ C_1 $ (which survives the limit) depends on $ K $ only (as already noted).

Multiplying \eqref{eq:firstcomp} by $\xi$ and integrating, we find
\begin{equation} \label{eq:wjaccomp}
	e^{ - ( C_2 m^{-\alpha/2} + C_1 ) } \int \xi e^{2\Lambda} \leq e^{-c_m} \int \xi J_{f_m} \leq e^{ C_2 m^{-\alpha/2} + C_1 } \int \xi e^{2\Lambda}.
\end{equation}
Lemmas \ref{lem:wcj} and \ref{lem:qsconv} give $ \limsup_{m\to\infty} \int \xi J_{f_m} = \int \xi J_{f} $ which (since $ J_{f} $ is locally integrable and almost everywhere greater than zero) is finite and positive. Taking the $ \limsup $ as $ m \to \infty $ of \eqref{eq:wjaccomp}, we find
\[
	e^{-C_1} \int \xi e^{2\Lambda} \leq \limsup_{m\to\infty} ( e^{-c_m} ) \int \xi J_{f} \leq e^{C_1} \int \xi e^{2\Lambda}.
\]
As $ \int \xi J_{f} < \infty $ and $ e^{-C_1} \int \xi e^{2\Lambda} > 0 $ we must have $ \limsup_{m\to\infty} ( e^{-c_m} ) > 0 $. Similarly, given that $ e^{C_1} \int \xi e^{2\Lambda} < \infty $ and $ \int \xi J_{f} > 0 $ we must have $ \limsup_{m\to\infty} ( e^{-c_m} ) < \infty $. Let $ c_0 := \limsup_{m\to\infty} e^{-c_m} $.

The above being true for all $\xi\in C_{0}^{\infty}$ with $\xi\geq 0$ and $\int\xi=1$ means it is true for the mollifier $ \xi_{q,r} $ of center $ q $ and radius $ r > 0 $ (we may use the standard (Euclidean) mollifiers here, there is no need for `twisted convolution'). As both $ J_f $ and $ e^{2\Lambda} $ are locally integrable they have Lebesgue points almost everywhere. At a common Lebesgue point $ q $ we have
\[
	\int \xi_{q,r} e^{2\Lambda} \to e^{ 2 \Lambda(q) } \quad \text{ and} \quad \int \xi_{q,r} J_f \to J_{f}(q)  
\]
as $ r \to 0 $. See \cite{ZiemerBook} for these last couple of points. Hence we arrive at $c_0 J_f \simeq e^{2\Lambda}$ almost everywhere.

It might seem natural to include $ c_0 $ in the implied constant of comparability. It is likely, however, that $c_0$ depends not only on $K$ but on $R$ the radius of support of $ \psi $. As will soon become clear, it is important that the constant of comparability is independent of $R$. Instead we postcompose $ f $ with the dilation $\delta_{r_0}$, $r_0 := c_0^{1/4}$ and call the result $f$ again. We thus have a $ K' $-quasiconformal mapping $f$ such that $J_f \simeq e^{2\Lambda}$ almost everywhere, and with the implied constant dependent on $ K $ only.
\end{proof}
\subsection{Conclusion of the Proof of Theorem \ref{thm:mainint}}
Moving from the special case of the preceding subsection to the general case now follows as it does in the Euclidean case. With Proposition \ref{pro:mainpsi} in place, progress rests principally on Lemmas \ref{lem:theta1} and \ref{lem:theta2}. Let $\cM_0$ and $\cM_0 (\epsilon)$ comprise the compactly supported members of $\cM$ and $\cM (\epsilon)$, respectively.

\begin{pro} \label{pro:maincptsupp}
Given $ K \geq 1 $, there exist $ \epsilon = \epsilon(K) > 0 $ and $ K' = K'(K) \geq 1 $ such that, if $ \mu \in \cM_0 (\epsilon) $ and $ g \in \cQ_0 (K) $ then there is $f \in \cQ (K')$ with $J_f \simeq_K e^{2\Lambda_{\mu} \circ g}$ almost everywhere.
\end{pro}

\begin{proof}
Let $ \Lambda = \Lambda_\mu \circ g $ be a  quasilogarithmic potential with $ \mu \in \cM_0 $ and $ g \in \cQ_0(K) $. Let $ ( \psi_k ) $ be a sequence of smooth regularizations of $ \mu $ as in \eqref{eq:regmu}. Since $\mu$ is compactly supported so too are the $ \psi_k $. Let $ \Lambda_k := \Lambda_{\psi_k} \circ g $.

Proposition \ref{pro:mainpsi} tells us there exists $ \epsilon' > 0 $ such that, if $ \| \mu \| < \epsilon' $ (so that each $ \| \psi_k \|_1 < \epsilon' $) then for each $ k $ there is a quasiconformal mapping $ f_k $ with $ f_k(0) = 0 $ and
\begin{equation} \label{eq:kcomp}
	J_{f_k} \simeq e^{2 \Lambda_k}
\end{equation}
almost everywhere. The dilatation of $ f_k $ and the constant of comparability in \eqref{eq:kcomp} are both dependent only on $ K $ (the dilatation of the given $ g $), hence are each independent of $ k $.

If $ \theta = \theta(K) > 0 $ is as given by Lemma \ref{lem:theta1} and we let $ \epsilon'' > 0 $ be defined by $ \epsilon'' = \theta / 2 $, then if $ \| \mu \| < \epsilon'' $ (so that each $ \| \psi_k \|_1 < \epsilon'' $) we may conclude $ e^{ 2 \Lambda } \in L_{\text{loc}}^1 $ and that for any ball $ B \subset \mH $,
\[
	\int_B | e^{ 2 \Lambda_k } - e^{ 2 \Lambda } | \to 0 \quad \text{as} \quad k \to \infty.
\]   
It follows $ \int_{B(1)} e^{ 2 \Lambda_k } \to \int_{B(1)} e^{ 2 \Lambda } $ which combined with \eqref{eq:kcomp} gives $\int_{B(1)} J_{f_k} \simeq 1$ independently of $ k $. Using Lemma \ref{lem:jacconv} we may pass to a subsequence that converges, locally uniformly, to a $ K' $-quasiconformal mapping with $K'=K'(K)$.

Weak convergence of the Jacobians (as in the proof of Proposition \ref{pro:mainpsi}) gives $J_f \simeq e^{ 2 \Lambda }$ almost everywhere. The proof is completed by identifying the required $ \epsilon > 0 $ as $ \epsilon = \min \{ \epsilon' , \epsilon'' \} $.
\end{proof}
We are now ready to conclude the proof of Proposition \ref{pro:mainthmred} and so the proof of Theorem \ref{thm:mainint}.
\begin{proof}[Proof of Proposition \ref{pro:mainthmred}]
Let $ \Lambda = \Lambda_\mu \circ g $ be a quasilogarithmic potential with $\mu\in \cM$ and $ g \in \cQ_0(K) $. Define $ \mu_k = \mu \big|_{B(k)}$ so that each $\mu_k$ is compactly supported. Let $ \Lambda_k := \Lambda_{\mu_k} \circ g $. If $ \| \mu \| < \epsilon $ with $ \epsilon > 0 $ as in Proposition \ref{pro:maincptsupp} then (obviously) $ \| \mu_k \| < \epsilon $ for all $ k $. By Proposition \ref{pro:maincptsupp}, there exist $ K'=K'(K)$-quasiconformal mappings $( f_k )$ with $ f_k(0) = 0 $ and $J_{f_k} \simeq e^{ 2 \Lambda_k }$ almost everywhere (for all $k$ and with all implied constants independent of $ k $). Now use Lemma \ref{lem:theta2} to proceed exactly as in the proof of Proposition \ref{pro:maincptsupp}.
\end{proof}

\section{Weighted Sub-Riemannian Metrics} \label{sec:wsrm}

We conclude this paper with a geometric application, proving Theorem \ref{thm:specialgeom} of the introduction. That statement follows from Theorem \ref{thm:mainint} if we demonstrate the following: when a weight $\omega : \mH \to [0,\infty)$ is continuous and comparable to a quasiconformal Jacobian, the canonical sub-Riemannian Heisenberg group $ ( \mH , g_0 ) $ and its `conformal' deformation $ ( \mH , \sqrt{\omega} g_0 ) $ are bi-Lipschitz equivalent.

We write $\rho$ for the Carnot-Carathéodory distance function associated to $ ( \mH , g_0 ) $ (as in Section \ref{subsec:hg}). The above use of the notation $ ( \mH , \sqrt{ \omega } g_0 ) $ is a slight abuse as (see below) we replace not only the metric but also the curve families used in the definition of the distance function (slight in that, were we to make the same replacement in our earlier definition of $ ( \mH , g_0 ) $ the resulting distance function would be identical to $ \rho $).

If $0\leq b < \infty$ then a (bi-infinite) sequence $(a_k)_{k\in \mZ} \subset [0,b]$ will be called end-point limiting if $a_k \to 0$ when $k\to -\infty$, $a_k \to b$ when $k\to \infty$, and for all $k\in \mZ$ we have $a_k \leq a_{k+1}$. A bi-infinte series $\sum_{k=-\infty}^{\infty} c_k$ is defined to be $\sum_{k=0}^{\infty} c_{-k} + \sum_{k=1}^{\infty} c_{k}$ (we only consider such series for non-negative $c_k$ so it is likely other sensible definitions will coincide). The following definition is motivated by our later reliance on the curve families constructed in \cite{NagesCurves} (we need curves of this type to be considered among the competitors over which the distance function is to be defined). 
\begin{defn} \label{def:adcurve}
Let $ p , q \in \mH $. An admissible curve for $(p,q)$ is a number $ 0 \leq b < \infty $, a continuous mapping $ \gamma : [0,b] \to \mH $, an end-point limiting sequence $(a_k)_{k\in \mZ} \subset [0,b]$, and a sequence of horizontal curves $ ( \gamma_k : [a_k,a_{k+1}] \to \mH )_{k\in \mZ} $ such that
\begin{enumerate}
	\item $ \gamma(0) = p $, $ \gamma(b) = q $,
	\item $ \gamma(s) = \gamma_k (s) $ when $a_k \leq s \leq a_{k+1}$, and
	\item $ \sum_{k=-\infty}^{\infty} \int_{a_{k}}^{a_{k+1}} \left| \gamma_{k}'(s) \right|_{H} \, \mathrm{d}s < \infty $. 
\end{enumerate}
\end{defn}
Horizontal curves were defined in Section \ref{subsec:hg}. To call $ \gamma $ an admissible curve is to indicate that there exist $b$ etc. such that $(b,\gamma,\ldots)$ is an admissible curve for some $(p,q)$. For $ \gamma $ an admissible curve and $ g : \mH \to [0,\infty) $ continuous, let
\[
	\int_{\gamma} g := \sum_{k=-\infty}^{\infty} \int_{a_k}^{a_{k+1}} g(\gamma_k (s)) | \gamma_k '(s) |_{H} \, \mathrm{d}s.
\]

Given continuous $ \omega : \mH \to [0,\infty) $, let
\[
	\rho_{\omega}(p,q) := \inf_{\gamma} \int_{\gamma} \omega^{\frac{1}{4}},
\]
where the infimum is taken over all admissible curves for $(p,q)$. As things stand, $\rho_{\omega}$ is not necessarily a metric (only a pseudometric) as we have not assumed anything about the set on which $\omega$ vanishes.

The goal of a large part of this section is stated precisely as follows.
\begin{pro} \label{pro:bilip}
Suppose $ \omega : \mH \to [0,\infty) $ is continuous, and there exist $ C > 0 $ and a $ K $-quasiconformal mapping $ f : \mH \to \mH $ with
\[
	\frac{1}{C} \omega \leq J_{f} \leq C \omega 
\]
almost everywhere. Then there exists $ L \geq 1 $ such that for all $ p , q \in \mH $,
\[
	\frac{1}{L}\rho_{\omega}(p,q)\leq \rho(f(p),f(q))\leq L\rho_{\omega}(p,q).
\]
\end{pro}
In these circumstances $ \rho_{\omega} $ is a genuine metric and a rewording of the conclusion is that $ f $ is a bi-Lipschitz mapping between the metric spaces $ ( \mH , \rho_{\omega} ) $ and $ ( \mH , \rho ) $ (these are the metric spaces intended as implicit in our references to the bi-Lipschitz equivalence of $(\mH,\sqrt{w}g_0)$ and $(\mH,g_0)$). If we write $ \rho_{f}(p,q) = \rho(f(p),f(q))$ we may state our goal as $\rho_{\omega}\simeq\rho_{f}$.

Proposition \ref{pro:bilip} will be a corollary to the lemmas that follow. Let us fix $ \omega $ and $f$ as in the statement of the proposition.

If $U\subset \mH$ is Lebesgue measurable we write $\nu(U)$ for the $\omega(p)\,\mathrm{d}p$ measure of $U$, that is
\[
	\nu(U)=\int_{U}\omega.
\]
Now define the auxiliary function
\[
	d_{\omega}(p,q) = \nu^{\frac{1}{4}}\left(B_{p,q}\right),
\]
where $B_{p,q} := B(p,d(p,q))\cup B(q,d(q,p))$. Despite the suggestive notation, this is not in general a metric but only a quasimetric. The quasimetric space $ ( \mH , d_{\omega} ) $ is called a David-Semmes deformation of $ \mH $, a fascinating topic in its own right. The function $d_{\omega}$ has been introduced since we find it convenient to achieve $\rho_{\omega} \simeq \rho_{f}$ by proving $\rho_f\simeq d_{\omega}$ and $d_{\omega}\simeq\rho_{\omega}$. Before doing so, we take care of the following technicality.
\begin{lem} \label{lem:doubling}
$\nu$ is doubling, that is, there exists $C>0$ such that for all $p\in\mH$ and $r>0$,
\[
	\nu\left(B(p,2r)\right)\leq C\nu\left(B(p,r)\right).
\]
\end{lem}
\begin{proof}
Using our assumed comparability of weight and Jacobian along with the change of variable formula for quasiconformal mappings \eqref{eq:cvf} we find
\begin{equation} \label{eq:dblc1}
	\nu\left(B(p,2r)\right) = \int_{B(p,2r)} \omega \lesssim \int_{B(p,2r)} J_{f} = | fB(p,2r) |.
\end{equation}
Let $ \eta $ be the quasisymmetric control function of $f$ as in \eqref{eq:qscontrol}. With $ s $ defined as the minimum of $ d(f(p),f(q)) $ over $ \partial B(p,r) $ and $ \eta_2 := \eta(2) $, we have
\[
	B(f(p),s) \subset fB(p,r) \subset fB(p,2r) \subset B(f(p),\eta_{2}s).
\]
Using this (and that the Lebesgue measure is doubling with respect to the metric $d$) it follows
\begin{equation} \label{eq:dblc2}
	|fB(p,2r)|\leq |B(f(p),\eta_{2}s)|\lesssim |B(f(p),s)| \lesssim |fB(p,r)| \lesssim \nu\left(B(p,r)\right).
\end{equation}
Putting together \eqref{eq:dblc1} and \eqref{eq:dblc2} we have that $ \nu $ is doubling as required.
\end{proof}
\begin{lem} \label{lem:mu&rhof}
$ d_{\omega}\simeq\rho_{f} $ .
\end{lem}
\begin{proof}
Using the inclusion $B_{p,q}\subset B(p,2d(p,q))$ and the doubling property of $\nu$, we get
\[
	\nu\left(B_{p,q}\right) \leq \nu\left(B(p,2d(p,q))\right) \lesssim \nu\left(B(p,d(p,q))\right).
\]
It follows that
\[
	d_{\omega}(p,q) \simeq \left(\int_{B(p,d(p,q))}J_f\right)^{\frac{1}{4}}\simeq d(f(p),f(q)),
\]
where for the last comparison we use the change of variable formula \eqref{eq:cvf} and \eqref{eq:bdist} (recall that in deriving \eqref{eq:bdist} we use quasisymmetric control (along with the fact that Lebesgue measure is Ahlfors $4$-regular with respect to the metric $d$) in a similar manner to the proof of Lemma \ref{lem:doubling}).

As remarked in Section \ref{subsec:hg}, it is well known that $ d \simeq \rho $ so that $ d_{\omega} \simeq \rho_f $ is now immediate.
\end{proof}
With Lemma \ref{lem:mu&rhof} in place it remains to show $ d_{\omega} \simeq \rho_{\omega} $; this is the content of following two lemmas.
\begin{lem} \label{lem:mu&rhoomega1}
$ d_{\omega} \lesssim \rho_{\omega} $.
\end{lem}
\begin{proof}
For a horizontal curve $ \gamma_k : [a_k,a_{k+1}] \to \mH $, define the $ \omega $-length of $ \gamma_k $ as
\[
	l_{\omega}( \gamma_k ) = \limsup_{M\to\infty} \sum_{i=1}^{M} d_{\omega}( \gamma_{k}(s_{i-1}) , \gamma_{k}(s_{i}) ),
\]
where for each $ M \in \mN $ the $ s_i $ partition $ [ a_k , a_{k+1} ] $ into $ M $ equal length intervals (with $ s_0 = a_k $ and $ s_M = a_{k+1} $). For an admissible curve $ \gamma $ with sequence of horizontal curves $( \gamma_k : [a_k,a_{k+1}] \to \mH )_{k\in \mZ}$ define
\[
	l_{\omega}(\gamma)=\sum_{k=-\infty}^{\infty} l_{\omega}(\gamma_k).
\]
Now let $ p , q \in \mH $ and $ \gamma $ an admissible curve for $ (p,q) $ be given. We focus for a time on a single horizontal subcurve $\gamma_k$ and write $b_k$ for $a_{k+1}$. Let $ \epsilon > 0 $. Note that $ \gamma_k $ is uniformly continuous on $ [a_k , b_k] $ and there exists a compact set containing the image of $ \gamma_k $ on which $ \omega $ (and so $ \omega^{\frac{1}{4}} $) is uniformly continuous. It follows there exists an $ M_k < \infty $ such that, whenever $ (s_i) $ is a partition with $ | s_i - s_{i-1} | \leq (b_k - a_k) / M_k $ we have
\begin{equation} \label{eq:mlarge}
	\omega^{\frac{1}{4}}(u) \leq \omega^{\frac{1}{4}}(\gamma_{k}(s_{i-1})) + \epsilon
\end{equation}
for all $ u \in B_{\gamma_{k}(s_{i-1}) , \gamma_{k}(s_i)} $.

Assume that a partition $(s_i)$ of $ [a_k,b_k] $ is sufficiently fine  (as in the preceding paragraph) and define $ u_{k,i} = \gamma_{k}(s_{i}) $ with $B_{k,i-1} :=  B_{\gamma_{k}(s_{i-1}) , \gamma_{k}(s_i)} $. Then from the definition of $ d_{\omega} $,
\[
	d_{\omega}(u_{k,i-1},u_{k,i}) \leq \left( \sup_{B_{k,i-1}} \omega \right)^{\frac{1}{4}} | B_{k,i-1} |^{\frac{1}{4}},
\]
or equivalently
\[
	d_{\omega}(u_{k,i-1},u_{k,i}) \leq \sup_{B_{k,i-1}} \left( \omega^{\frac{1}{4}} \right) | B_{k,i-1} |^{\frac{1}{4}}.
\]
Ahlfors 4-regularity again and observation \eqref{eq:mlarge} give that
\[
	d_{\omega}(u_{k,i-1},u_{k,i}) \lesssim \omega^{\frac{1}{4}}(u_{k,i-1})d(u_{k,i-1},u_{k,i}) + \epsilon\, d(u_{k,i-1},u_{k,i}).
\]
It now follows from (the proof of) Lemma 2.4 in \cite{TysonHBook} that
\[
	l_{\omega}(\gamma_{k}) \lesssim \int_{a_{k}}^{b_{k}}\omega^{\frac{1}{4}}(\gamma_{k}(s))|\gamma_{k}'(s)|_{H}\,\mathrm{d}s+\epsilon\, l_{d}(\gamma_k), 
\]
where $ l_d $ is length with respect to $ d $ as defined in Section \ref{subsec:hg}. Since $ \epsilon $ was arbitrary and $ l_{d}(\gamma_k) $ is finite, this improves to
\begin{equation}\label{eq:lomegaintomega}
	l_{\omega}(\gamma_k) \lesssim \int_{a_{k}}^{b_{k}}\omega^{\frac{1}{4}}(\gamma_{k}(s))|\gamma_{k}'(s)|_{H}\,\mathrm{d}s. 
\end{equation}
By the continuity of $ \omega^{\frac{1}{4}} $ on a compact set containing the image of $ \gamma $ and property 3 of Definition \ref{def:adcurve} we have that $\int_{\gamma} \omega^{\frac{1}{4}} < \infty$. The comparison \eqref{eq:lomegaintomega} now leads to
\begin{equation} \label{eq:lmu&omega}
l_{\omega}(\gamma)\lesssim \int_{\gamma}\omega^{\frac{1}{4}}.
\end{equation}
By Lemma \ref{lem:mu&rhof} and the observation that concluded its proof we have $ d_f \simeq d_\omega $, where for $u,u'\in\mH$ we define $ d_f (u,u') = d(f(u),f(u')) $. It follows (using the triangle inequality for $d$) that for any finite collection of points $ u_0 , u_1 , \ldots , u_M \in \mH $,
\[
	d_{\omega}(u_0,u_M) \lesssim d_{f}(u_0,u_M) \leq \sum_{i=1}^{M} d_{f}(u_{i-1},u_{i}) \lesssim \sum_{i=1}^{M} d_{\omega}(u_{i-1},u_{i}).
\]
It is now straightforward that for each horizontal curve $\gamma_{k}$ we have
\[
	d_{\omega}(\gamma_{k}(a_k),\gamma_{k}(a_{k+1})) \lesssim l_{\omega}( \gamma_k ).
\]
We can use this to observe that for all $N,M\in\mZ$ with $N<M$ we have
\[
	d_{\omega}(p,q) \lesssim d_f (p,\gamma(a_N)) + \sum_{k=N}^{M} l_{\omega} (\gamma_k) + d_f (\gamma(a_M),q).
\]
Letting $N\to -\infty$ and $M\to \infty$, we can use appropriate continuity properties of $d_f$ to find that
\[
	d_{\omega}(p,q) \lesssim \sum_{k=-\infty}^{\infty} l_{\omega}(\gamma_k) = l_{\omega}(\gamma).
\]
Consequently, by \eqref{eq:lmu&omega} we find
\[
d_{\omega}(p,q)\lesssim \inf_{\gamma}l_{\omega}(\gamma)\lesssim \rho_{\omega}(p,q),
\]
where the infimum is taken over all admissible curves for $(p,q)$.
\end{proof}
Before proving the other side of the comparison, we state the proposition of \cite{NagesCurves} that dictated our definition of admissible curves (the statement does not appear as a proposition in that paper -- it can be deduced from the discussion in their Section 6.2 combined with their equation (3.3)).
\begin{pro}[Korte, Lahti, Shanmugalingam] \label{pro:curvefams}
There exist $ \lambda > 1 $ and $ C > 0 $ with the following property: for all $ p , q \in \mH $ there is a family $ \Gamma $ of admissible curves for $ (p,q) $ and a probability measure $ \alpha $ on $ \Gamma $ with
\[
	\int_{\Gamma} \left( \int_{\gamma} \omega^{\frac{1}{4}} \right) \, \mathrm{d} \alpha (\gamma) \leq C \int_{B(p,\lambda d(p,q))} \omega^{\frac{1}{4}}(u) \left( \frac{1}{d(p,u)^{3}} + \frac{1}{d(q,u)^{3}} \right) \, \mathrm{d}u.
\]
\end{pro}
We are now able to conclude the proof of Proposition \ref{pro:bilip} using (once again) the reverse H\"{o}lder inequality for the Jacobian of a quasiconformal mapping \eqref{eq:rhi}. The proof is based on an argument found in \cite{SemmesStronger} (see the proof of Proposition 3.12 of that paper).
\begin{lem} \label{lem:mu&rhoomega2}
$ d_{\omega} \gtrsim \rho_{\omega} $.
\end{lem}
\begin{proof}
Fix $ p , q \in \mH $. By the definition of $ \rho_\omega $ we have $ \rho_\omega (p,q) \leq \int_\gamma \omega^{1/4} $ for any admissible $ \gamma $ joining $ p $ and $ q $. Let $ \Gamma $ and $ \alpha $ be as in Proposition \ref{pro:curvefams}. Then (as $ \alpha $ is a probability measure) that proposition gives us $ \lambda > 1 $ such that
\[
	\rho_{\omega}(p,q) \lesssim \int_{B(p,\lambda d(p,q))} \omega^{\frac{1}{4}}(u) G(u) \, \mathrm{d}u,
\]
where $ G(u) := d(p,u)^{-3} + d(q,u)^{-3} $.

Since $\omega$ is comparable to a quasiconformal Jacobian it also satisfies a reverse H\"{o}lder inequality: there exists $s>1$ such that if $ B \subset \mH $ is a ball,
\begin{equation}\label{eq:revforomega}
	\left( \frac{1}{|B|} \int_{B} \omega^{s} \right)^{\frac{1}{s}} \lesssim \frac{1}{|B|} \int_{B} \omega
\end{equation}
independently of $ B $.

Let $ B := B(p,\lambda d(p,q)) $ and $ R := d(p,q) $. Let $r$ be the exponent conjugate to $4s$ (so that $1<r<4/3$). It follows from a change of variable and \eqref{eq:polarint} that
\[
    \int_B \frac{1}{d(p,u)^{3r}}\,\mathrm{d}u \lesssim \int_0^{\lambda R} a^{3-3r}\,\mathrm{d}a.
\]
Consequently,
\[
    \int_B \frac{1}{d(p,u)^{3r}}\,\mathrm{d}u \lesssim R^{4-3r}
\]
with $4-3r > 0$ and
\[
    \left( \frac{1}{|B|} \int_{B} G^{r} \right)^{\frac{1}{r}} \lesssim |B|^{-\frac{1}{r}} R^{\frac{4}{r} - 3}.
\]
Ahlfors 4-regularity (which says that $R^{4/r}\lesssim |B|^{1/r}$) allows us to deduce that
\[
    \left( \frac{1}{|B|} \int_{B} G^{r} \right)^{\frac{1}{r}} \lesssim R^{-3}.
\]
This along with \eqref{eq:revforomega} is used to find
\begin{align*}
	\int_{B}\omega^{\frac{1}{4}} G 
	&\lesssim \left( \int_{B} \omega^{s} \right)^{\frac{1}{4s}} \left( \int_{B} G^{r} \right)^{\frac{1}{r}} \\
	&\lesssim R^{4} \left( \frac{1}{|B|} \int_{B} \omega^{s} \right)^{\frac{1}{4s}} \left( \frac{1}{|B|} \int_{B} G^{r} \right)^{\frac{1}{r}} \\
	&\lesssim R^{4} \left( \frac{1}{|B|} \int_{B} \omega \right)^{\frac{1}{4}} R^{-3} \\
	&\lesssim \left( \int_{B} \omega \right)^{\frac{1}{4}} \lesssim d_{\omega}(p,q),
\end{align*}
as required. 
\end{proof}
The proof of Lemma \ref{lem:mu&rhoomega2} concludes the proof of Proposition \ref{pro:bilip}. That proposition combined with Theorem \ref{thm:mainint} gives the following (of which Theorem \ref{thm:specialgeom} is a special case).
\begin{thm} \label{thm:maingeom}
Given $ K \geq 1 $, there exist $ \epsilon = \epsilon(K) > 0 $ and $ L = L(K) \geq 1 $ such that, if $\mu \in \cM(\epsilon)$ and $ g \in \cQ (K) $ then $ (\mH , g_0) $ and $ (\mH , e^{\Lambda_{\mu} \circ g} g_0) $ are $L$-bi-Lipschitz equivalent. 
\end{thm}

\section{Appendix} \label{sec:apx}

For $f: \mH \to \mH$, let $\widehat{f} := (f_1,f_2)$. For $C^1$-smooth $g : \mH \to \mR$, define $\nabla_{\mH}g = (Xg,Yg)$.

If $f : \mH \to \mH$ is $C^1$-smooth and contact (satisfies equations \eqref{eq:contact1} and \eqref{eq:contact2}) and $g : \mH \to \mR$ is $C^1$-smooth then
\[
    X (g \circ f) (p) = \nabla_{\mH} g (f(p)) \cdot X \widehat{f} (p) \quad \text{and} \quad Y (g \circ f) (p) = \nabla_{\mH} g (f(p)) \cdot Y \widehat{f} (p).
\]

For $p = (x,y,t) \in \mH$, let $N(p) = \|p\|^4 = (x^2 + y^2)^2 + t^2$. For the record,
\begin{align*}
    XN &= 4 x (x^2 + y^2) + 4yt, \\
    YN &= 4 y (x^2 + y^2) - 4xt,
\end{align*}
and
\[
\begin{matrix}
    XXN = 12 (x^2 + y^2), & YXN = 4t,\\
    XYN = -4t, & YYN = 12 (x^2 + y^2).
\end{matrix}
\]

With $d_q (p) = d(p,q)$ for $p,q\in\mH$, we have
\[
    X \left(d_q^4\right) (p) = \nabla_{\mH} N \left(L_{q^{-1}}(p)\right) \cdot X\widehat{L_{q^{-1}}} (p), 
\]
and
\[
    Y \left(d_q^4\right) (p) = \nabla_{\mH} N \left(L_{q^{-1}}(p)\right) \cdot Y\widehat{L_{q^{-1}}} (p).
\]
Since $X (L_{q^{-1}})_{1}(p) = 1$, $X (L_{q^{-1}})_{2}(p) = 0$, $Y (L_{q^{-1}})_{1}(p) = 0$, and $Y (L_{q^{-1}})_{2}(p) = 1$, we have that
\[
    X\left(d_q^4\right)(p) = (XN)\circ L_{q^{-1}} (p) \quad \text{and} \quad Y\left(d_q^4\right)(p) = (YN) \circ L_{q^{-1}}(p).
\]
Similarly (with the notation of Section \ref{sec:ap}),
\begin{align*}
    X_j X_i \left(d_q^4\right)(p) = (X_j X_i N) \circ L_{q^{-1}}(p).
\end{align*}

Note, if $p = (x,y,t) \in \mH$ then $|x| \leq \|p\|$, $|y| \leq \|p\|$, and $|t|\leq \|p\|^2$. This means that for $i = 1,2$ we have $|(L_{q^{-1}})_{i}(p)| \leq d(p,q)$ and that
$|(L_{q^{-1}})_3 (p)| \leq d^2 (p,q)$.

The above can be used to find that
\[
    |X_i \left(d_q^4\right)(p)| \lesssim d_q^3 (p) \quad \text{and} \quad |X_j X_i \left(d_q^4\right)(p)| \lesssim d_q^2 (p).
\]

The somewhat artificial device of writing $d_q = \left( d_q^4 \right)^{1/4}$ leads to
\[
    X_i d_q = \left(\frac{1}{4}\right) \frac{X_i \left( d_q^4 \right)}{d_q^3}
\]
on $\mH \setminus \{q\}$ so that $|X_i d_q| \lesssim 1$ there. By the chain rule 
\begin{equation}\label{eq:1ddistest}
    |X_i \left( d_q^k \right)(p)| \lesssim d_q^{k-1}(p), \quad p \in \mH\setminus \{q\}.
\end{equation} 

Taking a second horizontal derivative results in
\[
    X_j X_i d_q = \left(\frac{1}{4}\right) \frac{\left(X_j X_i \left(d_q^4\right)\right)d_q^3 - X_i \left(d_q^4\right) X_j \left(d_q^3\right)}{d_q^6}
\]
on $\mH\setminus\{q\}$. It follows that $|X_j X_i d_q| \lesssim 1/d_q$ on $\mH\setminus\{q\}$ and 
\begin{equation} \label{eq:2ddistest}
\left|X_j X_i \left(d_q^k\right)\right| \lesssim d_q^{k-2}, \quad p \in \mH\setminus\{q\}.
\end{equation}

\bibliographystyle{plain}
\bibliography{aabibdb} 
\addresses
\end{document}